\documentclass[a4paper,12pt]{amsart}

\usepackage{lmodern}
\usepackage[T1]{fontenc}
\usepackage[utf8]{inputenc}
\usepackage[english]{babel}

\usepackage{amssymb,amsfonts,amsthm,amsmath,latexsym}
\usepackage{enumerate, caption, array, esint, bm, graphicx}
\usepackage{hyperref, mathrsfs, color}

\usepackage[margin=1.1in]{geometry}

\usepackage[normalem]{ulem}

\theoremstyle{plain}
\newtheorem{theorem}{Theorem}

\newtheorem{definition}{Definition}
\newtheorem{lemma}{Lemma}
\newtheorem{proposition}{Proposition}

\theoremstyle{remark}
\newtheorem*{remark}{Remark}

\newcommand\numberthis{\stepcounter{equation}\tag{\theequation}}

\renewcommand{\setminus}{\smallsetminus}
\newcommand{\ssum}[1]{\sum_{\substack{#1}}}

\newcommand{\e}{{\rm e}}
\newcommand{\dd}{{\rm d}}

\newcommand{\ee}{{\varepsilon}}
\newcommand{\chih}{{\hat \chi}}

\newcommand{\psih}{{\hat \psi}}
\newcommand{\phih}{{\hat \phi}}
\newcommand{\sg}{\langle}
\newcommand{\sd}{\rangle}

\DeclareMathOperator{\id}{id}

\newcommand{\C}{{\mathbb C}}
\newcommand{\R}{{\mathbb R}}
\newcommand{\Q}{{\mathbb Q}}
\newcommand{\Z}{{\mathbb Z}}
\newcommand{\N}{{\mathbb N}}
\newcommand{\1}{{\mathbf 1}}

\newcommand{\B}{{\mathcal B}}

\newcommand{\cD}{{\mathcal D}}
\newcommand{\cE}{{\mathcal E}}
\newcommand{\cF}{{\mathcal F}}

\newcommand{\cW}{{\mathcal W}}
\newcommand{\FR}{{\mathscr F}}

\newcommand{\fa}{{\mathfrak a}}

\newcommand{\fc}{{\mathfrak c}}
\newcommand{\fd}{{\mathfrak d}}
\newcommand{\fD}{{\mathfrak D}}
\newcommand{\fm}{{\mathfrak m}}
\newcommand{\fn}{{\mathfrak n}}
\newcommand{\fp}{{\mathfrak p}}
\newcommand{\ft}{{\mathfrak t}}

\newcommand{\gh}{{\hat g}}

\renewcommand{\O}{{\mathcal O}}
\newcommand{\Ov}{{\mathcal O}^{\vee}}
\newcommand{\Ok}{{\mathcal O}}
\newcommand{\cN}{{\mathcal N}}
\newcommand{\tc}[1]{{\left\langle {#1} \right\rangle}}

\newcommand{\rh}{\Theta}
\newcommand{\rb}{\theta}

\DeclareMathOperator{\card}{card}

\DeclareMathOperator{\meas}{meas}

\DeclareMathOperator{\supp}{supp}
\DeclareMathOperator{\Tr}{Tr}

\renewcommand{\tilde}{\widetilde}
\renewcommand{\bar}{\overline}
\renewcommand{\hat}{\widehat}
\newcommand{\qt}{{\tilde q}}

\newcommand{\floor}[1]{{\left\lfloor {#1} \right\rfloor}}
\renewcommand{\mod}[1]{\ ({\rm mod\ }#1)}

\newcommand{\abs}[1]{\left| #1 \right|}
\newcommand{\norm}[1]{\| #1 \|}

\numberwithin{equation}{section}

\title[Digits of primes in principal number fields]{The Thue-Morse and Rudin-Shapiro sequences at primes in principal number fields}

\date{\today}

\author{S. Drappeau}
\address{Aix Marseille Université, CNRS, Centrale Marseille, I2M UMR 7373, 13453, Marseille, France}
\email{sary-aurelien.drappeau@univ-amu.fr}

\author{G. Hanna}
\address{Institut Élie Cartan de Lorraine, Université de Lorraine, Site de Nancy, B.P. 70239, F-54506 Vandoeuvre-lès-Nancy Cedex}
\email{gautier.hanna@univ-lorraine.fr}

\subjclass[2010]{11R44 (Primary); 11A63, 11B85 (Secondary)}

\begin{document}

\maketitle

\begin{abstract}
  We consider a numeration system in the ring of integers~$\O_K$ of a number field, which we assume to be principal. We prove that the property of being a prime in~$\O_K$ is decorrelated from two fundamental examples of automatic sequences relative to the chosen numeration system: the Thue-Morse and the Rudin-Shapiro sequences. This is an analogue, in~$\O_K$, of results of Mauduit-Rivat which were concerned with the case~$K=\Q$.
\end{abstract}

\section{Introduction}

\subsection{Digits and multiplicative structure}

The present work is concerned with the interaction between the additive, multiplicative, and numeration properties of numbers, which is a reccurrent motivating theme in analytic number theory. The recent years, a lot of progress has been made on our understanding of digits of multiplicatively constrained integers (\emph{e.g.} primes): see~\cite{FouvryMauduit1996,FouvryMauduit1996a,DartygeTenenbaum2005,MauduitRivat2010,MR,Hanna2017,DrmotaEtAl2009} for sum of digits of primes in residue classes, \cite{HarmanKatai2008,Bourgain2015,Swaenepoel} for primes with restricted digits, or \cite{MR-Carres, DrmotaEtAl2011, DrmotaMorgenbesser2012, MauduitRivat2018, DrmotaEtAl2019} for digits of polynomials. Here we are interested in two particular digital functions (defined in terms of digit expansion), the sum-of-digits function
$$ s_q(n) := \sum_{0\leq j < J} b_j \qquad \text{if } n = \sum_{0\leq j < J} b_j q^j, \ b_j \in \{0, \dotsc, q-1\}. $$
and the Rudin-Shapiro sequence
$$ r(n) := \sum_{0\leq j < J-1} b_j b_{j+1} \qquad \text{if } n = \sum_{0\leq j < J} b_j 2^j, \ b_j\in\{0, 1\}. $$
Given a fixed integer~$m$, the functions~$n\mapsto s_q(n)\mod{m}$ and~$n\mapsto r(n)\mod{m}$ are two particular instances of automatic sequences, and it is predicted by Sarnak's Möbius randomness conjecture~\cite{Sarnak} (in one of its lowest complexity case) that they should not be correlated with integer factorization, in the precise sense that the Möbius function should have average zero along automatic sequences. For the sum-of-digit function, this expectation goes back to conjectures of Gel'fond~\cite{Gelfond1967/1968}. This question was solved, in a strong quantitative form, by Mauduit and Rivat~\cite{MauduitRivat2010} for the sum-of-digit case, then by the same authors~\cite{MR} for the Rudin-Shapiro case; and finally the full Sarnak conjecture for automatic sequences was proved by Müllner~\cite{Muellner2017}. The arguments in~\cite{MR} are one of the crucial inputs in~\cite{Muellner2017}.

\subsection{Digits of integers in number fields}

Our aim is to take up the study \cite{MR} and explore the corresponding questions in number fields. Let~$K/\Q$ be an algebraic extension, and~$\O_K$ be its ring of integers. We endow~$\O_K$ with a numeration structure, in the following way.
\begin{definition}
  Let~$q\in\O_K \setminus\{0\}$ and~$\cD \subset \O_K$ be a set of representatives of~$\O_K / (q)$. We call the pair~$(q, \cD)$ a \emph{number system with the finiteness property} (FNS) if:
  \begin{itemize}
    \item $0\in\cD$,
    \item the Galois conjugates of~$q$ have moduli larger than~$1$,
    \item every~$n\in\O_K$ has an expansion of the form~$n=\sum_{0\leq j < J} b_j q^j$, where~$b_j \in \cD$.
  \end{itemize}  
\end{definition}
We make a small account of works on these systems in Section~\ref{sec:setting} below; see also Section~3.1 of the survey~\cite{BaratBertheEtAl2006} for a discussion in the broad context of numeration systems.

The smallest~$J\in \N_{\geq 0}$ such that~$b_j=0$ for~$j\geq J$ will be called the \emph{length} of~$n$. The simplest non-rational example is the case~$K = \Q(i)$, $q=-1+i$, and~$\cD = \{0, 1\}$; see~\cite[p.~206]{Knuth1981}. We make an account of existing works on number systems relevant to our case in Section~\ref{sec:setting} below, and refer to~\cite{BaratBertheEtAl2006} for more references on this topic.

Our aim is to show that this numeration structure does not correlate with the multiplicative structure of~$\O_K$. We will assume, throughout, that~$\O_K$ is principal, so that it is a unique factorization domain. We present our results in the representative cases of the generalized sum-of-digit and Rudin-Shapiro functions.

Define, for all~$n\in\O_K$, the sum-of-digits function~$s(n) = s_{q,\cD}(n)$ as
\begin{equation}
  s_{q,\cD}(n) := \sum_{0\leq j < J} b_j \qquad \text{if } n = \sum_{0\leq j < J} b_j q^j, \ b_j \in \cD. \label{eq:def-sqK}
\end{equation}
Several aspects of this function have been studied in the past: asymptotic formula for the mean-value and fluctuations in its the constant term~\cite{GrabnerEtAl1998,Thuswaldner1998}, equidistribution modulo~$1$~\cite{Grabner1999}, central limit theorems~\cite{GittenbergerThuswaldner2000,Madritsch2010}, and equidistribution along squares~\cite{Morgenbesser2010}. In the case~$q=-1+i$, $\cD=\{0, 1\}$, we have~$s_{q,\cD}(n) \in \N$, and as a special case of~\cite[Theorem~11]{Grabner1999} we have that for any~$\alpha\in\R\setminus\Q$, the multi-sets
$$ \{ \alpha s_q(n),\ n\in\O_K \text{ of length }\leq J\} $$
become equidistributed modulo~$1$ as~$J\to\infty$.

We are interested in this question when~$n$ is restricted to be prime in~$\O_K$. In~\cite{DrmotaRivatEtAl2008,Morgenbesser}, this problem was addressed in the Gaussian integer setting~$K = \Q(i)$, using the approach of~\cite{MauduitRivat2010}. The question of whether the same method holds for other number systems was left open; the case when~$K$ is imaginary quadratic has specific aspects, notably the fact that multiplication by a complex number is a similarity, which are implicitely at play in~\cite{Morgenbesser}. We show that the expected statement in fact holds in full generality.

\begin{theorem}\label{thm:main-thue}
  Suppose that~$\O_K$ is a unique factorization domain, let~$(q, \cD)$ be a number system with the finiteness property, and~$\phi:K \to \R$ be a linear form. Then, as~$J\to\infty$, the multi-sets
  \begin{equation}
    \{\phi(s_q(p)),\ p\in\O_K \text{ prime of length }\leq J\}\label{eq:ens-main-thue}
  \end{equation}
  becomes equidistributed modulo~$1$ if and only if~$\phi(b) \in \R\smallsetminus\Q$ for some~$b\in\cD$.
\end{theorem}

Under the appropriate conditions, which are more involved, a similar equidistribution statement holds for linear maps~$\phi: K\to \R^d$ where~$d=[K:\Q]$. Note also that we have chosen, for simplicity only, to control the size of~$p$ by its digital length.

We next turn to the Rudin-Shapiro sequence, which was introduced due to the extremal properties of its associated trigonometric polynomials~\cite{Rudin1959, Shapiro1953}.  We consider the multidimensional variants constructed in~\cite{BarbeHaeseler2005}: we call~$(q, \cD)$ a \emph{binary} FNS if~$\card(\cD) = 2$. Binary NFS were characterized in~\cite{BarbeHaeseler2006}. We define a function~$r_{q,\cD} : \O_K \to \N$ by
\begin{equation}
  r_{q,\cD}(n) := \sum_{0\leq j < J-1} \1(b_j b_{j+1} \neq 0) \qquad \text{if } n = \sum_{0\leq j < J} b_j q^j, \ b_j \in \cD,\label{eq:def-rqK}
\end{equation}
where~$\1(n\neq 0)$ is~$0$ or~$1$ according to whether~$n=0$ or not.

This sequence is a non-trivial and natural instance of a digital function which has much less useful analytic properties than~$s_{q,\cD}$: it is not $q$-additive, and by analogy with the rational case, we do not expect its discrete Fourier transforms to have better than square-root cancellation in~$L^1$ norm (as opposed to~$s_{q,\cD}$). The arguments in~\cite{MR} were partly designed to work without these useful analytic properties.

In~\cite{DrmotaEtAl2008}, this function was considered in the general setting of ``block-additive functions''. There it was shown, using ergodic methods, that for any~$\alpha\in\R\setminus\Q$, the multi-sets
$$ \{ \alpha r_{q,\cD}(n),\ n\in \O_K \text{ of length }\leq J\} $$
become equidistributed modulo~$1$ as~$J\to \infty$\footnote{The authors of~\cite{DrmotaEtAl2008} work exclusively in the Knuth setting~$(q, \cD) = (-1+i, \{0, 1\})$, but the statement above can be easily deduced from our arguments.}. 

We show that the corresponding statement holds for primes in full generality.

\begin{theorem}\label{thm:main-rudin}
  Suppose that~$K$ is a unique factorization domain, let~$(q, \cD)$ be a binary FNS, and~$\alpha \in \R\setminus\Q$. Then, as~$J\to\infty$, the multi-sets
  \begin{equation}
    \{ \alpha r_{q,\cD}(p), \ p\in\O_K\text{ prime of length} \leq J\}\label{eq:ens-main-rudin}
  \end{equation}
  becomes equidistributed modulo~$1$.
\end{theorem}

As we mentioned in the introduction, in the recent work~\cite{Muellner2017} pertaining to the case~$K = \Q$, the Sarnak conjecture was fully solved for automatic sequences detecting integers given by their usual digital expansion. By combining the arguments of Sections~4.1--4.2 of~\cite{Muellner2017} with the work presented here, we expect that the M\"obius function
$$ \mu_K(n) = \begin{cases} (-1)^k & \text{ if } n \text{ is, up to units, a product of } k \text{ distinct primes}, \\ 0 & \text{ otherwise}, \end{cases} $$
is asymptotically orthogonal to the output of an automaton reading the digits of~$f$ in any FNS. Here, however, we choose to remain in the formalism of~\cite{MR}, having in mind only the sum-of-digits and the Rudin-Shapiro sequence. The input required to handle arbitrary automatic sequences does not substantially differ from~\cite{Muellner2017} and we believe it would obfuscate the ``number field'' aspects of our arguments. We also believe that by mixing the arguments of~\cite{MauduitRivat2018} with the ones presented here, one should be able to show that the multi-sets~$\{\alpha r_{q, \cD}(n^2), n\in\O_K \text{ of length} \leq J\}$ become equidistributed as~$J\to\infty$ if~$\alpha\not\in\Q$.

Another interesting direction would be to restrict the sets~\eqref{eq:ens-main-thue}, \eqref{eq:ens-main-rudin} to \emph{rationals} that are prime in~$\O_K$, which form a very sparse subset of all primes. The arguments presented here are, in their present form, not effective enough to address this question. Note however that partial results have been proved in~\cite{GrabnerEtAl1998} for the average of the sum-of-digit function, without the primality condition.

\subsection{Overview}

The difficulties we encounter in Theorems~\ref{thm:main-thue} and \ref{thm:main-rudin} are of two kinds.

The first is related to point-counting on lattices, and the ``skewing phenomenon''. The multiplication by~$q$, viewed as a map on the lattice~$\O_K$, can be quite far from a similarity in general, depending on the relative moduli of Galois conjugates of~$q$. This can induce an inefficiency in lattice-point counting estimates in large dilates~$q^\nu \O_K$. The effect of this skewing will be counteracted by systematically using Dirichlet's theorem on the structure of the group of units.

The second, more substantial difficulty is the harmonic analysis of the fundamental tile
$$ \cF = \{ \sum_{j\geq 1} b_j q^{-j}, b_j \in \cD\}. $$
By contrast with earlier works, the method of~\cite{MR}, which we take up here, makes a particularly extensive use of information on the Fourier transform~${\hat \chi_\cF}$ of the indicator function~$\chi_\cF$ of~$\cF$ (viewed as a subset of~$\R^d$ through a choice of basis of~$\O_K$). In the classical case~$K=\Q$, the fundamental tile is an interval (see~\cite[Lemma~1]{MR}), so that we have explicit expressions and bounds for its Fourier transform. In the general case, and in fact already for the Knuth setting~$K = \Q(i)$, $(q, \cD) = (-1+i,\{0, 1\})$, the fundamental tile has a non-trivial fractal boundary, known as the ``twin-dragon'' in the Knuth case (\cite[p.~66]{Mandelbrot1982}, \cite[p.~206]{Knuth1981}). The Fourier transform~${\hat\chi_{\cF}}$ does not decay uniformly enough for the method to naively go through\footnote{The analysis of the rates of decay of functions such as~${\hat\chi}_{\cF}$ is in fact an important object of study in wavelet theory; see the references in Section~\ref{sec:harm-analys-fund}.}. To handle this, we rework the arguments of~\cite{MR} so as to require as few information of the Fourier transform as possible: as we will show, the arguments of~\cite{MR} can be recast so that the only essential input is an~$L^2$ bound on~${\hat\chi}_{\cF}$, which we will obtain easily from Parseval's identity.

\section{Setting}

\subsection{Number field}\label{sec:setting}

Let~$K$ be a number field, with its trace map denoted~$\tc{x} = \Tr(x) = \Tr_{K/\Q}(x)$. We abbreviate throughout
$$ \O := \O_K. $$
We denote~$\Ov$ the dual of~$\O$ for the scalar product~$(x, y) \mapsto \Tr(xy)$. It is a fractional ideal, and the different ideal~$\fD_K := (\Ov)^{-1}\subset \O$ is of norm equal to the discrimant of~$K$~\cite[Chapter~4.1]{Narkiewicz}.

Given a base~$q\in\O$, all of whose conjugates have moduli greater than~$1$, and a set of digits~$\cD$, assume that any element of~$n\in\O$ has a unique base~$q$ expansion
$$ n = \sum_{j=0}^r b_j q^j, \qquad r\in\N, b_j\in\cD. $$
On the other hand, when the ring~$\O$ is principal, $n$ also possesses a factorisation~$n = p_1 \dotsb p_\ell$ as a product of prime elements, which is also unique up to order and multiplication by units.

Let~$(\omega_1, \dotsc, \omega_d)$ be a~$\Z$-basis of~$\O$, and~$(\omega^\vee_1, \dotsc, \omega^\vee_d)$ its dual basis~$\Ov$. For any $(x_j)_{1\leq j \leq d}$, $(y_j)_{1\leq j \leq d} \in \R^d$, denote
$$ \iota(x_1, \dotsc, x_d) = \sum_{j=1}^d x_j \omega_j, \qquad \iota^\vee(y_1, \dotsc, y_d) = \sum_{j=1}^d y_j \omega^\vee_j. $$
Note that~$K = \iota(\Q^d) = \iota^\vee(\Q^d)$, and
\begin{equation}
  \O = \iota(\Z^d), \qquad \Ov = \iota^\vee(\Z^d).\label{eq:Ok-lattice-ident}
\end{equation}
We fix a norm~$\|\cdot\|$ on~$\R^n$, and when~$x\in K$, we use the notation~$\|x\|$ to mean~$\|\iota^{-1}(x)\|$.

We denote~$G_K := {\rm Gal}(K/\Q)$, and given~$\pi \in G_K$ and~$x\in K$, we denote~$x^\pi := \pi(x)$.

We pick a base~$q\in\O$ and assume that all conjugates of~$q$ have modulus~$>1$. Let~$\cD$ be a set of representatives of~$\O/(q)$ containing~$0$. Borrowing the terminology of~\cite{PethoeThuswaldner2017}, we call such a pair~$(q, \cD)$ a number system. If every~$n\in\O$ has a finite expression
$$ n = \sum_{j=0}^r b_j q^j $$
with~$b_j\in\cD$ and~$r\geq 0$, then we say that~$(q, \cD)$ has the finiteness property. Note that such an expansion, if it exists, is unique. We use the abbreviation FNS to designate a number system with the finiteness property.

Given a pair~$(q, \cD)$, Kov\'acs and Peth\H{o}~\cite{KovacsPethHo1991} have shown that the question of whether~$(q, \cD)$ is a FNS can be decided algorithmically in finite time (they also characterize completely such number systems in positive characteristics). Gröchenig and Haas~\cite[Theorem~2.2]{GroechenigHaas1994} have shown that it corresponds exactly to a certain explicit matrix having spectral radius~$<1$ (which is equivalent to the existence of cycles in a certain directed graph, which we will mention below in Section~\ref{sec:carry-propagation}). The FNS are characterized for~$d=1$ in Theorem~2.3 of~\cite{GroechenigHaas1994}; in the same paper, the authors characterize the numbers~$q$ which can arise as the bases of FNS for the field~$K = \Q(i)$.

If~$(q,\cD)$ is a number system with~$\cD = \{0, \dotsc, N(q)-1\}$, the pair~$(q, \cD)$ is called a \emph{canonical number system} (CNS). The fields~$K$ which admit a CNS with the finiteness property have been characterized in~\cite{Kovacs1981}: they are exactly those fields for which~$\Ok$ has a primitive element. Many works have been devoted to deciding whether a given pair~$(q, \cD)$ is a CNS with the finiteness property. The problem was completely solved in the quadratic case~$d=2$ by K\'atai and Szab\'o~\cite{KataiSzabo1975} and K\'atai and Kov\'acs~\cite{KataiKovacs1980,KataiKovacs1981}. For~$d\geq 3$, only partial results are known; Akiyama and Peth\H{o}~\cite{AkiyamaPethHo2002} construct an algorithm which determines whether~$(q, \cD)$ is a CNS with the finiteness property using only the coefficients of the minimal polynomial of~$q$. Other partial results have been proved for~$d=3$~\cite{AkiyamaEtAl2003} and~$d=4$~\cite{BrunotteEtAl2006}.

Returning to general number systems, Germ\'an and Kov\'acs~\cite{GermanKovacs2007} proved that any~$q$ having all its conjugates of moduli~$<1/2$ admits a set of digits~$\cD$ for which~$(q, \cD)$ is a FNS.

From now on, we assume that~$(q, \cD)$ has the finiteness property.

Let~$\cF$ be the fundamental tile
\begin{equation}
  \cF = \Big\{ \sum_{j=1}^r b_j q^{-j},  r\geq 1, b_j \in \cD\Big\}.\label{eq:def-F}
\end{equation}
We will state in Section~\ref{sec:harm-analys-fund} below the basic properties of~$\cF$; for now, let us simply mention that there exist~$R_\cF^-, R_\cF^+ > 0$ (depending on~$(q,\cD)$ and~$\|\cdot\|$) such that
\begin{equation}
  \{x\in K, \|x\| \leq R_{\cF}^{-}\} \subset \cF \subset \{x\in K, \|x\| \leq R_{\cF}^+\}.\label{eq:F-subsup}
\end{equation}
In particular, since all the conjugates of~$q$ have moduli~$>1$, for some~$\Lambda\in\N$ there holds
\begin{equation}\label{eq:def-Lambda}
  (\cF + \cF) \cup (-\cF) \cup (\cF \cdot \cF) \subset q^{\Lambda}\cF.
\end{equation}
For any integer~$\kappa\geq 0$, we define
$$ \cN_{\kappa} := \Big\{\sum_{j=0}^{\kappa-1} b_j q^j, b_j\in\cD\Big\}. $$

\subsection{Hypotheses on~$f$ and~$(q, \cD)$}

We work with the formalism introduced in~\cite{MR}, which assumes two hypotheses of different nature on~$f$.

\begin{definition}
  We say that~$f$ satisfies the \emph{Carry property} if there exists a number~$\eta_1>0$ such that for any~$\kappa, \lambda, \rho\in\N$ with~$\rho \leq \lambda$, the number of~$v\in\cN_{\lambda}$ such that
  \begin{equation}
    f(u_1+u_2+vq^\kappa) \bar{f(u_1+vq^\kappa)} \neq f_{\kappa + \rho}(u_1+u_2+vq^\kappa) \bar{f_{\kappa + \rho}(u_1+vq^\kappa)}\label{eq:carry-prop}
  \end{equation}
  for some~$(u_1, u_2)\in\cN_\kappa^2$, is bounded by~$O(N(q)^{\lambda-\eta_1\rho})$.
\end{definition}

\begin{definition}
  We say that~$f$ satisfies the \emph{Fourier property} if there exist a non-decreasing function~$\gamma:\N\to\R_+$, and~$c>0$, such that uniformly for~$\lambda\in\N$, $\kappa\leq c\lambda$ and~$t\in K$,
  \begin{equation}
    \sum_{v\in\cN_\lambda} f(vq^\kappa) \e^{2\pi i\tc{tv}} \ll N(q)^{\lambda - \gamma(\lambda)}.\label{eq:Fourier-prop}
  \end{equation}
\end{definition}
As is noted in~\cite[eq. (26)]{MR}, if~\eqref{eq:Fourier-prop} holds then we always have
\begin{equation}
  \gamma(\lambda)\leq \lambda/2.\label{eq:bestbound-gamma}
\end{equation}

We define the following two ``distortion'' parameters on the number system:
\begin{equation}
  \label{eq:def-q-rh}
  \rh := \max_{\pi\in G_K}\frac{d \log |q^\pi|}{\log N(q)} \geq 1,
\end{equation}
\begin{equation}
  \label{eq:def-q-rb}
  \rb := \min_{\pi\in G_K}\frac{d \log |q^\pi|}{\log N(q)} > 0.
\end{equation}
Note that the inequality in~\eqref{eq:def-q-rh} is obvious, and that the inequality in~\eqref{eq:def-q-rb} follows from the assumption, made in Section~\ref{sec:setting}, that the Galois conjugates of~$q$ have moduli~$>1$, in other words~$\abs{x^\pi}>1$ for all~$\pi \in G_K$.

As a consequence of~\eqref{eq:def-q-rb}, the multiplication matrix associated to~$q^{-1}$ has spectral radius at most~$N(q)^{-\rb}<1$ (it is asymptotically contractant). We will use repeatedly the Gelfand inequality in the form
\begin{equation}
  \label{eq:Gelfand}
  \|q^\lambda\| \ll_q \lambda^{d-1} N(q)^{\rh\lambda},  \qquad  \|q^{-\lambda}\| \ll_q \lambda^{d-1} N(q)^{-\rb\lambda},
\end{equation}
for~$\lambda\in\N_{>0}$, see~\cite[Lemma~2.3]{Wirth98}.

\subsection{Main result}

Our main result is the proof of the following statement, which shows that the analogue of~\cite{MR} holds in number fields in the most general formulation.

\begin{theorem}\label{thm:meanvalue-prime}
  Assume that~$\Ok$ is principal,~$(q, \cD)$ is a FNS, and~$f:\O\to\C$ has the Carry and Fourier properties with the above notations, and~$c\geq 20\rh\rb^{-1}$. There exist~$C, \delta, \eta_2>0$, with~$\delta \asymp \eta_1\eta_2 d^{-1}\min\{\eta_1\eta_2, \theta\}$, such that for all~$\lambda\in\N_{>0}$, we have
  \begin{equation}\label{eq:meanvalue-prime}
    \ssum{p\in\cN_\lambda \\ p\text{ prime}} f(p) \ll_{K,q,\cD} \lambda^{C} N(q)^{\lambda - \delta\gamma(\frac\lambda{100\rh\rb^{-1}})}.
  \end{equation}
\end{theorem}
The constant~$\eta_2$ is a natural parameter associated the addition automaton of the NFS~$(q, \cD)$; in particular it depends only on~$(q, \cD)$. It is formally introduced below in Lemma~\ref{lem:carry}. In Appendix~\ref{sec:asympt-behav-addit} below, we study the asymptotic behaviour of this constant in infinite families of canonical number systems~$q = -m+x$, $m\in\N$, $m\to \infty$.

\subsection{Plan of the paper}

After compiling technical lemmas in Section~\ref{sec:lemmas}, we state and prove our type~I and~II estimates in Sections~\ref{sec:type-i-sums} and~\ref{sec:type-ii-sums}. We then prove Theorem~\ref{thm:meanvalue-prime} in Section~\ref{sec:sums-over-prime}, and deduce Theorems~\ref{thm:main-thue} and \ref{thm:main-rudin} in Section~\ref{sec:two-arithm-appl}. Appendix~\ref{sec:asympt-behav-addit} is concerned with a subsidiary result on asymptotic behaviour of carry constants.

\subsection{Notations}

In the sequel, we abbreviate
$$ Q := N(q), $$
$$ \e(z) := \e^{2\pi i z}. $$
It will also be useful to denote, for~$\lambda\in\N$ and~$t\in K$,
\begin{equation}\label{eq:def-elambdat}
  \e_\lambda(t) = \e\Big(\tc{\frac{t}{q^\lambda}}\Big).
\end{equation}
We recall the definitions~\eqref{eq:F-subsup}, and we let further
\begin{equation}\label{eq:F-norm}
  R_\cF^* = \sup_{x\in\cF} \prod_{\pi\in G_K}(1+\abs{x^\pi}).
\end{equation}

All implied constants will be allowed to depend on~$K$, $q$ and~$\cN$, unless otherwise stated.

\section{Lemmas}\label{sec:lemmas}

On many occasions, we will use the following simple bounds on norms of products.

\begin{lemma}\label{lem:norms}\ 
  \begin{enumerate}
    \item For all~$x, y\in K$, $\|xy\| \ll \|x\| \|y\|$,
    \item For all~$x\in K\smallsetminus\{0\}$, $\|x^{-1}\| \ll N(x)^{-1} \|x\|^{d-1} $.
  \end{enumerate}
\end{lemma}
\begin{proof}
  The first part is obvious. The second part follows from~$N(x) = \prod_{\pi\in G_K} x^{\pi}$. Indeed, writing~$x = \sum_{i=1}^d x_i \omega_i$ with~$x_i\in\Q$, then for any~$k\in\{1, \dotsc, d\}$, we have
  \begin{align*}
    \tc{\omega^\vee_k x^{-1}} ={}& N(x)^{-1} \tc{\omega^\vee_k \sum_{(i_\pi)_{\pi\neq \id}} \prod_{\pi\neq\id} x_{i_\pi} \omega_{i_\pi}} \\
    ={}& N(x)^{-1}\sum_{(i_\pi)_{\pi\neq \id}}  \Big(\prod_{\pi\neq\id} x_{i_\pi}\Big)  \tc{\omega^\vee_k \prod_{\pi\neq\id}\omega_{i_\pi}} \\
    \ll{}& N(x)^{-1} \|x\|^{d-1}.
  \end{align*}
\end{proof}

We also state now an upper-bound for the number of units in a certain angle.
\begin{lemma}\label{lem:count-units}
  For all~$x\in K^\ast$, we have
  $$ \card\big\{\ee\in\Ok^\ast, \|\ee / x\| \leq 1 \big\} \ll (\log(2 + N(x)))^{d-1}, $$
  where the implicit constant may depend on~$K$ and~$\|\cdot\|$.
\end{lemma}
\begin{proof}
  Shifting by a suitable unit (as in~\cite[p.55, eq.~(1.4)]{Murty2007}), we may assume that~$\abs{x^\pi} \asymp N(x)^{1/d}$ for all~$\pi\in G_K$. The condition~$\|\ee/x\|\leq 1$ then implies~$\abs{\ee^\pi} \ll N(x)^{1/d}$. Since~$\ee\in\Ok^\ast$, we also deduce~$\abs{\ee^\pi} = \prod_{\pi'\neq \pi} \abs{\ee^{\pi'}}^{-1} \gg N(x)^{1/d-1}$. Let~$(\ee_1, \dotsc, \ee_r)$ be a~$\Z$-basis of the free part of~$\Ok^\ast$~\cite[Theorem~I.7.3]{Neukirch1999} (where~$r \leq d-1$). We are then reduced to counting the number of tuples~$(n_1, \dotsc, n_r)\in\Z^r$ such that
  $$ \sum_{j=1}^r n_j \log|\ee_j^\pi| = O(\log(2+N(x))) $$
  for all~$\pi\in G_K$. Inverting this condition by using a subset of embeddings of size~$r$ as in~\cite[p.55]{Murty2007}, we find that there are at most~$O(\log(2+N(x))^r)$ solutions, whence the claimed bound.
\end{proof}

\subsection{Additive characters and van der Corput's inequality}

\subsubsection{Orthogonality}

We recall the following orthogonality relations.
\begin{lemma}\label{lem:orthog}
  For all~$q\in\O\smallsetminus\{0\}$ and~$\xi\in\Ov$, we have
  $$ \frac1{Q}\sum_{n\in\Ok/q} \e( \tc{q^{-1}n\xi} ) = \begin{cases} 1 & \text{if } \xi\in q \Ov, \\ 0 & \text{otherwise,} \end{cases} $$
  and similarly, for all~$n\in\O$,
  $$ \frac1{Q}\sum_{\xi\in\Ov/q} \e( \tc{q^{-1}\xi n} ) = \begin{cases} 1 & \text{if } n\in q \O, \\ 0 & \text{otherwise.} \end{cases} $$
\end{lemma}

\subsubsection{Counting additive characters}

In Section~\ref{sec:type-i-sums} below, we will require properties of additive characters in~$\Ok$, which we quote from~\cite[p.~179]{Huxley1968}. We recall that given an integral ideal~$\fm$ and an additive character~$\sigma\pmod{\fm}$, we say that~$\sigma$ is a \emph{proper} additive character modulo~$\fm$ if~$\sigma$ is not periodic~$\fn$ for any integral ideal~$\fn\supsetneq\fm$.

We will mainly work with additive characters of the form~$n\mapsto\e(\tc{nk/m})$, for~$m\in\Ok$, $k\in\Ov/m$ and $m\neq 0$. In this context, given an additive character~$\sigma$, let us denote
$$ (k, m) \sim \sigma \quad \iff \quad \forall n\in\Ok, \sigma(n) = \e(\tc{nk/m}). $$
Note that for any such~$k$ and~$m$, there is a unique pair~$(\fm, \sigma)$, where~$\fm$ containing~$m$, and a unique proper additive character~$\sigma\mod{\fm}$, such that~$(k, m)\sim \sigma$.
\begin{lemma}\label{lem:addchar-count}
  Let~$\mu\in\N_{>0}$, and~$\sigma\mod{\fm}$ be a proper additive character. Then
  $$ \ssum{m\in\cN_\mu \\ k\in\Ov/m \\ (k, m)\sim \sigma} \frac1{N(m)} \ll \frac{\mu^d}{N(\fm)}. $$
\end{lemma}
\begin{proof}
  For any~$m$ on the left-hand side, there can be at most one~$k\in\Ov/m$ for which~$(k, m)\sim\sigma$. Moreover, since~$\sigma$ is proper, this can only happen if~$\fm \mid m$. Therefore, the quantity on the left-hand side is at most~$\sum_{m\in\cN_\mu\cap \fm} \frac1{N(m)}$. We sort this sum according the principal ideal~$\fa = (m)$. First note that~$N(m) \leq R_\cF^\ast Q^\mu$ (we recall the definition~\eqref{eq:F-norm}). Then
  $$ \sum_{m\in\cN_\mu\cap \fm} \frac1{N(m)} \leq \ssum{\fa \text{ principal} \\\fm \mid \fa \\ N(\fa) \leq R_\cF^\ast Q^\mu} \frac1{N(\fa)} \ssum{m\in\cN_\mu \\ (m) = \fa} 1. $$
  For all~$\fa$ in the first sum, we pick a generator~$u\in\Ok$ such that~$|u^\pi| \asymp N(\fa)^{1/d}$ for all field imbedding~$\pi\in G_K$. Then the second sum is
  $$ \ssum{m\in\cN_\mu \\ (m) = \fa} 1 = \ssum{\ee \in \Ok^\ast \\ \ee u \in \cN_\mu} 1 \leq \ssum{\ee\in\Ok^\ast \\ \|\ee u / q^{\mu+C}\| \leq 1} \ll \mu^{d-1} $$
  for a constant~$C>0$ (depending on~$\cF$) and by Lemma~\ref{lem:count-units}. We deduce
  $$ \ssum{\fa \text{ principal} \\\fm \mid \fa \\ N(\fa) \leq R_\cF^* Q^\mu} \frac1{N(\fa)} \ssum{m\in\cN_\mu \\ (m) = \fa} 1
  \ll \mu^{d-1} \ssum{\fa \text{ ideal} \\\fm \mid \fa \\ N(\fa) \leq R_\cF^* Q^\mu} \frac1{N(\fa)} \ll \frac{\mu^d}{N(\fm)} $$
  as claimed.
\end{proof}

\subsubsection{Van der Corput's inequality}

For all~$\rho\in\N$, we define the set
\begin{equation}
  \Delta_\rho = \cN_\rho - \cN_\rho = \{ m - n, (m, n) \in \cN_\rho^2\}.\label{eq:def-Delta}
\end{equation}

\begin{lemma}\label{lem:van-der-corput}
  Let~$\rho, \kappa, \nu\in\N$ with~$\rho + \kappa \leq \nu$, and~$(z_n)_{n\in\Ok}$ be complex numbers satisfying~$z_n=0$ when~$n\not\in\cN_\nu$. There exists an even function~$w_\rho:\Ok\to\N$, such that~$|w_\rho(r)|\ll Q^\rho$ uniformly in~$r\in\Ok$, and
  $$ \Big|\sum_n z_n\Big|^2 \ll Q^{\nu - 2\rho} \sum_{r\in\Delta_\rho} w_\rho(r) \sum_n z_{n+q^\kappa r}\bar{z_n}. $$
\end{lemma}

\begin{proof}
  By following the proof of~\cite[Lemma~17]{MR-Carres}, we find
  $$ \Big| \sum_n z_n \Big|^2 \leq \card\Big(\bigcup_{r\in\cN_\rho}(\cN_\nu - q^\kappa r)\Big) Q^{-2\rho} \sum_{r\in\Delta_\rho} w_\rho(r) \sum_n z_{n+q^\kappa r} \bar{z_n}, $$
  where~$w_\rho(r) = \card\{(r_1, r_2) \in \cN_\rho^2, r = r_1 - r_2\}$. The claimed bound follows by our hypothesis~$\nu \geq \kappa + \rho$, which implies that the sets~$q^{-\nu}(\cN_\nu - q^\kappa r)$ are uniformly bounded for~$r\in\cN_\rho$.
\end{proof}

\subsubsection{Majorants of the fundamental tile and Poisson summation}

We will rely on the Poisson summation formula: for any continuous function~$V_0:\R^d \to \C$ satisfying~$V_0(x) \ll (1+\|x\|)^{-d-1}$, with Fourier transform~${\hat V_0}(\xi) = \int_{\R_d} V_0(x)\e(\tc{\xi, x})\dd x$, any invertible linear map~$B$, and any~$t\in\R^d$ we have
$$ \sum_{n\in\Z^d} V_0(B^{-1} n) \e(\tc{n, t}) = {\det}(B) \sum_{\xi\in\Z^d} {\hat V_0}(B(\xi + t)). $$
We deduce, in particular, that for~$V_0$ as above and~$V := V \circ \iota$, any~$\eta\in K\smallsetminus\{0\}$ and~$t\in K$, we have
\begin{equation}
  \sum_{n\in\Ok} V\Big(\frac{n}{\eta}\Big)\e(\tc{nt}) = N(\eta) \sum_{\xi\in\Ov} {\hat V}(\eta(\xi + t)).\label{eq:poisson-K}
\end{equation}  

It will be convenient to work with smooth majorant functions having a compactly supported Fourier transform.

\begin{lemma}\label{lem:majo-fourier-trunc}
  For any bounded set~$B\subset\R^d$, there exists a function~$V_0:\R^d \to \R_+$ in the Schwartz class, depending on~$B$, satisfying the following:
  \begin{enumerate}
    \item for all~$x\in \R^d$, if~$x\in B$, then~$V_0(x) \geq 1$,
    \item the Fourier transform~${\hat V_0}(\xi) = \int_{\R_d} V_0(x)\e(\tc{\xi, x})\dd x$ vanishes unless~$\|\xi\|_\infty \leq 1$.
  \end{enumerate}
\end{lemma}
\begin{proof}
  We take~$V_0(x) = \alpha \prod_{j=1}^d {\hat f}(\beta x_j)$, where~$f$ is given as in Theorem~A.3 of~\cite{Swaenepoel} for some small enough~$\beta>0$ and large enough~$\alpha$ depending on~$B$.
\end{proof}

\subsubsection{The large sieve inequality}

The following is a multidimensional version of the large sieve inequality, and corresponds to Theorem~2 of~\cite{Huxley1968}. The main difference lies in the scaling of the set of points: when dealing with ideals (rather than arbitrary lattices), we can avoid the ``skewing'' phenomenon refered to above. We refer to~\cite{Montgomery} for history and additional references on the large sieve.
\begin{lemma}\label{lem:large-sieve}
  Let~$\alpha\in K^\ast$, $X\in[1, \infty)$ and~$(c(n))_{n\in\Ok}$ be complex numbers. Then
  $$ \ssum{\fm\text{ ideal} \\ 0<N(\fm) \leq X} \ssum{\sigma\mod{\fm} \\ \text{proper}} \Bigg| \ssum{n\in\Ok \\ \|n/\alpha\| \leq 1} c(n) \sigma(n) \Bigg|^2 \ll (N(\alpha) + X^2) \ssum{n \in \Ok \\ \|n/\alpha\|\leq 1} \abs{c(n)}^2,  $$
  where in the sum on the left-hand side, $\sigma$ runs over proper additive characters~$\mod{\fm}$.
\end{lemma}

\begin{proof}
  Let~$\ee\in\Ok^\ast$ be such that~$|(\ee \alpha)^\pi| \asymp N(\alpha)^{1/d}$, and denote~$\alpha' = \ee \alpha$. For any~$\fm$, the map~$\sigma \mapsto (a\mapsto \sigma(\ee^{-1}a))$ is a permutation of the proper additive characters~$\mod{\fm}$. Therefore, we have
  $$ \ssum{\fm\text{ ideal} \\ 0<N(\fm) \leq X} \ssum{\sigma\mod{\fm} \\ \text{proper}} \Bigg| \ssum{n\in\Ok \\ \|n/\alpha\| \leq 1} c(n) \sigma(n) \Bigg|^2 = \ssum{\fm\text{ ideal} \\ 0<N(\fm) \leq X} \ssum{\sigma\mod{\fm} \\ \text{proper}} \Bigg| \ssum{n\in\Ok \\ \|n/\alpha'\| \leq 1} c'(n) \sigma(n) \Bigg|^2 $$
  with~$c'(n) = c(\ee^{-1}n)$. The condition~$\|n/\alpha'\|\leq 1$ implies~$\|n\| \ll \|\alpha'\| \ll N(\alpha)^{1/d}$, and so Theorem~2 of~\cite{Huxley1968} can be applied with~$N_j \ll N(\alpha)^{1/d}$, which yields the claimed result.
\end{proof}

\subsection{Numeration}

\subsubsection{Carry propagation}\label{sec:carry-propagation}

Let~$r_{\nu, \mu}(n)$ be the integer formed with the digits of~$n$ of indices~$\{\nu, \dotsc, \mu-1\}$, so that if~$n = \sum_{j\geq 0} n_j q^j$, then~$r_{\nu, \mu}(n) = \sum_{0 \leq j < \mu-\nu} n_{\nu+j} q^j$. We write~$r_{\nu, \infty}(n) = \lim_{\mu\to\infty} r_{\nu, \mu}(n)$. We wish to quantify the fact that propagation of a carry is an exponentially rare event. This has been studied in particular in~\cite{Grabner1999,MTT-Fractal}. In fact the following lemma can be seen as an arithmetic restatement of a weaker version of~\cite[Proposition~4.1]{MTT-Fractal}.
\begin{lemma}\label{lem:carry}
  There exists~$\eta_2 = \eta_2(q, \cD) \in (0, 1]$, such that for all integers~$0\leq \rho \leq \nu \leq \mu$, we have
  \begin{equation}
    \card\{m\in\cN_\mu, \exists n\in\cN_{\nu-\rho}, r_{\nu,\infty}(m+n) \neq r_{\nu,\infty}(m) \} \ll Q^{\mu-\eta_2\rho}.\label{eq:carry-upperbound}
  \end{equation}
\end{lemma}
However, in~\cite{MTT-Fractal} and with their notations, the authors work under the assumption that a certain graph~${\tilde G}(S)$ is primitive. Since we are only interested in the upper-bound~\eqref{eq:carry-upperbound}, we do not need this assumption here. What is required is that, from any vertex, there is a path leading to an absorbing state; the possibility of the matrix of the graph~${\tilde G}(S)$ having multiple dominant eigenvalues does not affect us. In~\cite{ScheicherThuswaldner2002}, the primitivity of~${\tilde G}(S)$ is proved, however in the more specific case of canonical number systems. For these reasons we include a self-contained proof of Lemma~\ref{lem:carry}.

\begin{proof}
  We have assumed that~$0\in\cD$, and that every element in~$\Ok$ has a base $q$ expansion. Let~$\B = \cD - \cD$, and define a sequence of sets by
  $$ \B_0 = \{0\}, \qquad \B_{j+1} = \B + q\B_j $$
  for all~$j\geq 0$. Note that~$\B_j = \B + q\B + \dotsb q^{j-1}\B$ for~$j\geq 1$, so that this sequence is increasing. Since~$\cD \subset \B$, we have~$\cN_j \subset \B_j$. Moreover, for all~$j>0$ and all~$n\in\B_j$, there exist~$a\in\cD$ and~$m\in \B_{j-1}$ such that~$n + a \in qm + \cD$.

  Next we let~$\{0\}\subset \B_{st}\subset\Ok$ be the smallest set such that~$\B_{st} + \cD + \cD \subset \cD + q\B_{st}$; the existence of~$\B_{st}$ is ensured by boundedness of~$\{\sum_{j=1}^r (n_{j,1}+n_{j,2}-n_{j,3})q^{-j}, r\geq 1, n_{j,k} \in \cD\}$. The set~$\B_{st}$ is our initial set of carries. We define a Markov chain on the set of states~$\B_{st}$ by setting, for every~$n\in\B_{st}$ and digit~$a\in\cD$, an edge
  $$ n \overset{a}{\longrightarrow} m \qquad \iff \qquad n + a \in qm + \cD, $$
  each digit~$a\in\cD$ being chosen with equal probability. Note that by construction, we do have~$m\in\B_{st}$. The main point is that~$0$ is an absorbing state for this chain. Therefore, a random walk on~$\B_{st}$ of length~$\rho\in\N$, starting at any given vertex, has probability~$O(c^\rho)$ of not ending at~$0$, for some~$c\in (0, 1)$.

  Let~$m = \sum_{j=0}^{\mu-1} m_j q^j$, with~$m_j\in\cD$. Suppose that there is an~$n\in\cN_{\nu-\rho}$ such that~$(m+n)_{[\nu,\infty]} \neq (m)_{[\nu,\infty]}$. Consider the sequences of carries~$(b_j)_{j\geq -1}$ in the addition~$m+n$. More precisely, if we let~$n = \sum_{j\geq 0} n_j q^j$ with~$n_j=0$ if~$j\geq\nu-\rho$, then~$b_{-1}=0$ and for all~$j\geq 0$, $b_j$ is the unique element of~$\Ok$ such that~$b_{j-1} + m_j + n_j \in \cD + qb_{j}$. By construction, we have~$b_j \in \B_{st}$ for all~$j\geq -1$. For all~$j\geq\nu$, the recurrence relation reads, with our above notations,
  $$ b_{j-1} \overset{m_j}{\longrightarrow} b_j. $$
  Our hypothesis on~$m$ and~$n$ implies that~$b_\nu \neq 0$. Therefore, the tuple~$(m_j)_{\nu-\rho< j \leq \nu}$ describes a walk on~$\B_{st}$ of length at least~$\rho$, not ending at~$0$. The number of such tuples is at most~$O(c^{\rho})$, and so the number of possibilities for~$m$ is at most~$O(Q^{\mu - \eta_2 \rho})$ with~$\eta_2 = -(\log c)/\log Q > 0$
\end{proof}

\begin{remark}
  We will call any admissible constant~$\eta_2$ in Lemma~\ref{lem:carry} a \emph{carry constant}. When~$K = \Q(i)$, we may choose
  $$ \eta_2 =
  \begin{cases}
    0.238186\dots, & (q, \cD) = (-1+i, \{0, 1\}), \\
    0.195636\dots, & (q, \cD) = (-2+i, \{0, 1, 2, 3, 4\}), \\
    0.053205\dots, & (q, \cD) = (-2+i, \{0, -2i, 2, 3, 4\}).
  \end{cases} $$
  These values were obtained by approximating the spectral radius of the adjacency matrix associated with the graph on~$\B_{st}$ considered above (with the absorbing state removed).
\end{remark}

\subsubsection{Harmonic analysis of the fundamental tile}\label{sec:harm-analys-fund}

In this section we study some harmonic analytic properties of the indicator function of the fundamental tile defined in~\eqref{eq:def-F}.
For this purpose we will study the closure of this tile in~$\R^d$,
$$ \FR = \bar{\iota^{-1}(\cF)} \subset \R^d $$
In our context the set~$\FR$ plays the rôle of the unit interval~$[0, 1]$ in~\cite{MR}. For example, when~$d=1$, $q<-1$ and~$\cD = \{0, \dotsc, \abs{q}-1\}$, we have explicitly
$$ \FR = \Big[\frac{q}{1-q}, \frac{1}{1-q}\Big]. $$
In general however, the set~$\FR$ is of a more complicated nature, and is a main object of study in the theory of self-similar tilings of~$\R^n$ (see~\cite{LagariasWang1996}).

Here, we have assumed that~$(q, \cD)$ is a FNS; we refer to Proposition~2.1 of~\cite{MTT-Fractal} for general properties of~$\FR$. In particular, $\FR$ is compact, measurable with Lebesgue measure~$\meas(\FR) = 1$; by Theorem~1 of~\cite{Vince1995}, $\FR$ contains an open neighborhood of the origin; and finally
$$ \meas( \FR \cap (\FR + a)) = 0 \qquad (a\in\Z^d\smallsetminus\{0\}). $$


The set~$\FR$ has been studied in a variety of cases:
\begin{enumerate}
\item For CNS and $d=2$, in~\cite{Gilbert1986} for $K = \Q(i)$, and in~\cite{Thuswaldner,Thuswaldner2001} for all quadratic fields.
\item For CNS and arbitrary~$d$, under a generic condition on the minimal polynomial of~$q$, the upper-box dimension of~$\partial\FR$ is obtained in~\cite{ScheicherThuswaldner2002}.
\item For general FNS, Keesling~\cite{Keesling1999} has shown that the Hausdorff dimension $\dim_H(\partial\FR)$ is always strictly less than~$d$, and that it can be arbitrarily close to~$d$. Note that~$\dim_H(\partial\FR)\geq d-1$, and that equality is achieved in the case~$q = -2$, $\cD = \{ \sum_{i\in I} \omega_i, I\subset\{1, \dotsc, d\}\}$, for which~$\FR = [-\tfrac23,\tfrac13]^d$. Upper and lower bounds on~$\dim_H(\partial\FR)$ in terms of the carry constant~$\eta_2$ are obtained in~\cite{MTT-Fractal}; the bounds coincide when~$\rh = \rb = 1$.
\item The topological properties of~$\FR$ can sometimes be counter-intuitive, in particular, it can be disconnected, see~\cite[Figure~2.1]{LagariasWang1996} and~\cite[Th. 2.3]{GroechenigHaas1994}.
\end{enumerate}

In the present work, we will require some information of the decay of the Fourier transform of the characteristic function of~$\FR$. This is ultimately due to the fact that the information we will use of the function~$f$ concerns its correlation with additive characters~\eqref{eq:Fourier-prop} and its behaviour with respect to the digital expansion~\eqref{eq:carry-prop}.

For~$x, \xi \in \R^d$, let
$$ \chi(x) = \1_{x\in\FR}, \qquad \chih(\xi) = \int_{\FR} \e(-\sg x, \xi\sd)\dd x.  $$
For~$\lambda\in\N_{>0}$, define
$$ \psi_\lambda(x) := \sum_{k\in\Z^d} \chi(q^\lambda(x+k)), $$
where~$q$ is viewed as a linear map of~$\R^d$ by~$qx = \iota^{-1}(q\iota(x))$ for all~$x\in\R^d$.
The function~$\psi_\lambda$ is~$\Z^d$-periodic, and its Fourier coefficients are
$$ \psih_\lambda(\xi) = Q^{-\lambda}\chih(\qt^{-\lambda}\xi) \qquad (\xi\in\Z^d), $$
where~${\tilde q}$ is the adjoint of~$q$.

The less regular~$\partial\FR$ is, the more slowly the function~$\chih$ decays: for instance, in the case~$K=\Q(i)$, $q=-1+i$, $\cD = \{0, 1\}$, it is shown by Cohen and Daubechies~\cite[eq. (5.3)]{CohenDaubechies1993} that~$\chih(\xi) \ll \|\xi\|^{-1/2}$ for all~$\xi\in\R^d$, the exponent~$1/2$ being in fact optimal. This and related questions are known in wavelet theory as the regularity problem of self-refinable functions; see in particular~\cite{Cohen1996,Pollicott2008}. Here we are largely able to avoid this question altogether. We essentially require two informations: a truncated Fourier expansion of~$\psi_\lambda$, and an estimate for~$L^2$ norms.

In the context of distribution of $q$-additive functions, this difficulty has been encountered in~\cite{GittenbergerThuswaldner2000} in the case~$K=\Q(i)$, and~\cite{Madritsch2010} for general~$K$ (see~\cite{Steiner2002} for related earlier computations in the context of Parry expansions). To truncate the Fourier series of~$\psi_\lambda$, one wishes to smooth out the function~$\chih$. In the above-mentioned works, this is done by convolving with the characteristic function of a hypercube (Urysohn approximation), however, it is technically convenient to use smooth, compactly supported majorants, so that sums over lattices can be estimated more easily by Poisson summation.

\begin{lemma}\label{lem:Fourier-psi}
  Let~$\lambda, \tau\in\N_{>0}$. There exist complex numbers~$(a_{\lambda, \tau}(\xi))_{\xi\in\Z^d}$ and~$(b_{\lambda, \tau}(\xi))_{\xi\in\Z^d}$ satisfying the following.
  \begin{enumerate}[(i)]
    \item For all fixed~$A\geq 0$, we have
    \begin{equation}
      \abs{a_{\lambda, \tau}(\xi)} \ll_{A} Q^{-\lambda} (1 + \|\qt^{-\lambda-\tau}\xi\|)^{-A}, \qquad \abs{b_{\lambda, \tau}(\xi)} \ll_{A}  Q^{-\lambda-\eta_2\tau} (1 + \|\qt^{-\lambda-\tau}\xi\|)^{-A}.\label{eq:fourier-bound-coeff}
    \end{equation}
    \item The functions~$A_{\lambda, \tau}(x)$ and~$B_{\lambda, \tau}(x)$ defined by
    $$ A_{\lambda, \tau}(x) = \sum_{\xi\in\Z^d} a_{\lambda, \tau}(\xi) \e(\sg \xi, x \sd), \qquad B_{\lambda, \tau}(x) = \sum_{\xi\in\Z^d} b_{\lambda, \tau}(\xi) \e(\sg \xi, x \sd) $$
    satisfy
    $$ \abs{\psi_{\lambda}(x) - A_{\lambda, \tau}(x)} \leq B_{\lambda, \tau}(x). $$
    \item For all~$\kappa\in\{0, \dotsc, \lambda\}$ and~$\xi_0\in\Z^d$, we have
    \begin{equation}
      \sum_{\xi\in\Z^d} \abs{a_{\lambda, \tau}(\xi_0 + \qt^\kappa \xi)}^2 \ll Q^{-\lambda-\kappa},\label{eq:L2-A-AP}
    \end{equation}
    \begin{equation}
      \sum_{\xi\in\Z^d} \abs{b_{\lambda, \tau}(\xi_0 + \qt^\kappa \xi)}^2 \ll Q^{-\lambda-\kappa-\eta_2\tau}.\label{eq:L2-B-AP}
    \end{equation}
  \end{enumerate}
\end{lemma}
\begin{proof}
  Let~$\phi:\R^d\to[0, 1]$ be a smooth function satisfying
  $$ \1_{\|x\|\leq 1/2} \leq \phi(x) \leq \1_{\|x\|\leq 2}, \qquad \int_{\R^d} \phi(x)\dd x = 1, $$
  where~$\|x\|$ is the euclidean norm. Define~$\phi_\tau := Q^\tau \phi(q^\tau x)$. Define
  \begin{equation}
    \chi_\tau := \chi \ast \phi_\tau,\label{eq:def-chitau}
  \end{equation}
  and let~$V_1 = \partial\FR + q^{-\tau}B(0,2)$ and~$V_2 = \partial\FR + q^{-\tau}B(0,2) + q^{-\tau}B(0, 2)$. Note that~$V_1\subset V_2$, and
  $$ \{x\in\R^d:\ \chi(x)\neq \chi_\tau(x)\} \subset V_1. $$
  We focus first on~$V_2$. Let~$\rho$ be an integer such that~$B(0,2) + \FR + (-\FR) \subset q^\rho\FR$. Each number~$x\in q^\tau V_2$ can be decomposed uniquely as~$x=m+y$ where~$m\in\Z^d$ and~$y\in\FR$.
  
By hypothesis, there exist~$u_1, u_2 \in\FR$, such that~$m+y-q^\rho(u_1+u_2)\in\partial(q^\tau\FR)$. Since~$q^\tau\FR$ is tiled by~$\Z^d$-translates of~$\FR$, and~$B(0, 1 + {\rm diam}(\FR)) \subseteq q^\rho\FR$, we deduce the existence of~$z_+, z_- \in \FR$ such that~$m-q^\rho(u_1+u_2+z_+)\in \Z^d\cap q^\tau\FR$ and~$m-q^\rho(u_1+u_2+z_-) \in\Z^d \smallsetminus q^\tau\FR$. Therefore, by~\eqref{eq:def-Lambda}, there exist~$n_+, n_- \in\Z^d \cap q^{\rho+3\Lambda}\FR$ such that~$m+n_+\in q^\tau \FR$ and~$m+n_-\notin q^\tau\FR$. By Lemma~\ref{lem:carry} the number of such~$m$ is at most~$\ll Q^{(1-\eta_2)\tau}$, and therefore
  \begin{equation}
    \meas(V_2) = Q^{-\tau} \sum_{m\in\Z^d} \int_{\FR} \1(m+y \in q^\tau V_2)\dd y \ll Q^{-\eta_2\tau}.\label{eq:fourier-bound-volV2}
  \end{equation}
  We now write
  $$ \abs{\chi(x) - \chi_\tau(x)} \leq \1_{V_1}(x) \leq (\1_{V_2}\ast \phi_\tau)(x). $$
  Define now the smooth,~$\Z^d$-periodic functions
  $$ A_{\lambda, \tau}(x) = \sum_{k\in\Z^d} \chi_\tau(q^\lambda(x+k)), \qquad B_{\lambda, \tau}(x) = \sum_{k\in\Z^d} (\1_{V_2}\ast \phi_\tau)(q^\lambda(x+k)). $$
  We have~$B_{\lambda, \tau}(x) \ll 1$. Let~$a_{\lambda, \tau}(\xi)$ and~$b_{\lambda, \tau}(\xi)$ be the coefficients in the Fourier expansions of~$A_{\lambda, \tau}(x)$ and~$B_{\lambda, \tau}(x)$, respectively. We have
  $$ a_{\lambda, \tau}(\xi) = Q^{-\lambda} \chih( \qt^{-\lambda} \xi) \phih( \qt^{-\lambda-\tau} \xi), $$
  $$ b_{\lambda, \tau}(\xi) = Q^{-\lambda} \phih( \qt^{-\lambda-\tau} \xi) \int_{V_2} \e(\tc{\qt^{-\lambda} \xi,  y})\dd y. $$
  By partial summation, we have the bound
  \begin{equation}
    \abs{\phih(\xi)} \ll_A (1+\|\xi\|)^{-A}\label{eq:bound-der-phi}
  \end{equation}
  for any~$A>0$. By~\eqref{eq:fourier-bound-volV2}, we deduce parts~(i) and~(ii) as claimed.

  For part~(iii), let us consider the case of~$b_{\lambda, \tau}$. By absolute convergence and orthogonality (Lemma~\ref{lem:orthog}), we have
  \begin{equation}
    \sum_{\xi\in\Z^d} \abs{b_{\lambda, \tau}(\xi_0+\qt^\kappa \xi)}^2 = Q^{-\kappa} \sum_{\ell \in \Z^d/B^\kappa \Z^d} \e(-\tc{\qt^{-\kappa}\xi_0, \ell}) \sum_{\xi\in\Z^d} \abs{b_{\lambda, \tau}(\xi)}^2 \e(\tc{\qt^{-\kappa} \xi, \ell}).\label{eq:fourier-parseval-B}
  \end{equation}
  On the other hand, by Poisson summation, we have
  $$ \sum_{\xi\in\Z^d} \abs{b_{\lambda, \tau}(\xi)}^2 \e(\sg \qt^{-\kappa} \xi, \ell \sd) = \int_{\R^d/\Z^d} B_{\lambda,\tau}(x+q^{-\kappa}\ell) B_{\lambda,\tau}(x)\dd x. $$
  Let~$V_3 = V_2 + q^{-\tau}B(0,2)$; it is a bounded set, and by reasoning similarly as in~\eqref{eq:fourier-bound-volV2}, we have~$\meas(V_3) \ll Q^{-\eta_2 \tau}$. By construction, the support of~$B_{\lambda, \tau}$ is included in~$\Z^d + q^{-\lambda}V_3$. Therefore, for any given~$x\in\R^d$, the integrand above vanishes unless there exists~$k\in\Z^d$ such that
  $$ q^\kappa k + \ell \in q^{\kappa-\lambda}(V_3 - V_3), $$
  and the latter is a bounded set since~$\kappa\leq \lambda$, therefore, there are at most a bounded number of~$\ell$ contributing to the sum on the right-hand side of~\eqref{eq:fourier-parseval-B}. We deduce
  $$ \sum_{\xi\in\Z^d} \abs{b_{\lambda, \tau}(\xi_0+\qt^\kappa \xi)}^2 \ll Q^{-\kappa-\lambda-\eta_2\tau} $$
  by~\eqref{eq:fourier-bound-volV2}, and we obtain~\eqref{eq:L2-B-AP}. The bound~\eqref{eq:L2-A-AP} is proved using identical argument in a simpler way; therefore we do reproduce the details.
\end{proof}

\begin{remark}
  It is an important point in the proof that the smooth majorant~$\phi$ is scaled down by powers of~$q$, rather than \emph{e.g.} homotheties (in which case the carry constant~$\eta_2$ would be replaced by the upper-box dimension of~$\partial\FR$, which is less understood in general~\cite{MTT-Fractal,ScheicherThuswaldner2002}).

  There would be much technical simplification to be gained, in our later arguments, by having an analogue of Vaaler's construction of band-limited majorants~\cite{Vaaler1985}, as was used in~\cite{MR}; our attempts were unsuccessful.
\end{remark}

In the sequel, for~$x\in K$ and $\xi\in\Ov$, we will denote
\begin{equation}
  \label{eq:notation-ab-iota}
  \psi_{\lambda}(x) = \psi_{\lambda}(\iota^\vee(x)), \qquad a_{\lambda,\tau}(\xi) = a_{\lambda,\tau}(\iota^\vee(\xi))
\end{equation}
and similarly for~$b_{\lambda,\tau}$.

\subsection{Fourier estimates}

In this section, we prove an analogue of Lemma~10 of~\cite{MR}, concerning restricted~$L^2$ estimates for the discrete Fourier transform of~$f$.

\subsubsection{Fourier property over the middle digits}

The next lemma concerns a variant of the Fourier property~\eqref{eq:Fourier-prop}, in which the sum is effectuated only over the middle digits. This additional flexibility comes at the price of a numerically smaller gain in the exponent.
\begin{lemma}\label{lem:FourierProp-mid}
  Let~$\alpha, \beta, \delta \in\N$ satisfy~$\delta \leq \alpha + \beta$, and let~$\lambda := \alpha + \beta + \delta$. For all~$f$ satifying the Fourier property, we have
  $$ \frac1{Q^\beta}\sum_{u_1 \in \cN_\beta} f_{\kappa+\lambda}(q^\kappa(u_0 + q^\alpha u_1 + q^{\alpha+\beta}u_2)) \e(\tc{u_1 t}) \ll Q^{-\eta' \gamma(\lambda) + \alpha + \delta} $$
  uniformly for~$u_0\in\cN_\alpha$, $u_2\in\cN_\delta$, $t\in K$ and~$\kappa\leq c\lambda$, where~$\eta' = \eta_2(1+\eta_2)^{-1}$.
\end{lemma}

\begin{proof}
  We recall the notation~\eqref{eq:def-elambdat}. By orthogonality, our sum is
  \begin{align*}
    {}& \frac1{Q^\beta}\sum_{u_\in \cN_\lambda} f(q^\kappa u) \e_\alpha(ut) \1_{u-u_0\in q^\alpha\Ok} \1_{\frac{u}{q^{\alpha+\beta}}\in u_2 + \cF + q^\delta \Ok} \e_\kappa(-u_0t)\e_{-\beta}(-u_2 t) \\
    = {}& \e_\kappa(-u_0t)\e_{-\beta}(-u_2 t)\sum_{\ell\in\Ok/q^\alpha} \e_\alpha(-u_0\ell) S\Big(\ell, \frac{t+\ell}{q^\alpha}\Big),
  \end{align*}
  where
  $$ S(\ell, t) = \frac1{Q^{\alpha+\beta}} \sum_{u \in \cN_\lambda} f(q^\kappa u) \e(\tc{ut}) \psi_\delta\Big(\frac{u}{q^\lambda} - \frac{u_2}{q^\delta}\Big). $$
  Let~$\tau\in\N$, $\tau\leq \alpha+\beta$, be a parameter. At this point, we wish to apply Lemma~\ref{lem:Fourier-psi}, to replace~$\psi_{\delta}$ by its smoothed version~$A_{\delta, \tau}$. The ensuing main term is
  $$ S_\tau(\ell, t) = \frac1{Q^{\alpha+\beta}} \sum_{u \in \cN_\lambda} f(q^\kappa u) \e(\tc{ut}) A_{\delta,\tau}\Big(\frac{u}{q^\lambda} - \frac{u_2}{q^\delta}\Big), $$
  where we recall the notation~\eqref{eq:notation-ab-iota}. We Fourier expand, use the Fourier property and the bound~\eqref{eq:fourier-bound-coeff} (along with Poisson summation~\eqref{eq:poisson-K}), getting
  \begin{align*}
    S_\tau(\ell, t)
    = {}& \sum_{\xi\in\Ov} a_{\delta,\tau}(\xi) \e_\delta(-\xi u_2) \frac1{Q^{\alpha+\beta}} \sum_{u\in\cN_\lambda} f(q^\kappa u) \e_\lambda(u(\xi+tq^\lambda)) \\
    \ll {}& Q^{\delta-\gamma(\lambda)} \sum_{\xi\in\Ov} |a_{\delta,\tau}(\xi)| \\
    \ll {}& Q^{\delta-\gamma(\lambda) + \tau}.
  \end{align*}
  We justify the Fourier truncation (replacement of~$\psi_\sigma$ by~$A_{\sigma,\tau}$) in a similar way to the proof of Lemma~\ref{lem:Fourier-psi} (see~\eqref{eq:fourier-bound-volV2}). Assuming~$\delta$ is large enough in terms of~$q$, by Lemma~\ref{lem:Fourier-psi}, we know that
  \begin{equation}
    \psi_\delta\Big(\frac{u}{q^\lambda} - \frac{u_2}{q^\delta}\Big) - A_{\delta,\tau}\Big(\frac{u}{q^\lambda} - \frac{u_2}{q^\delta}\Big)
    \ll \1\Big(\frac{u}{q^{\lambda-\delta}} - u_2 \in \partial\cF + B(0, 2/H) + q^\delta \Ok\Big).\label{eq:lemme10-bound-psi-A}
  \end{equation}
  Let~$\tau\in\N$ be the largest integer such that~$B(0, 2/H) \subset q^{-\tau} \cF$; note that~$\tau \geq O(1) + \frac{\log H}{2\rh \log Q}$.
  By a reasoning similar to \eqref{eq:fourier-bound-volV2}, we may find an element~$n'\in\cN_{\alpha+\beta-\tau+4\Lambda}$ such that~$r_{\alpha+\beta, \infty}(u - q^{\alpha+\beta}u_2 - q^\lambda k + n') \neq r_{\alpha+\beta, \infty}(u - q^{\alpha+\beta}u_2 - q^\lambda k)$.
  Denoting~$u' = u - q^{\alpha+\beta}u_2 - q^\lambda k$, we deduce
  \begin{align*}
    |S(\ell,t) - S_\tau(\ell,t)| \ll {}& \frac1{Q^{\alpha+\beta}} \sum_{u\in\cN_\lambda}\1\Big(\frac{u}{q^{\lambda-\delta}} - u_2 \in \partial\cF + B(0, 2/H) + q^\delta \Ok\Big) \\
    \ll {}& \frac1{Q^{\alpha+\beta}} \sum_{k\in\cN_{6\Lambda}} \card\{u'\in\cN_{\lambda+3\Lambda}, \exists n\in\cN_{\alpha+\beta-\tau+4\Lambda}, \\
    {}& \hspace{10em} r_{\alpha+\beta,\infty}(u') \neq r_{\alpha+\beta,\infty}(u'+n)\} \\
    \ll {}& Q^{\delta-\eta_2\tau}
  \end{align*}
  by Lemma~\ref{lem:carry}. Using~\eqref{eq:bestbound-gamma}, we may optimize~$\tau$ under the condition~$\tau\leq \alpha+\beta$ and find
  $$ S(\ell, t) \ll Q^{-\eta' \gamma(\lambda) + \delta} $$
  with~$\eta' = \eta_2(1+\eta_2)^{-1}$, as claimed.
\end{proof}

\subsection{Sums over lattices}

In this section, we estimate sums over lattices that will appear repeatedly later in our arguments. As we already mentioned an additional difficulty when~$d>1$ is the possibility of the multiplication by~$q$ skewing the lattice~$\Ok$ (see~\cite[pp. 203--204]{Madritsch2010}). This issue does not occur in~$\O$ thanks to the additional structure of the unit group~$\Ok^\ast$.

\begin{lemma}\label{lem:lattice-count}
  Let~$\fa$ be a fractional ideal, and~$R\geq 0$. Then
  $$ \card\{h\in \fa\smallsetminus\{0\}, \|h\| \leq R\} \ll N(\fa)^{-1} R^d. $$
\end{lemma}
\begin{proof}
  As in the proof of Lemma~\ref{lem:large-sieve}, let~$\fc\subset\Ok$ be an ideal in the same class as~$\fa$, with~$1\leq N(\fc) \ll 1$. Then~$\fa \subset \fc^{-1}\fa = (u)$ for some~$u\in K$ with~$N(u)\asymp N(\fa)$, and by multiplying by units we may impose~$\abs{u^\pi} \asymp N(u)^{1/d}$. Then
  \begin{align*}
    \card\{h\in \fa, h\neq 0, \|h\| \leq R\} \leq {}& \card\{h\in \Ok, h\neq 0, \|uh\| \leq R\} \\ \leq {}& \card\{h\in \Ok, h\neq 0, \|h\| \leq CRN(u)^{-1/d}\}
  \end{align*}
  for~$C \ll 1$, since~$\|uh\|\leq R$ implies~$\|h\| \ll \|u^{-1}\| \ll N(u)^{-1/d}$. The last cardinality is simple to evaluate, since the basis~$(\omega_j)$ of~$\Ok$ satisfies~$\|\omega_j\| \ll 1$.
\end{proof}

\begin{lemma}\label{lem:solution-gcd}
  Let~$\ft$ be an integral ideal,~$\alpha, \beta\in \ft^{-1}$ and $q\in\Ok$. Then
  $$ \card\{ n\in\Ok/q, \alpha n + \beta \in q\ft^{-1}\} \leq N( \alpha \ft + (q)). $$
\end{lemma}
\begin{proof}
  By homogeinizing, we have
  $$ \card\{ n\in\Ok/q, \alpha n + \beta \in q\ft^{-1} \} \leq \card\{n\in\Ok/q, n\alpha \in q\ft^{-1}\}. $$
  Let~$\alpha_0 = \alpha \ft$ and~$\fd = \alpha_0 + (q)$. The condition~$n\alpha \in q\ft^{-1}$ becomes~$(q) \mid n\alpha_0$ and so~$n\in I$, where~$I = (q)\fd^{-1}$ is an integral ideal. But~$|I / (q)| = |\Ok / \fd| = N(\fd)$ as claimed.
\end{proof}

\begin{lemma}\label{lem:smooth-doublesum}
  Let~$s_1, s_2, q \in \Ok$ with~$q\mid s_j$. Let~$\ft$ be an integral ideal, and~$\alpha, \beta\in\ft^{-1}$. Let~$V_0 : \R^d \to \C$ be in the Schwartz class, and define two functions on~$K$ by~$V = V_0 \circ \iota^{-1}$ and~${\hat V} = {\hat V_0} \circ (\iota^\vee)^{-1}$. Then
  $$ \sum_{m\in\Ok} V\Big(\frac m{s_1}\Big) \Big| \sum_{n\in\Ok} V\Big(\frac n{s_2}\Big) \e\Big(\tc{\frac{n(\alpha m + \beta)}{q}}\Big)\Big| \ll_V \frac{N(s_1) N(s_2)}{N(q)} N(\alpha \ft + (q)) N(\ft), $$
  where the implied constant depends at most on~$K$, $q$, $\cN$ and~$V$.
\end{lemma}
\begin{proof}
  By Poisson summation~\eqref{eq:poisson-K}, we have
  \begin{align*}
    {}& \sum_{m\in\Ok} V\Big(\frac m{s_1}\Big) \Big| \sum_{n\in\Ok} V\Big(\frac n{s_2}\Big) \e\Big(\tc{\frac{n(\alpha m + \beta)}{q}}\Big)\Big| \\
    \leq {}& N(s_2) \sum_{m\in\Ok} V\Big(\frac m{s_1}\Big)  \sum_{\xi\in\Ov} \Big|{\hat V}\Big(s_2\Big(\frac{\alpha m + \beta}{q} + \xi\Big)\Big)\Big| \\
    \leq {}& N(s_2) \sum_{m\in\Ok} V\Big(\frac m{s_1}\Big)  \sum_{\xi\in\ft^{-1}\Ov} \Big|{\hat V}\Big(s_2\Big(\frac{\alpha m + \beta}{q} + \xi\Big)\Big)\Big| \\
    \leq {}& N(s_2) \sum_{m_0\in\Ok/q} \sum_{\xi\in\ft^{-1}\Ov} \Big|{\hat V}\Big(s_2\Big(\frac{\alpha m_0 + \beta}{q} + \xi\Big)\Big)\Big| \ssum{m\in\Ok \\ m \equiv m_0\mod{q}} V\Big(\frac{m}{s_1}\Big).
  \end{align*}
  Again by Poisson summation,
  \begin{align*}
    \Bigg|\ssum{m\in\Ok \\ m \equiv m_0\mod{q}} V\Big(\frac{m}{s_1}\Big)\Bigg|
    = {}& \frac{N(s_1)}{N(q)} \Bigg|\sum_{\omega\in\Ov} {\hat V}\Big(\frac{s_1 \omega}{q}\Big) \e\Big(\tc{\frac{-m_0 \omega}q}\Big)\Bigg| \\
    \leq {}& \frac{N(s_1)}{N(q)} \sum_{\omega\in\Ov} \Big|{\hat V}\Big(\frac{s_1 \omega}{q}\Big)\Big| \\
    \ll_V {}&  \frac{N(s_1)}{N(q)}.
  \end{align*}
  Next, by Lemma~\ref{lem:solution-gcd} with~$\ft\gets \ft \fD_K$ (where we recall that~$\fD_K=(\Ov)^{-1}$ is the different ideal), for each~$\gamma\in\ft^{-1}\Ov/q$, the number of~$m_0\in\Ok/q$ such that~$\alpha m_0 + \beta \equiv \gamma \mod{q\ft^{-1}\Ov}$ is at most~$N(\alpha \ft \fD_k + (q)) \ll N(\alpha\ft+(q))$. Therefore,
  \begin{align*}
    \sum_{m_0\in\Ok/q} \sum_{\xi\in\ft^{-1}\Ov} \Big|{\hat V}\Big(s_2\Big(\frac{\alpha m_0 + \beta}{q} + \xi\Big)\Big)\Big|
    \ll {}& N(\alpha \ft + (q)) \sum_{\xi\in\ft^{-1}\Ov} \sum_{\gamma\in\ft^{-1}\Ov/q} \Big|{\hat V}\Big(s_2\Big(\frac{\gamma}{q} + \xi\Big)\Big)\Big| \\
    = {}& N(\alpha \ft + (q)) \sum_{\xi\in\ft^{-1}\Ov} \Big|{\hat V}\Big(\frac{s_2 \xi}{q}\Big)\Big| \\
    \ll_V {}& N(\alpha \ft + (q)) \sum_{\xi\in\ft^{-1}\Ov} \frac1{(1+\|\xi\|)^{d+1}}.
  \end{align*}
  By Lemma~\ref{lem:lattice-count} and partial summation, we obtain~$\sum_{\xi\in\ft^{-1}\Ov} (1+\|\xi\|)^{-d-1} \ll_V N(\ft)$, which concludes our proof.
\end{proof}

\begin{lemma}\label{lem:sum-gcd-h}
  Let~$R\geq 0$, $\ft$ be an integral ideal, and~$q\in\Ok$. Then
  $$ \ssum{h\in \ft^{-1}\Ok \\ 0<\|h\| \leq R} N(h\ft + (q)) \ll \tau(q) R^d N(\ft), $$
  where~$\tau(q)$ is the number of integral ideal divisors of~$(q)$, and the implicit constant depends on~$K$ only.
\end{lemma}
\begin{proof}
  In our sum, we sort according to the ideal~$\fd=h\ft + (q)$ and use Lemma~\ref{lem:lattice-count}, getting
  \begin{align*}
    \ssum{h\in\ft^{-1}\Ok \\ 0<\|h\| \leq R} N(h\ft + (q)) \leq {}& \sum_{\fd\mid q} N(\fd) \ssum{h\in \fd \ft^{-1}\Ok \\ 0<\|h\|\leq R} 1 \\
    {}& \ll \sum_{\fd\mid q} R^d N(\ft) \\
    {}& \ll \tau(q) R^d N(\ft).
  \end{align*}
\end{proof}

\begin{lemma}\label{lem:sum-gcd-hh}
  Let~$\ft$ be an integral ideal, $q\in\Ok$ and~$R_0, R_1 \in\R_+$. Then
  $$ \ssum{h_0, h_1 \in \ft^{-1} \\ h_0 + h_1 \neq 0 \\ \| h_j \| \leq R_j} N((h_0+h_1)\ft + (q)) \ll \tau(q) N(\ft^2) (R_0+1)^d (R_0 + R_1)^d. $$
\end{lemma}
\begin{proof}
  Given a non-zero fractional ideal~$\fa\subset\ft^{-1}$, we have
  \begin{align*}
    \ssum{h_0, h_1 \in \ft^{-1} \\ \|h_j\| \leq R_j \\ 0 \neq h_0 + h_1 \in \fa } 1
    \leq {}& \ssum{h_0 \in \ft^{-1} \\ \|h_0\|\leq R_0} \ssum{h' \in \fa\smallsetminus\{0\} \\ \|h'\| \leq R_0 + R_1} 1 \ll N(\ft\fa^{-1}) (1+R_0)^d(R_0+R_1)^d
  \end{align*}
  By Lemma~\ref{lem:lattice-count}. The conclusion follows by setting~$\fa = \fd \ft^{-1}$ and summing over~$\fd\mid q$, against~$N(\fd) = \det(\fa) N(\ft)$, similarly as in Lemma~\ref{lem:sum-gcd-h}.
\end{proof}

\subsubsection{Incomplete $L^2$ bound on the Fourier transform}

The statements of this section depend of certain parameters which will be introduced later in Section~\ref{sec:type-ii-sums}. For now, we let~$\mu, \mu_0, \mu_1$ and~$\mu_2$ be natural numbers subject to
$$ \mu_0 < \mu_1 < \mu < \mu_2. $$
We let~$\sigma = \mu_2 - \mu_0$, and define, for all~$n\in\Ok$,
\begin{equation}
  g(n) = f_{\mu_2}(q^{\mu_0}n)\bar{f_{\mu_1}}(q^{\mu_0}n).\label{eq:def-g}
\end{equation}
We recall the definition of the discrete Fourier transform of~$g$,
\begin{equation}
  \gh(h) := \frac1{Q^{\sigma}} \sum_{u\in\Ok/q^\sigma} g(u)\e_\sigma(-uh).
\end{equation}

\begin{proposition}\label{prop:Fourier-middle}
  With the above notation and hypotheses, for all~$t\in K$ and~$\lambda\in\N$, if~$c^{-1}\mu_0 \leq \lambda \leq \sigma$, then we have
  $$ \ssum{h \in \Ov \\ \|h/q^{\sigma-\lambda}\|\leq 1} |\gh(h + t)|^2 \ll Q^{2(\mu_1 - \mu_0)} ( Q^{-\eta'' \gamma(\lambda)} + Q^{-\eta_1(\sigma-\lambda)}) $$
  where~$\eta'' = 2\eta_1\eta_2(2+\eta_1)^{-1}(1+\eta_2)^{-1}$.
\end{proposition}

\begin{proof}
  The proof mirrors that of~\cite{MR}: the point is that we may use the carry property to essentially factor~$\gh(h+t)$ as a sum over~$\cN_\lambda$ times a sum over~$\cN_{\sigma-\lambda}$. Parseval's identity will be applied to the second factor, to recover the full~$h$-sum, while the Fourier property on the first factor will allow for an extra saving. For each~$h\in\cN_{\sigma-\lambda}$, we write
  \begin{align*}
    \gh(t) = {}& \frac1{Q^\sigma} \sum_{u\in\cN_\lambda} \sum_{v\in\cN_{\sigma-\lambda}} g(u+q^\lambda v) \e_\sigma(-ut)\e_{\sigma-\lambda}(-vt).
  \end{align*}
  Here, we have by periodicity
  $$ g(u + q^\lambda v) = f(q^{\mu_0}(u + q^\lambda v)) \bar{f_{\mu_1}(q^{\mu_0}u)}. $$
  Let~$\rho_3 \leq \sigma - \lambda$. By the carry property~\eqref{eq:carry-prop}, we have
  $$ f(q^{\mu_0}(u + q^\lambda v)) = f_{\mu_0 + \lambda + \rho_3}(q^{\mu_0}(u+q^\lambda v)) f(q^{\mu_0+\lambda}v) \bar{f_{\mu_0+\lambda+\rho_3}(q^{\mu_0+\lambda}v)} $$
  except when~$u+q^\lambda v\in\cW_{\rho_3}$, for some set~$\cW_{\rho_3}$ of cardinality at most~$Q^{\sigma-\eta_1\rho_3}$. Therefore,
  $$ \gh(t) = G_1(t) + G_2(t), $$
  \begin{align*}
    G_1(t) = {}& \frac1{Q^\sigma} \sum_{u\in\cN_\lambda} \sum_{v\in\cN_{\sigma-\lambda}} f_{\mu_0 + \lambda + \rho_3}(q^{\mu_0}(u+q^\lambda v)) f(q^{\mu_0+\lambda}v) \times \\
    {}& \hspace{6em} \times \bar{f_{\mu_0+\lambda+\rho_3}(q^{\mu_0+\lambda}v)}  \bar{f_{\mu_1}(q^{\mu_0}u)} \e_\sigma(-ut)\e_{\sigma-\lambda}(-vt), \\
    G_2(t) = {}& \frac1{Q^\sigma} \sum_{w\in\cN_\sigma} b(w) \e_\sigma(-wt),
  \end{align*}
  with~$|b(w)|\leq 2$, supported on~$\cW_{\rho_3}$.

  In the sum~$G_1(t)$, we detect the congruence class~$w = v\mod{q^{\rho_3}}$ by orthogonality, and write
  $$ f_{\mu_0+\lambda+\rho_3}(q^{\mu_0+\lambda}v) = f_{\mu_0+\lambda+\rho_3}(q^{\mu_0+\lambda}w), $$
  $$ f_{\mu_0 + \lambda + \rho_3}(q^{\mu_0}(u+q^\lambda v)) = f_{\mu_0 + \lambda + \rho_3}(q^{\mu_0}(u+q^\lambda w)). $$
  We obtain
  $$ G_1(t) = \sum_{\ell\in\Ov/q^{\rho_3}} d_h(\ell) \frac1{Q^{\rho_3}} \sum_{w\in\Ok/q^{\rho_3}} \e_{\rho_3}(-\ell w) c_h(w), $$
  where
  $$ d_t(\ell) = \frac1{Q^{\sigma-\lambda}} \sum_{v\in\cN_{\sigma-\lambda}} f(q^{\mu_0+\lambda}v) \e_{\sigma-\lambda}(-vt) \e_{\rho_3}(v\ell), $$
  $$ c_h(w) = \frac{\bar{f_{\mu_0+\lambda+\rho_3}(q^{\mu_0+\lambda}w)}}{Q^\lambda} \sum_{u\in\cN_\lambda} f_{\mu_0+\lambda+\rho_3}(q^{\mu_0}(u+q^\lambda w)) \bar{f_{\mu_1}(q^{\mu_0}u)} \e_\sigma(-ut). $$
  By splitting again~$u = u_0 + q^{\mu_1 - \mu_0} u_1$ with~$u_0\in\cN_{\mu_1-\mu_0}$ and~$u_1 \in \cN_{\lambda-\mu_1+\mu_0}$, we get that under the additional assumption~$\rho_3\leq \lambda$,
  \begin{align*}
    |c_h(w)| \leq {}& \frac{1}{Q^{\mu_1-\mu_0}} \sum_{u_0 \in \cN_{\mu_1-\mu_0}} \Big| \sum_{u_1\in\cN_{\lambda-\mu_1+\mu_0}} f(q^{\mu_0}(u_0 + q^{\mu_1-\mu_0}u_1 + q^\lambda w)) \e_{\sigma-\mu_1+\mu_0}(-u_1t) \Big| \\
    \ll {}& Q^{-\eta' \gamma(\lambda+\rho_3) + \mu_1 - \mu_0 + \rho_3}
  \end{align*}
  by Lemma~\ref{lem:FourierProp-mid}. We can now sum over~$h$. Using the Cauchy--Schwarz inequality and Parseval's equality as in~\cite{MR}, we get
  $$ \ssum{h \in \Ov \\ \|h/q^{\sigma-\rho}\|\leq 1} |G_1(h+t)|^2 \leq \sup_{\substack{t'\in K \\ w\in\Ok/q^{\rho_3}}} \abs{c_{t'}(w)}^2 \sup_{\ell\in\Ov/q^{\rho_3}} \ssum{h\in\Ov \\ \|h/q^{\sigma-\lambda}\|\leq 1} \abs{d_h(\ell)}^2. $$
  We write
  $$ \ssum{h\in\Ov \\ \|h/q^{\sigma-\lambda}\|\leq 1} \abs{d_{h+t}(\ell)}^2 = \sum_{\alpha\in\Ov/q^{\sigma-\lambda}} \abs{d_{\alpha+t}(\ell)}^2 \ssum{h\in\Ov \\ h-\alpha\in q^{\sigma-\lambda}\Ov \\ \|h/q^{\sigma-\lambda}\|\leq 1} 1. $$
  Note that the last sum is~$O(1)$, and the remaining sum over~$\alpha$ is again~$O(1)$ by Parseval's identity. We deduce
  $$ \ssum{h \in \Ov \\ \|h/q^{\sigma-\rho}\|\leq 1} |G_1(h+t)|^2 \ll Q^{-2\eta' \gamma(\lambda) + 2(\mu_1 - \mu_0 + \rho_3)}. $$

  On the other hand, by Parseval's equality and reasoning as above,
  $$ \ssum{h\in\Ov \\ \|h/q^{\sigma-\lambda}\|\leq 1} |G_2(h+t)|^2 \leq \ssum{h\in\Ov \\ \|h/q^\sigma\|\ll 1} |G_2(h+t)|^2 \ll \frac1{Q^\sigma} \sum_{w\in\cN_\sigma} |b(w)|^2 \ll Q^{-\eta_1\rho_3}, $$
  and by optimising~$\rho_3$ (note that we always have~$\eta'\gamma(\lambda)\leq \lambda$ by~\eqref{eq:bestbound-gamma}), the result follows.
\end{proof}

\section{Type I sums}\label{sec:type-i-sums}

The following estimate is a generalization of Proposition~1 of~\cite{MR}.

\begin{proposition}\label{prop:type-i}
  Let~$f:\O \to \C$ satisfy the Carry and Fourier properties~\eqref{eq:carry-prop}--\eqref{eq:Fourier-prop}. Let~$V_0:\R^d\to\C$ be a smooth map, compactly supported inside~$\R^d\smallsetminus\{0\}$. Let~$V = V_0 \circ \iota^{-1} : K \to \C$, define~${\hat V} = {\hat V_0} \circ (\iota^\vee)^{-1}$ and
  $$ \Sigma_V := \sum_{\xi\in\Ov} |{\hat V}(\xi)|. $$
  Then for $\mu \leq \frac c{c+2}\nu$, we have
  \begin{equation}
    S_I = \sum_{m\in\cN_\mu} \Big| \sum_{n\in \Ok} V\Big(\frac{mn}{q^{\mu + \nu}}\Big)f(mn) \Big| \ll \Sigma_V\mu^{d+1} Q^{\mu + \nu - \frac{\eta_1}{1+\eta_1}\gamma(\nu-\mu)}.\label{eq:bound-tI}
  \end{equation}
  The implied constant depends on~$(q, \cD)$, and on the diameter of the support of~$V$.
\end{proposition}

\begin{remark}\ 
  \begin{itemize}
    \item The same bound holds, with the same proof, for the more general quantity
    \begin{equation}
      \sum_{m\in\cN_\mu} \max_{a\in\Ok / m} \Big| \sum_{n\in \Ok} V\Big(\frac{mn + a}{q^{\mu + \nu}}\Big)f(mn + a) \Big|.\label{eq:typeI-with-a}
    \end{equation}
    \item The bounds~\eqref{eq:bound-tI} and~\eqref{eq:typeI-with-a} can be viewed as a statement on cancellations of~$f(n)$ on average over arithmetic progressions~$n \equiv 0 \mod{m}$; this is an analogue of the Bombieri-Vinogradov theorem in the context of multiplicative number theory. Bounds of the type~\eqref{eq:bound-tI} go back to work of Fouvry and Mauduit~\cite{FouvryMauduit1996}. The quality of the bound~\eqref{eq:bound-tI} can be measured by the exponent of distribution, which is the maximum asymptotically allowable value for the ratio~$\frac{\mu}{\mu+\nu}$. As in~\cite{MR}, we have~$\vartheta = \frac c{2(c+1)}$, independently of~$\gamma$, and this value on the exponent is precisely the analogue of the Bombieri-Vinogradov theorem if~$c$ can be picked arbitrarily large; in both cases, the limitation arises from the large sieve inequality.
    \item Obtaining an exponent of distribution greater than~$1/2$ is a challenging question in general. In the sum-of-digits case~$f(n)=(-1)^{s_q(n)}$, such a result was obtained in~\cite{FouvryMauduit1996} with a value~$\vartheta\geq 0.55711$ (and a slightly larger exponent for~$q=2$). This has been improved to~$\vartheta \geq 2/3$ in~\cite{MuellnerSpiegelhofer2017}; a proof of the value~$\vartheta = 1$ has recently been announced by Spiegelhofer~\cite{Spiegelhofer2018}.
  \end{itemize}
\end{remark}

\begin{proof}
  First note that replacing~$\nu$ by~$\nu + C$, for some~$C$ depending on~$(q, \cD)$ and the diameter of~$\supp V$, and rescaling~$V$ accordingly, we may assume that~$V(x)\neq 0 \implies x\in\cF$. For any~$\ell\in\cN_{\mu+\nu}$, we have
  $$ V\Big(\frac{\ell}{q^{\mu+\nu}}\Big) = \ssum{u\in\Ok \\ q^{\mu+\nu} \mid u-\ell} V\Big(\frac{u}{q^{\mu+\nu}}\Big) $$
  by our hypothesis on the support of~$V$. Then
  \begin{align*}
    S_I {}& = \sum_{m\in\cN_\mu} \Big| \sum_{n\in \Ok} V\Big(\frac{mn}{q^{\mu + \nu}}\Big)f(mn) \Big| \\
    {}& = \sum_{m\in\cN_\mu} \frac1{N(m)} \Big| \sum_{k\in\Ov/m} \sum_{\ell\in \cN_{\mu+\nu}} \e\Big(\tc{\frac{k\ell}m}\Big )V\Big(\frac{\ell}{q^{\mu + \nu}}\Big)f(\ell) \Big| \\
    {}& = \sum_{m\in\cN_\mu} \frac1{N(m)Q^{\mu+\nu}} \Big| \sum_{k\in\Ov/m} \sum_{h\in\Ov/q^{\mu+\nu}} \sum_{\ell\in \cN_{\mu+\nu}} \e\Big(\tc{\frac{k\ell}m}\Big)
    \e_{\mu+\nu}(-h\ell) f(\ell) \times \\
    {}& \hspace{12em} \times \sum_{u\in\Ok} V\Big(\frac{u}{q^{\mu + \nu}}\Big) \e_{\mu+\nu}(hu) \Big|.
  \end{align*}
  The Poisson formula yields
  $$ \sum_{u\in\Ok} V\Big(\frac{u}{q^{\mu + \nu}}\Big) \e_{\mu+\nu}(hu) = Q^{\mu+\nu} \sum_{v\in\Ov} {\hat V}(h-q^{\mu+\nu}v), $$
  and so
  $$ \sum_{h\in\Ov/q^{\mu+\nu}} \Big| \sum_{u\in\Ok} V\Big(\frac{u}{q^{\mu + \nu}}\Big) \e_{\mu+\nu}(hu) \Big| \leq Q^{\mu+\nu} \sum_{v\in\Ov} |{\hat V}(v)| \ll \Sigma_V Q^{\mu+\nu}. $$
  Therefore,
  $$ S_I \ll \Sigma_V Q^{\mu+\nu} \sup_{t\in K} \sum_{m\in\cN_\mu} \frac1{N(m)} \sum_{k\in\Ov/m} \Big| {\hat f}_{\mu+\nu}\Big(t - \frac km q^{\mu+\nu}\Big)\Big|, $$
  where
  $$ {\hat f}_\lambda(t) = \frac1{Q^\lambda}\sum_{\ell\in \cN_\ell} f(\ell) \e_\lambda(-t \ell). $$

  Now, by computations identical to~\cite[pp. 2606-2607]{MR}, we write
  \begin{equation}
    {\hat f}_{\mu+\nu}(t) = G_{\kappa, 1}(t) + G_{\kappa, 2}(t),\label{eq:typeI-decomp-fGG}
  \end{equation}
  where
  \begin{align*}
    G_{\kappa, 1}(t) = \sum_{h\in \Ov/q^{\rho_1}} \Big( \frac1{Q^\kappa} {}& \sum_{u\in\cN_\kappa} c_{\kappa, \rho_1}(u, h) \e_{\mu+\nu}(-ut)\Big) \\
    {}& \times \Big( \frac1{Q^{\mu+\nu-\kappa}} \sum_{v\in\cN_{\mu+\nu-\kappa}} f(vq^\kappa) \e_{\mu+\nu-\kappa}(-tv) \e_{\rho_1}(hv) \Big),
  \end{align*}
  $$ c_{\kappa, \rho_1}(u, h) = \frac1{Q^{\rho_1}} \sum_{w\in\cN_{\rho_1}} f_{\kappa+\rho_1}(u+wq^\kappa) \bar{f_{\kappa+\rho_1}(wq^\kappa)} \e_{\rho_1}(-hw), $$
  and
  \begin{align*}
    G_{\kappa, 2}(t) = \frac1{Q^{\mu+\nu}} \sum_{(u, v)\in\cN_\kappa \times \cN_{\mu+\nu-\kappa}} {}& f(vq^\kappa) \e_{\mu+\nu}(-(u+vq^k)t) \\
    {}& \times \big( f(u+vq^\kappa)\bar{f(vq^\kappa)} - f_{\kappa+\rho_1}(u+vq^\kappa)\bar{f_{\kappa+\rho_1}(vq^\kappa)} \big).
  \end{align*}
  By the carry property~\eqref{eq:carry-prop}, we have~$f(u+vq^\kappa)\bar{f(vq^\kappa)} = f_{\kappa+\rho_1}(u+vq^\kappa)\bar{f_{\kappa+\rho_1}(vq^\kappa)}$ unless~$(u, v)$ belongs to a subset~$\cW_{\kappa,\rho_1}$ of~$\cN_\kappa \times \cN_{\mu+\nu-\kappa}$ of size at most
  $$ |\cW_{\kappa,\rho_1}| \ll Q^{\mu+\nu-\eta_1\rho_1}. $$
  If~$\kappa$ satisfies~$(c+1)\kappa \leq c(\mu+\nu)$, then we have
  \begin{equation}\label{eq:typeI-majoG1}
    G_{\kappa, 1}(t) \ll Q^{-\gamma(\mu+\nu-\kappa)} \sum_{h\in \Ov/q^{\rho_1}}  \frac1{Q^\kappa} \Big| \sum_{u\in\cN_\kappa} c_{\kappa, \rho_1}(u, h) \e_{\mu+\nu}(-ut)\Big|
  \end{equation}
  uniformly.

  For all~$m\in\cN_\mu$ and~$k\in\Ov/m$, there is a unique ideal~$\fm$ dividing~$m$, and proper additive character~$\sigma\mod{\fm}$ such that~$\sigma(\xi) = \e(\tc{\xi k/m})$ for all~$\xi\in\Ok$; we write~$(k, m) \sim \sigma$. Note that we have~$N(\fm) \ll Q^\mu$. We rearrange our sum as
  \begin{align*}
    {}& \sum_{m\in\cN_\mu} \frac1{N(m)} \sum_{k\in\Ov/m} \Big| {\hat f}_{\mu+\nu}\Big(t - \frac km q^{\mu+\nu}\Big)\Big| \\
    = {}& \ssum{\fm\text{ ideal} \\ N(\fm)\ll Q^\mu} \sum_{\sigma\mod{\fm}^\ast} \sum_{m\in\cN_\mu} \frac1{N(m)} \ssum{k\in\Ov/m \\ (m, k)\sim \sigma} \Big| {\hat f}_{\mu+\nu}\Big(t - \frac km q^{\mu+\nu}\Big)\Big|.
  \end{align*}
  For each~$\fm$ in this sum, we apply the decomposition~\eqref{eq:typeI-decomp-fGG} with the unique integer~$\kappa_\fm$ for which~$Q^{\kappa_\fm-1} < N(\fm)^2 \leq Q^{\kappa_\fm}$. Hence~$0\leq \kappa_\fm \leq 2\mu + C$ where~$C = 1+\floor{\frac{2\log R_\cF^*}{\log 2}}$. Call~$S_{I, 1}$, resp.~$S_{I,2}$ the contribution of~$G_{\kappa, 1}$, resp.~$G_{\kappa,2}$. The inequality~\eqref{eq:typeI-majoG1} holds if we assume~$\mu \leq \frac c{c+2}\nu - C\frac{c+1}{c+2}$. We obtain, using Cauchy--Schwarz,
  $$ S_{I,1} \ll \Sigma_V Q^{\mu+\nu+\rho_1/2} \sup_{t\in K}\sum_{\kappa=0}^{2\mu + C} \frac{(T_1(\kappa) T_2(\kappa))^{1/2}}{Q^{\kappa+\gamma(\mu+\nu-\kappa)}} , $$
  where
  $$ T_1(\kappa) :=  \sum_{h\in\Ov/q^{\rho_1}} \sum_{Q^{(\kappa-1)/2} < N(\fm) \leq Q^{\kappa/2}} \sum_{\sigma\mod{\fm}^\ast} \Big| \sum_{u\in\cN_\kappa} c_{\kappa, \rho_1}(u,h) \e_{\mu+\nu}(-ut) \sigma(u) \Big|^2.  $$
  $$ T_2(\kappa) :=  \sum_{Q^{(\kappa-1)/2} < N(\fm) \leq Q^{\kappa/2}} \sum_{\sigma\mod{\fm}^\ast} \Big( \sum_{m\in\cN_\mu} \ssum{k\in\Ov/m \\ (m, k)\sim \sigma} \frac1{N(m)} \Big)^2.  $$
  Using Lemma~\ref{lem:addchar-count}, we get~$T_2(\kappa) \ll \mu^{2d}$. On the other hand, by Lemma~\ref{lem:large-sieve}, we have
  $$ T_1(\kappa) \ll \kappa^{d(d-1)} \sum_{h\in\Ov/q^{\rho_1}} Q^{\kappa} \sum_{u\in \cN_\kappa} |c_{\kappa,\rho_1}(u, h)|^2 = \kappa^{d(d-1)}Q^{2\kappa}, $$
  and we conclude
  $$ S_{I,1} \ll \mu^{d^2} \Sigma_V Q^{\mu + \nu + \rho_1/2 - \gamma(\nu-\mu)} $$
  whenever~$\mu \leq \frac c{c+2}\nu-C\frac{c+1}{c+2}$ and~$\rho_1\leq \nu-\mu$.

  Let~$d_\kappa(u, v) := f(u+vq^\kappa)\bar{f(vq^\kappa)} - f_{\kappa+\rho_1}(u+vq^\kappa)\bar{f_{\kappa+\rho_1}(vq^\kappa)}$, which is of modulus at most~$2$ and vanishes unless~$(u, v)\in \cW_{\kappa,\rho_1}$. We have
  $$ |G_{\kappa,2}(t)| \leq Q^{-\mu-\nu} \sum_{v\in\cN_{\mu+\nu-\kappa}} \Big|\sum_{u\in\cN_\kappa} d_\kappa(u, v)\e_{\mu+\nu}(-ut) \Big|, $$
  from which we deduce, similarly as above,
  \begin{align*}
    S_{I, 2} \ll {}& \Sigma_V Q^{\mu+\nu} \sup_t \sum_{m\in\cN_\mu} \frac1{N(m)} \sum_{k\in\Ov/m} \Big| G_{\kappa,2}\Big(t-\frac km q^{\mu+\nu}\Big)\Big| \\
    \ll {}& \mu^d \Sigma_V \sup_t \sum_{\kappa=0}^{2\mu + C} Q^{-\kappa/2} \ssum{\fm\text{ ideal} \\ N(\fm)^2 \leq Q^{\kappa}} \sum_{\sigma\mod{\fm}^\ast} \sum_{v\in\cN_{\mu+\nu-\kappa}} \Big| \sum_{u\in\cN_\kappa} d_\kappa(u, v) \e_{\mu+\nu}(-ut) \sigma(u)\Big| \\
    \ll {}& \mu \Sigma_V Q^{\frac{\mu+\nu}2} \sup_t \sum_{\kappa=0}^{2\mu + C} Q^{-\kappa/2} \times \\
    {}& \qquad \qquad \times \Big( \sum_{v\in\cN_{\mu+\nu-\kappa}} \ssum{\fm\text{ ideal} \\ N(\fm) \leq Q^{\kappa}} \sum_{\sigma\mod{\fm}^\ast}\Big| \sum_{u\in\cN_\kappa} d_\kappa(u, v) \e_{\mu+\nu}(-ut) \sigma(u)\Big|^2 \Big)^{1/2} \\
    \ll {}& \mu^{d+1} \Sigma_V Q^{\mu+\nu - \eta_1\rho_1/2}.
  \end{align*}
  We choose~$\rho_1 = \frac{2}{1+\eta_1} \gamma(\nu-\mu)$. This gives the bound~\eqref{eq:bound-tI} if~$\mu \leq \frac c{c+2}\nu-C\frac{c+1}{c+2}$. If~$C>0$, then replacing~$\nu$ by~$\nu + \floor{C(c+1)/c}+1$ and rescaling~$V$ accordingly yields our result as stated.
\end{proof}

\section{Type II sums}\label{sec:type-ii-sums}

The following estimate is an analogue of Proposition~2 of~\cite{MR}, and is the core of the argument. Given a sequence~$(\alpha_m)_{m\in\O}$ and~$p\geq 1$, we denote by~$\|\alpha\|_p$ the usual~$\ell^p$ norm of~$(\alpha_m)$.

\begin{proposition}\label{prop:type-ii}
  Let~$f:\O \to \C$ satisfy the Carry and Fourier properties~\eqref{eq:carry-prop}--\eqref{eq:Fourier-prop}, for some~$c\geq 20\rh\rb^{-1}$. Let~$2\leq \mu \leq \nu$, $(\alpha_m)_{m\in\cN_\mu}$ and~$(\beta_n)_{n\in\cN_\nu}$ be two sequences of complex numbers, and~$\psi: K \to \R$ be a linear map. Then we have
  \begin{equation}
    S_{II} = \sum_{m\in\cN_\mu} \sum_{n\in\cN_\nu} \alpha_m \beta_n f(mn) \e(\psi(mn)) \ll \mu^{O(1)} \|\alpha\|_2 \|\beta\|_4 Q^{\mu/2+3\nu/4-\delta \gamma(\floor{\frac{\mu}{20\rh\rb^{-1}}})},\label{eq:bound-typeii}
  \end{equation}
  where
  $$ \delta = c\min\{\eta_1^2\eta_2, \eta_1\rb\} $$
  for some absolute constant~$c>0$. The implied constant depends at most on~$(q, \cD)$ and~$\|\cdot\|$.
\end{proposition}

We let~$V_1$ be given as in Lemma~\ref{lem:majo-fourier-trunc}, and as before define~$V, {\hat V} : K \to\C$ by~$V = V_1 \circ \iota^{-1}$ and~${\hat V} = {\hat V}_1 \circ (\iota^\vee)^{-1}$, so that for any~$\lambda\in\N$, we have~$\1_{n\in\cN_\lambda} \leq V(n/q^\lambda)$, and~${\hat V}(\xi) = 0$ for~$\|\xi\|>1$.

\subsection{Preparatory lemma}

As in~\cite{MR}, we will now use the carry property~\eqref{eq:carry-prop} in the context of a multiplicative convolution~$mn = u_1+q^\kappa v$, and so we wish to count the pairs~$(m,n)$ yielding exceptional values of~$v$. The following lemma is the analogue of Lemmas~7 to 9 of~\cite{MR}.
\begin{lemma}\label{lem:carryprop-plus}\ 
  \begin{enumerate}
    \item\label{item:carryprop-1} For any finite set~$\B\subset\Ok$ and~$\mu, \mu', \nu\in\N$ with~$\mu' \geq \mu$, we have
    $$ \card\Big\{(m, n)\in\cN_\mu\times\cN_\nu, \exists u\in\cN_{\mu'}, v\in\B, mn = u + q^{\mu'}v \Big\} \ll \mu^d Q^{\mu'} \card\B. $$
    \item\label{item:carryprop-2} For~$\mu, \nu, \rho \in\N$ with~$\rho \leq 2\nu$, we have
    \begin{align*}
      {}& \card\Big\{(m, n)\in\cN_\mu\times\cN_\nu, \exists k \in \cN_{\mu+\rho}, f(mn+k)\bar{f(mn)} \neq f_{\mu+2\rho}(mn+k)\bar{f_{\mu+2\rho}(mn)} \Big\} \\
      {}& \hspace{25em} \ll \mu^d Q^{\mu+\nu-\eta_1\rho}.
    \end{align*}
    \item\label{item:carryprop-3} Let~$\mu, \nu, \mu_0, \mu_1, \mu_2 \in\N$, and assume that~$\mu_0 \leq \mu_1 \leq \mu \leq \mu_2$. For all~$a,b,c\in\Ok$, the number~$\cE(a,b,c)$ of pairs~$(m, n)\in\cN_\mu\times\cN_\nu$ such that
    \begin{align*}
      {}& f_{\mu_2}(mn+am+bn+c)\bar{f_{\mu_2}(q^{\mu_0}r_{\mu_0,\mu_2}(mn+am+bn+c))} \\
      {}& \qquad \neq f_{\mu_1}(mn+am+bn+c)\bar{f_{\mu_1}(q^{\mu_0}r_{\mu_0,\mu_2}(mn+am+bn+c))}
    \end{align*}
    satisfies
    $$ \cE(a,b,c) \ll \mu_2^{O_q(1)} Q^{\mu+\nu - \eta_1(\mu_1 - \mu_0)}. $$
  \end{enumerate}
\end{lemma}
\begin{proof}
  \begin{enumerate}
    \item Following~\cite[p.2603]{MR}, the quantity we wish to bound is at most
    \begin{align*}
      {}& \sum_{m\in\cN_\mu} \sum_{v\in\B} \ssum{u\in\Ok \\ u\equiv -q^{\mu'}v \mod{m}} V\Big(\frac{u}{q^{\mu'}}\Big) \\
      = {}& Q^{\mu'}\sum_{m\in\cN_\mu} \frac1{N(m)} \sum_{v\in\B}\sum_{k\in\Ov/m} \e\Big(\tc{\frac{q^{\mu'}kv}{m}}\Big)\sum_{u\in\Ok} {\hat V}\Big(q^{\mu'}\Big(\xi + \frac{k}{m}\Big)\Big) \\
      \ll {}& (\card \B) Q^{\mu'} \sum_{m\in\cN_\mu} \frac1{N(m)} \card\Big\{ \xi\in\Ov, \|q^{\mu'}\xi/m\| \leq 1 \Big\}.
    \end{align*}
    The claimed bound then follows from the fact that the condition~$\|q^{\mu'}\xi/m\| \leq 1$ implies~$\|\xi\| \ll \|m/q^{\mu'}\| \ll 1$ (since~$\mu'\geq \mu$), and by Lemma~\ref{lem:addchar-count} with~$\fm=(1)$ (so that the condition~$(k, m)\sim\sigma$ is equivalent to~$k=0$).
    \item Using point~(\ref{item:carryprop-1}) and the carry property~\eqref{eq:carry-prop}, the argument given in~\cite{MR} can be applied with no modifications.
    \item We use the carry property~\eqref{eq:carry-prop} with~$\kappa\gets\mu_0$, $\lambda \gets \mu_2-\mu_0$, $\rho \gets \mu_1-\mu_0$. We deduce that for some set~$\B\subset\cN_{\mu_2-\mu_0}$, with~$\card\B \ll Q^{\mu_2-\mu_0 - \eta_1(\mu_1-\mu_0)}$, we have
    \begin{align*}
  \;    \cE(a, b, c) \leq {}& \sum_{\ell\in\B} \card\{(m, n)\in\cN_\mu\times\cN_\nu, r_{\mu_0,\mu_2}(mn+am+bn+c) = \ell\} \\
      \leq {}& \sum_{\ell\in\B} \sum_{m\in\Ok} V\Big(\frac{m}{q^{\mu}}\Big) \sum_{n\in\Ok} V\Big(\frac{n}{q^{\nu}}\Big) \psi_{\mu_2-\mu_0}\Big(\frac{mn+am+bn+c}{q^{\mu_2}} - \frac{\ell}{q^{\mu_2-\mu_0}}\Big).
    \end{align*}
    We apply Lemma~\ref{lem:Fourier-psi} with~$\tau=0$ and~$\lambda = \mu_2-\mu_0$, and use the triangle inequality along with the bound~\eqref{eq:fourier-bound-coeff} with~$A=d+1$, obtaining \begin{align*}
  \qquad \;    \cE(a,b,c) \ll \frac{\card \B}{Q^{\mu_2-\mu_0}} \sum_{\xi\in\Ov} \frac1{(1+\|q^{-\mu_2+\mu_0}\xi\|)^{d+1}} \sum_{n\in\Ok} V\Big(\frac{n}{q^{\nu}}\Big) \Big|\sum_{m\in\Ok} V\Big(\frac{m}{q^{\mu}}\Big) \e_{\mu_2}(\xi m(n+a)) \Big|.   
    \end{align*}

    The contribution of~$\xi=0$ is~$\ll (\card \B) Q^{\mu+\nu-\mu_2+\mu_0} \ll Q^{\mu+\nu-\eta_1(\mu_1-\mu_0)}$. To bound the remainder, we apply Lemma~\ref{lem:smooth-doublesum} with~$\ft = q^{\mu_2-\mu} \fD_K$ (this gives a slight loss, which is why we isolated~$\xi=0$), getting
    $$ \sum_{n\in\Ok} V\Big(\frac{n}{q^{\nu}}\Big) \Big|\sum_{m\in\Ok} V\Big(\frac{m}{q^{\mu}}\Big) \e_{\mu_2}(\xi m(n+a)) \Big| \ll Q^{\nu+\mu_2-\mu} N(\xi \fD_k + (q^\mu)), $$
    and so, by Lemma~\ref{lem:sum-gcd-h} with~$\ft = q^{\mu_2-\mu_0} \fD_K$,
    \begin{align*}
      \cE(a,b,c) \ll {}& Q^{\mu+\nu-\eta_1(\mu_1-\mu_0)} + Q^{\nu+\mu_2-\mu-\eta_1(\mu_1-\mu_0)} \ssum{\xi\in q^{\mu_0-\mu_2} \Ov \\ \xi\neq 0} \frac{N(\xi q^{\mu_2-\mu_0}\fD_K + (q^\mu))}{(1+\|\xi\|)^{d+1}} \\
      \ll {}& Q^{\mu+\nu-\eta_1(\mu_1-\mu_0)} + \tau(q^{\mu_2-\mu_0} Q^{\mu+\nu - \mu_0 -\eta_1(\mu_1-\mu_0)},
    \end{align*}
    whence the claimed bound.
  \end{enumerate}
\end{proof}

\subsection{Van der Corput step}

The rest of this section is devoted to the proof of Proposition~\ref{prop:type-ii}. Let~$\rho_1, \rho_2, \rho \in \N$, assume that
\begin{equation}
  \rho_2 \leq \rho_1, \quad \rho_1+\rho \leq \frac\mu2, \label{eq:cond-rho2-rho1-rho}
\end{equation}
and define
$$ \mu_0 = \mu - 2(\rho_1+\rho), \quad \mu_1 = \mu - 2\rho_1, \quad \mu_2 = \mu + 2\rho_2. $$
We recall the definition~\eqref{eq:def-Delta}, and we define further, for all~$\lambda\in\N$,
\begin{equation}
  \Delta_\lambda^\ast = \Delta_\lambda \smallsetminus \{0\}.\label{eq:def-Delta-ast}
\end{equation}

The beginning of the argument mirrors closely pp. 2610-2613 of~\cite{MR}, using the van der Corput inequality in the form of Lemma~\eqref{eq:def-Delta}, twice. The computations being the same, we restrict to mentionning the main steps: we obtain, using Lemma~\ref{lem:van-der-corput}, Cauchy--Schwarz's inequality, and \ref{lem:carryprop-plus}.(2),
\begin{align*}
  |S_{II}| \leq {}& \sum_{m\in\Ok} V\Big(\frac{m}{q^\mu}\Big) \Big| \sum_{n\in\cN_\nu} \beta_n f(mn) \e(\psi(mn)) \Big| \\
  \ll {}& \|\alpha\|_2 \|\beta\|_2 Q^{\mu/2+\nu/2-\rho_2/2} \\
  {}& + \|\alpha\|_2 Q^{\nu/2} \Big(Q^{-\rho_2} \ssum{r\in\Delta_{\rho_2}^\ast} \sum_{n\in\cN_\nu} \abs{\beta_{n+r}\beta_n} \Big| \sum_{m\in\Ok} V\Big(\frac{m}{q^\mu}\Big) f(mn+mr)\bar{f(mn)} \e(\psi(mr)) \Big| \Big)^{1/2} \\
  \ll {}&  \mu^{d/4}\|\alpha\|_2 \|\beta\|_4 Q^{\mu/2+3\nu/4-\eta_1\rho_2/4} \\
  {}& + \|\alpha\|_2 Q^{\nu/2} \Big(Q^{-\rho_2} \ssum{r\in\Delta_{\rho_2}^\ast} \sum_{n\in\cN_\nu} \abs{\beta_{n+r}\beta_n}\Big| \sum_{m\in\Ok} V\Big(\frac{m}{q^\mu}\Big) f_{\mu_2}(mn+mr)\bar{f_{\mu_2}(mn)} \e(\psi(mr)) \Big| \Big)^{1/2} \\
  \ll {}&  \mu^{d/4}\|\alpha\|_2 \|\beta\|_4 Q^{\mu/2+3\nu/4-\eta_1\rho_2/4} + \|\alpha\|_2 \|\beta\|_4 Q^{\mu/4+\nu/2} \Big(Q^{-\rho_2-2\rho_1} \ssum{r\in\Delta_{\rho_2}^\ast \\ s\in\Delta_{2\rho_1}^\ast} \abs{S_{II,1}(r,s)} \Big)^{1/4}, \numberthis\label{eq:est-S}
\end{align*}
where
\begin{align*}
  S_{II,1}(r,s) = {}& \sum_{n, m \in \Ok} V\Big(\frac{n}{q^\nu}\Big)V\Big(\frac{m+q^{\mu_1}s}{q^\mu}\Big)V\Big(\frac{m}{q^\mu}\Big) f_{\mu_2}((m+q^{\mu_1}s)(n+r))f_{\mu_2}(mn) \times \\
  {}& \hspace{16em} \times \bar{f_{\mu_2}((m+q^{\mu_1}s)n)f_{\mu_2}(m(n+r))} \\
  = {}& \sum_{n, m \in \Ok} V\Big(\frac{n}{q^\nu}\Big)V_s\Big(\frac{m}{q^\mu}\Big) f_{\mu_1,\mu_2}((m+q^{\mu_1}s)(n+r))f_{\mu_1,\mu_2}(mn) \times \\
  {}& \hspace{15em} \times \bar{f_{\mu_1,\mu_2}((m+q^{\mu_1}s)n)f_{\mu_1,\mu_2}(m(n+r))}.
\end{align*}
Here we let~$V_s(x) = V(x+sq^{\mu_1}) V(x)$, and~$f_{\mu_1, \mu_2} = f_{\mu_2} \bar{f_{\mu_1}}$.
The part (\ref{item:carryprop-3}) of Lemma~\ref{lem:carryprop-plus} allows to replace, in~$S_{II,1}(r,s)$, each term~$f_{\mu_1,\mu_2}(u)$ by~$g(u) = f(q^{\mu_0} r_{\mu_0, \mu_2}(u))$. We deduce
\begin{equation}
  S_{II, 1}(r, s) = S_{II,2}(r, s) + O(\mu_2^{O(1)} Q^{\mu+\nu-2\eta_1\rho}),\label{eq:est-S1}
\end{equation}
where, abbreviating~$u_0 = r_{\mu_0, \mu_2}(mn)$, ~$u_1 = r_{\mu_0, \mu_2}(mn+mr)$,
$$ S_{II,2}(r, s) = \sum_{m\in\Ok} V_s\Big(\frac{m}{q^\mu}\Big) \sum_{n\in\Ok} V\Big(\frac{n}{q^\nu}\Big) g(u_1 + q^{\mu_1-\mu_0}sn + q^{\mu_1-\mu_0}sr) \bar{g}(u_1)\bar{g}(u_0+q^{\mu_1-\mu_0}sn)g(u_0) $$
Let~$\sigma = \mu_2 - \mu_0$. The definition of~$u_0$ and~$u_1$ is inserted as
\begin{align*}
  S_{II,2}(r,s) = \sum_{m\in\Ok}\sum_{n\in\Ok} {}&\sum_{u_0, u_1 \in \Ok / q^\sigma} V_s\Big(\frac m{q^\mu}\Big)V\Big(\frac n{q^\nu}\Big) \psi_\sigma\Big(\frac{mn}{q^{\mu_2}} - \frac{u_0}{q^\sigma}\Big) \psi_\sigma\Big(\frac{mn+mr}{q^{\mu_2}} - \frac{u_1}{q^\sigma}\Big) \times \\
  {}& \times g(u_1 + q^{\mu_1-\mu_0}sn + q^{\mu_1-\mu_0}sr)\bar{g}(u_1)\bar{g}(u_0+q^{\mu_1-\mu_0}sn)g(u_0).
\end{align*}
Let~$\tau\in\N$ be a parameter. We may proceed as in Lemma~2 of~\cite{MR} to deduce
\begin{equation}
  |S_{II,2}(r, s)| \leq |S_4(r,s)| + E_4(r,0) + E_4(0,r) + E_4'(r),\label{eq:UB-S2}
\end{equation}
where
\begin{align*}
  S_4(r,s) = \sum_{m\in\Ok}\sum_{n\in\Ok} {}& \sum_{u_0, u_1\in\Ok/q^\sigma} V_s\Big(\frac m{q^\mu}\Big)V\Big(\frac n{q^\nu}\Big) A_{\sigma,\tau}\Big(\frac{mn}{q^{\mu_2}} - \frac{u_0}{q^\sigma}\Big) A_{\sigma,\tau}\Big(\frac{mn+mr}{q^{\mu_2}} - \frac{u_1}{q^\sigma}\Big) \times \\
  {}& \times g(u_1 + q^{\mu_1-\mu_0}sn + q^{\mu_1-\mu_0}sr)\bar{g}(u_1)\bar{g}(u_0+q^{\mu_1-\mu_0}sn)g(u_0),
\end{align*}
$$ E_4(r, r') = \sum_{m\in\Ok}\sum_{n\in\Ok} \sum_{u_0\in\Ok/q^\sigma} V_s\Big(\frac m{q^\mu}\Big)V\Big(\frac n{q^\nu}\Big) B_{\sigma, \tau}\Big(\frac{mn+mr}{q^{\mu_2}} - \frac{u_0}{q^\sigma}\Big) \sum_{u_1\in\Ok/q^\sigma} \psi_\sigma\Big(\frac{mn+mr'}{q^{\mu_2}} - \frac{u_1}{q^\sigma}\Big), $$
$$ E_4'(r) = \sum_{m\in\Ok}\sum_{n\in\Ok} \sum_{u_0\in\Ok/q^\sigma} V_s\Big(\frac m{q^\mu}\Big)V\Big(\frac n{q^\nu}\Big) B_{\sigma, \tau}\Big(\frac{mn}{q^{\mu_2}} - \frac{u_0}{q^\sigma}\Big) \sum_{u_1\in\Ok/q^\sigma} B_{\sigma,\tau} \Big(\frac{mn+mr}{q^{\mu_2}} - \frac{u_1}{q^\sigma}\Big). $$
At this point, we are in a situation analogous to eq. (64) of~\cite{MR}.

\subsection{Bound on~$E_4(r, r')$}

In~$E_4(r, r')$, the~$u_1$-sum evaluates to~$1$. Therefore,
$$ E_4(r, r') = \sum_{m\in\Ok} \sum_{n\in\Ok} V_s\Big(\frac m{q^\mu}\Big)V\Big(\frac n{q^\nu}\Big) \sum_{u_0\in\Ok/q^\sigma} B_{\sigma, \tau}\Big(\frac{mn+mr}{q^{\mu_2}} - \frac{u}{q^\sigma}\Big). $$
By Lemme~\ref{lem:Fourier-psi},
\begin{align*}
  E_4(r, r') = {}& \sum_{h\in\Ov} b_{\sigma, \tau}(h) \sum_{m\in\Ok} \sum_{n\in\Ok} V_s\Big(\frac m{q^\mu}\Big)V\Big(\frac n{q^\nu}\Big)\sum_{u_0\in\Ok/q^\sigma} \e_{\mu_2}(h(mn+mr))\e_{\sigma}(-hu_0)  \\
  = {}&  Q^\sigma \sum_{h\in\Ov} b_{\sigma, \tau}(h q^\sigma) \sum_{m\in\Ok} \sum_{n\in\Ok} V_s\Big(\frac m{q^\mu}\Big)V\Big(\frac n{q^\nu}\Big) \e_{\mu_0}(h(mn+mr))
\end{align*}
by orthogonality. We apply Lemma~\ref{lem:smooth-doublesum} with~$\ft = \fD_K$, using the fact that~$V_s \ll V$, getting
\begin{align*}
  E_4(r, r') \ll {}& Q^{\mu+\nu+\sigma-\mu_0} \sum_{h\in\Ov} |b_{\sigma,\tau}(hq^\sigma)| N(h\fD_K + (q^{\mu_0})) \\
  \ll{}& Q^{\mu+\nu-\mu_0-\eta_2\tau} \sum_{h\in\Ov} \frac{N(h\fD_K + (q^{\mu_0}))}{(1+\|h/q^\tau\|)^{d+1}}.
\end{align*}
Here we may apply Lemma~\ref{lem:sum-gcd-h} after changing~$h$ to~$q^\tau h$, with~$\ft = q^\tau \fD_K$. Along with partial summation, we obtain
\begin{equation}\label{eq:UB-E4}
  E_4(r, r') \ll \mu^{O(1)} Q^{\mu+\nu} \big\{ Q^{-\eta_2 \tau} + Q^{(1-\eta_2)\tau + 2(\rho+\rho_1) - \mu}\big\}.
\end{equation}

\subsection{Bound on~$E_4'(r)$}

Similarly as before, we use Lemma~\ref{lem:Fourier-psi} to expand~$B_{\sigma, \tau}$, and we execute the~$u_j$-sums, which selects frequencies which are multiples of~$q^\sigma$. We get
$$ E'_4(r) = Q^{2\sigma} \sum_{h_0, h_1 \in \Ov} b_{\sigma,\tau}(h_0 q^\sigma) b_{\sigma,\tau}(h_1 q^\sigma) \sum_{m\in\Ok} \sum_{n\in\Ok} V_s\Big(\frac m{q^\mu}\Big) V\Big(\frac n{q^\nu}\Big) \e_{\mu_0}(mn(h_0+h_1) + mrh_1). $$
The contribution of the diagonal contribution~$h_0+h_1 = 0$ is bounded by
$$ \ll Q^{2\sigma+\mu+\nu} \sum_{h\in\Ov} |b_{\sigma,\tau}(h q^\sigma)|^2 \ll Q^{\mu+\nu-\eta_2\tau}. $$
Therefore, using again~$V_s \ll V$, it suffices to obtain a non-trivial bound for
$$ T_4 := Q^{2\sigma} \ssum{h_0, h_1 \in \Ov \\ h_0 + h_1 \neq 0} |b_{\sigma,\tau}(h_0 q^\sigma) b_{\sigma,\tau}(h_1 q^\sigma)| \sum_{m\in\Ok} V\Big(\frac m{q^\mu}\Big) \Big| \sum_{n\in\Ok} V\Big(\frac n{q^\nu}\Big) \e_{\mu_0}(mn(h_0+h_1))\Big|. $$
The~$(m, n)$-sums are bounded using Lemma~\ref{lem:smooth-doublesum} with~$\ft = \fD_K$, which yields
\begin{align*}
  T_4 \ll {}& Q^{\mu+\nu+2\sigma - \mu_0} \ssum{h_0, h_1 \in \Ov \\ h_0 + h_1 \neq 0} |b_{\sigma,\tau}(h_0 q^\sigma) b_{\sigma,\tau}(h_1 q^\sigma)| N((h_0+h_1)\fD_K + (q^{\mu_0})) \\
  \ll {}& Q^{\mu+\nu - \mu_0} \ssum{h_0, h_1 \in \Ov \\ h_0 + h_1 \neq 0} \frac{N((h_0+h_1)\fD_K + (q^{\mu_0}))}{(1+\|h_0/q^\tau\|)^{2d+1}(1+\|h_1/q^\tau\|)^{d+1}} \\
  \ll {}& \mu^{O(1)} Q^{\mu+\nu + 2\tau - \mu_0}
\end{align*}
by Lemma~\ref{lem:sum-gcd-hh} and partial summation. We conclude that
\begin{equation}
  E'_4(r) \ll Q^{\mu + \nu}\big\{ Q^{-\eta_2\tau} + \mu^{O(1)} Q^{2\tau-\mu_0}\big\}.\label{eq:UB-E4'}
\end{equation}

\subsection{Bound on~$S_4$}

In~$S_4(r, s)$, we expand~$A_{\sigma,\tau}$ in Fourier series, and we sort according to the values of~$u_3 = u_1 + q^{\mu_1-\mu_0}s(n+r)\mod{q^\sigma}$ and~$u_2 = u_0 + q^{\mu_1-\mu_0}sn\mod{q^\sigma}$. We get
\begin{equation}\label{eq:S4-fourier}
  S_4(r, s) = Q^{-2\sigma} \ssum{{\bm h} = (h_0, h_1, h_2, h_3) \\ h_0, h_1\in \Ov \\ h_2, h_3 \in \Ov/q^\sigma} a_{\sigma, \tau}(h_0) a_{\sigma,\tau}(h_1) \e_{\mu_2-\mu_1}(h_3sr) U({\bm h}) W({\bm h}),
\end{equation}
where
\begin{align*}
  {}& U({\bm h}) := \sum_{m, n\in \Ok} V_s\Big(\frac m{q^\mu}\Big)V\Big(\frac n{q^\nu}\Big) \e_{\mu_2}(mn(h_0+h_1) + mrh_1 + q^{\mu_1}ns(h_2+h_3)), \\
  {}& W({\bm h}) := \ssum{u_0, u_1, u_2, u_3 \\ u_j \in \Ok/q^\sigma} g(u_0){\bar g}(u_1){\bar g}(u_2)g(u_3) \e_\sigma(u_0(h_2-h_0)+u_1(h_3-h_1)-u_2h_2 - u_3h_3).
\end{align*}

With the notation
\begin{equation}
  \gh(h) := Q^{-\sigma} \sum_{u\in\Ok/q^\sigma} g(u)\e_\sigma(-uh),\label{eq:def-gh}
\end{equation}
we have~$W({\bm h}) = Q^{4\sigma} \gh(h_0-h_2){\bar \gh}(h_3-h_1) {\bar \gh}(-h_2) \gh(h_3)$.

\subsubsection{Off-diagonal terms}

First we consider the contribution~$S_4''(r, s)$ to the sum~\eqref{eq:S4-fourier} of those indices which satisfy~$h_0 + h_1 \neq 0$. By Lemma~\ref{lem:smooth-doublesum} with~$q\gets q^\mu$, $\alpha\gets q^{-2\rho}(h_0+h_1)$ and~$\ft\gets (q^{2\rho_2})$, we obtain
$$ U({\bm h}) \ll Q^{\nu+2\rho_2} N((h_0+h_1)\fD_K + (q^\mu)). $$
On the other hand, arguing as in p.~2621 of~\cite{MR} by Cauchy-Schwarz and Parseval's identity, for all~$h_0, h_1 \in \Ov$ we have
$$ \sum_{h_2, h_3 \in \Ov/q^\sigma} |W({\bm h})| \leq Q^{4\sigma}. $$
Therefore, we obtain
\begin{align*}
  S_4''(r, s) \ll {}& Q^{\nu+2\sigma+2\rho_2} \ssum{h_0, h_1 \in \Ov \\ h_0+h_1 \neq 0} |a_{\sigma,\tau}(h_0)a_{\sigma,\tau}(h_1)| N((h_0+h_1)\fD_K + (q^\mu)) \\
  \ll {}& Q^{\nu+2\rho_2} \ssum{h_0, h_1 \in \Ov \\ h_0+h_1 \neq 0} \frac{N((h_0+h_1)\fD_K + (q^\mu))}{((1+\|\frac{h_0}{q^{\sigma+\tau}}\|) (1+\|\frac{h_1}{q^{\sigma+\tau}}\|))^{2d+1}} \\
  = {}& Q^{\nu+2\rho_2} \ssum{h_0, h_1 \in (q^{\sigma+\tau} \fD_K)^{-1} \\ h_0+h_1 \neq 0} \frac{N((h_0+h_1)q^{\sigma+\tau} \fD_K + (q^\mu))}{((1+\|h_0\|) (1+\|h_1\|))^{2d+1}} \\
  \ll {}& \mu^{O(1)} Q^{\nu+2\rho_2+2(\sigma+\tau)-\mu} \numberthis\label{eq:UB-S4''}
\end{align*}
by Lemma~\ref{lem:smooth-doublesum} with~$\ft = q^{\sigma+\tau} \fD_K$ and partial summation.

\subsubsection{Diagonal terms}

Note that~$A_{\sigma,\tau}(\xi) \in \R$, so that~$a_{\sigma,\tau}(-\xi) = \bar{a_{\sigma,\tau}(\xi)}$. Let~$S_4'(r,s)$ denote the contribution to~$S_4(r,s)$ coming from indices~$h_0+h_1=0$, so that
\begin{equation}
  S_4(r,s) = S_4'(r,s) + S_4''(r,s).\label{eq:S4-decomp}
\end{equation}
We define
$$ U_1(h;r,s) := \sum_{m\in\Ok} V_s\Big(\frac{m}{q^\mu}\Big) \e_{\mu_2}(-mrh), \quad U_2(h') := \sum_{n\in\Ok} V\Big(\frac{n}{q^\nu}\Big) \e_{\mu_2-\mu_1}(nsh'), $$
so that
\begin{align*}
  S_4'(r,s) =  Q^{2\sigma} \ssum{h\in\Ov \\ h_2,h_3 \in \Ov/q^\sigma} |a_{\sigma,\tau}(h)|^2 \e_{\mu_2-\mu_1}(h_3sr) {}& U_1(h;r,s)U_2(h_2+h_3) \times \\
  {}& \times \gh(h-h_2){\bar \gh}(h_3+h) {\bar \gh}(-h_2) \gh(h_3)
\end{align*}
and consequently
\begin{align*}
  |S_4'(r,s)| \leq Q^{2\sigma} \ssum{h\in\Ov \\ h' \in \Ov/q^\sigma} |a_{\sigma,\tau}(h)|^2 {}& |U_1(h;r,s)| |U_2(h')| \times \\
  {}& \times \sum_{h_3\in\Ov/q^\sigma} |\gh(h-h'+h_3)\gh(h_3+h)\gh(-h'+h_3) \gh(h_3)|.
\end{align*}
Note that by Cauchy--Schwarz,
$$ \sum_{h_3\in\Ov/q^\sigma} |\gh(h-h'+h_3)\gh(h_3+h)\gh(-h'+h_3) \gh(h_3)| \leq W(h), $$
where 
$$ W(h) = \sum_{h_3\in\Ov / q^\sigma} |\gh(h_3+h)\gh(h_3)|^2. $$
We note for further reference that, using~$\abs{\gh}\leq 1$ and Parseval's identity,
\begin{equation}\label{eq:trivial-W}
  \abs{W(h)} \leq 1.
\end{equation}
Assume
\begin{equation}
  \nu\geq \mu_2 - \mu_1 = 2(\rho_1 + \rho_2).\label{eq:cond-nu}
\end{equation}
We have
\begin{align*}
  \sum_{h'\in\Ov/q^\sigma} |U_2(h')| {}& \leq Q^\nu \sum_{h'\in\Ov/q^\sigma} \sum_{\xi\in\Ov} |{\hat V}(q^{\nu-2(\rho_1+\rho_2)}(sh'+q^{2(\rho_1+\rho_2)}\xi))| \\
  {}& = Q^{\nu+2\rho'} \sum_{\xi\in\Ov} |{\hat V}(q^{\nu-2(\rho_1+\rho_2)}\xi)| \sum_{h'\in\Ov/q^{2(\rho_1+\rho_2)}} \1(sh' - \xi \in q^{2(\rho_1+\rho_2)}\Ov).
\end{align*}
By Lemma~\ref{lem:majo-fourier-trunc}, the only~$\xi$ contributing to the sum is~$\xi=0$. We bound the~$h'$-sum as in Lemma~\ref{lem:solution-gcd}: the condition~$sh' \in q^{2(\rho_1+\rho_2)}\Ov$ means~$h'\in q^{2(\rho_1+\rho_2)}\fd^{-1}\Ov$, where~$\fd = (s) + (q^{2(\rho_1+\rho_2)})$. Since~$|q^{2(\rho_1+\rho_2)}\fd^{-1}\Ov / (q^{2(\rho_1+\rho_2)})| = |\Ok / \fd \fD_K| = N(\fd \fD_K) \ll N(\fd)$, we find
$$ \sum_{h'\in\Ov/q^\sigma} |U_2(h')| \leq Q^{\nu+2\rho'} N((s) + (q^{2(\rho_1+\rho_2)})), $$
and so
$$ S_4'(r,s) \ll Q^{\nu+2\sigma+2\rho'} N((s)+(q^{2(\rho_1+\rho_2)})) \sum_{h\in\Ov} |a_{\sigma,\tau}(h)|^2 |U_1(h;r,s)| W(h), $$
where

We now execute the sum over~$s\in\Delta_{2\rho_1}^\ast$. Define
$$ U_1(h;r) = \sup_{s\in\Delta_{2\rho_1}^\ast}\abs{U_1(h;r,s)}. $$
Then, with~$C = 2R_\cF^+$ (where we recall the definition~\eqref{eq:F-subsup}), we have
\begin{align*}
  \frac1{Q^{2\rho_1}} \ssum{s\in \Delta_{2\rho_1}^\ast} N((s)+(q^{2(\rho_1+\rho_2)}))
  \leq {}& \frac1{Q^{2\rho_1}} \ssum{\fd \mid (q^{2(\rho_1+\rho_2)})} N(\fd) \card\Big\{s\in q^{-2\rho_1}\fd, 0<\|s\|\leq C \Big\} \\
  \ll  {}& \tau(q^{2(\rho_1+\rho_2)})
\end{align*}
by Lemma~\ref{lem:lattice-count}, and the last quantity is~$O(\rho_1^{O(1)})$. We deduce
$$ \frac1{Q^{2\rho_1}} \sum_{s\in \Delta_{2\rho_1}^\ast} |S_4'(r,s)| \ll \mu^{O(1)} Q^{\nu+2\sigma+2\rho'}\sum_{h\in\Ov} |a_{\sigma,\tau}(h)|^2 U_1(h;r) W(h). $$
Define~$\tau = \rho_2(2+\rb^{-1})$, $\tau' := \tau+\sigma + \floor{\mu\ee}$ for some parameter~$\ee\in(0, 1]$ to be chosen later, and impose the condition
$$ \rh\tau' \leq \tfrac12 \rb \mu. $$

We will prove the three bounds
\begin{align}
  {}& \ssum{h\in\Ov \\ \|h/q^{\tau'}\|>1} |a_{\sigma,\tau}(h)|^2 U_1(h;r) W(h) \ll_{\ee} Q^{-10\mu}, {}&  \label{eq:typeii-lastbound-1} \\ 
  {}& \frac1{Q^{\rho_2}} \sum_{r\in\Delta_{\rho_2}^\ast} U_1(h;r)  \ll_A \mu^{O(1)} \big( Q^{\mu-A\rb \rho_2} + Q^{-10\mu} \big), {}& \text{if } \|\tfrac{h}{q^{\tau'}}\|\leq 1, \|\tfrac{h}{q^\tau}\|>1, \label{eq:typeii-lastbound-2} \\
  {}& \ssum{h\in\Ov \\ \|h/q^\tau\|\leq 1} W(h) \ll Q^{2\rho'-\eta''\gamma(\sigma-\tau)} + Q^{2\rho'-\eta_1\tau}. {}& \label{eq:typeii-lastbound-3} 
\end{align}
Along with the bounds~\eqref{eq:trivial-W} and~$\abs{a_{\sigma,\tau}(h)} \ll Q^{-\sigma}$, this will yield
\begin{equation}
  \begin{aligned}
    {}& \frac1{Q^{\rho_1+2\rho_2}} \ssum{r\in\Delta_{\rho_2}^\ast \\ s\in \Delta_{2\rho_1}^\ast} |S_4'(r,s)| \\
    {}& \hspace{2em} \ll_A \mu^{O(A)} Q^{\mu+\nu+2\rho'}\big\{ Q^{-10\mu} + Q^{\tau'-A\rb\rho_2} + Q^{2\rho'-\eta''\gamma(\sigma-\tau)} + Q^{2\rho'-\eta_1\tau}\big\}.
  \end{aligned}\label{eq:UB-avg-S4'}
\end{equation}

\subsubsection{Large~$h$}

First, for~$\|h/q^{\tau'}\|>1$, we have
$$ \|q^{-\sigma-\tau}h\| \gg \| h/q^{\tau'} \| Q^{\rb \mu \ee} \mu^{1-d} $$
by Lemma~\ref{lem:norms} and our definition~\eqref{eq:def-q-rh}. By using Lemma~\ref{lem:Fourier-psi}, we have for any~$A\geq 1$,
$$ |a_{\sigma,\tau}(h)|^2 \ll_A \frac{1}{Q^{2\sigma+A\rb\mu\ee}\|h/q^{\tau'}\|^A}. $$
We deduce, by~\eqref{eq:trivial-W},
\begin{align*}
  \ssum{h\in\Ov \\ \|h/q^{\tau'}\|>1} |a_{\sigma,\tau}(h)|^2 U_1(h;r) W(h) \ll_A {}& Q^{-2\sigma-A\rb\mu\ee} \sum_{h\in q^{-\tau'}\Ov} (1+\|h\|)^{-d-1} \\
  \ll_\ee {}& Q^{-10\mu}
\end{align*}
by assuming~$\rho\leq\mu$ and by picking~$A$ large enough in terms of~$\ee$. This proves~\eqref{eq:typeii-lastbound-1}.

\subsubsection{Middle-sized $h$}

Assume that~$\|h/q^{\tau'}\|\leq 1$ and~$\|h/q^\tau\|>1$.  For all~$r\in \Delta_{\rho_2}^\ast$, we have
$$ \Big\|\frac{rh}{q^{2\rho_2}}\Big\| \gg \|h/q^\tau\| \|q^{2\rho_2-\tau} r^{-1}\|^{-1} $$
by the triangle inequality, while
$$ \|q^{2\rho_2-\tau} r^{-1}\| \ll \|q^{\rho_2-\tau}\|  Q^{\rho_2} N(r)^{-1} \|r/q^{\rho_2}\|^{d-1} \ll \|q^{\rho_2-\tau}\| Q^{\rho_2} \ll \rho_2^{1-d}Q^{-\rb\rho_2}, $$
by Lemma~\ref{lem:norms} and since~$r\in\Delta_{\rho_2} \subset\Ok$.
We conclude that for~$h\in\cN_{\tau'}\smallsetminus\cN_\tau$, for all~$r\in \Delta_{\rho_2}^\ast$, we have
\begin{equation}\label{eq:troncation-mino-rh}
  \|q^{-2\rho_2}rh\| \gg \rho_2^{O(1)} Q^{\rb\rho_2}.
\end{equation}
On the other hand, we have
\begin{align*}
  \frac1{Q^{\rho_2}}\sum_{r\in \Delta_{\rho_2}^\ast} U_1(h;r)
  = {}& \frac1{Q^{\rho}} \sum_{r\in \Delta_{\rho_2}^\ast} \sup_{s\in\Delta_{2\rho_1}^\ast} \Big|\sum_{m\in\Ok} V_s\Big(\frac{m}{q^\mu}\Big) \e_{\mu_2}(mrh)\Big| \\
  = {}& Q^{\mu-\rho} \sum_{r\in \Delta_{\rho_2}^\ast} \sup_{s\in\Delta_{2\rho_1}^\ast}\Big|\sum_{\xi\in\Ov} {\hat V_s}\Big(q^\mu \xi + \frac{rh}{q^{2\rho_2}}\Big)\Big|.
\end{align*}
Here have~$\|\frac{rh}{q^{2\rho_2}}\| \ll \|h\| \leq Q^{\rh \tau'}$, while for~$\xi\neq 0$,~$\|q^\mu\xi\| \gg \mu^{O(1)}Q^{\rb\mu}$. Moreover, since~$V_s(x)=V(x)V(x+q^{\mu_1-\mu}s)$, the derivatives of~$V_s$ are bounded uniformly in~$s$, and so~$\abs{{\hat V_s}(x)} \ll_A (1+\|x\|)^{-A}$ for all~$x\in K$ and~$A\geq 0$. Assuming
\begin{equation}
  \rh\tau' \leq \tfrac12 \rb \mu,\label{eq:cond-taup-mu}
\end{equation}
we obtain that for~$\mu$ large enough, either~$\xi = 0$ or
$$ \Big\| q^\mu \xi + \frac{rh}{q^{2\rho_2}} \Big\| \geq  \frac{\|q^\mu\xi\|}2 \gg \mu^{O(1)}Q^{\rb\mu} \|\xi\|. $$
Summarizing the above, we conclude that for~$\|h/q^{\tau'}\|\leq 1$ and~$\|h/q^\tau\|>1$,
$$ \frac1{Q^{\rho_2}}\sum_{r\in \Delta_{\rho_2}} U_1(h;r) \ll_A Q^\mu \mu^{O(A)} \big( Q^{-A\rb \rho_2} + Q^{-A\rb \mu} \big). $$
for any fixed~$A\geq 0$. This yields~\eqref{eq:typeii-lastbound-2}.

\subsubsection{Small~$h$}

Finally, we focus on the case~$\|h/q^\tau\|\leq 1$. In this range, we use the estimate~$|a_{\gamma,\tau}(h)| \ll Q^{-\sigma}$ from Lemma~\ref{lem:Fourier-psi}, and the trivial bound~$U_1(h;r)\ll Q^\mu$. We get
$$ \ssum{h\in\Ov \\ h\in\cN_\tau} |a_{\sigma,\tau}(h)|^2 U_1(h;r) W(h) \ll Q^{\mu-2\sigma} \ssum{h\in\Ov \\ \|h/q^\tau\|\leq 1} W(h). $$
Assuming that
$$ \mu_0 \leq c(\sigma - \tau), $$
Lemma~\ref{prop:Fourier-middle} applies, and yields
$$ \ssum{h\in\Ov \\ \|h/q^\tau\|\leq 1} W(h) \ll Q^{2\rho'}\big(Q^{-\eta''\gamma(\sigma-\tau)} + Q^{-\eta_1\tau}\big) \sum_{h_3\in\Ov/q^{\rho_3}} \abs{\gh(h_3)}^2. $$
The sum over~$h_3$ evaluates to~$1$ by Parseval's identity, and we obtain~\eqref{eq:typeii-lastbound-3}.

\subsection{Optimization}

Grouping successively the bounds~\eqref{eq:UB-S4''}, \eqref{eq:S4-decomp}, \eqref{eq:UB-avg-S4'}, and~\eqref{eq:est-S}--\eqref{eq:UB-E4'} yields
$$ S_{II} = \sum_{m\in\cN_\mu} \sum_{n\in\cN_\nu} \alpha_m \beta_n f(mn) \ll_\ee \mu^{O(1)} \|\alpha\|_2 \|\beta\|_4 Q^{\mu/2+3\nu/4-\delta/4}, $$
where
$$ \delta = \min\Big\{\eta_1\rho_2, 2\eta_1\rho', \eta_2\tau, \mu-2(\tau+\rho'+\rho_1), \eta''\gamma(\sigma-\tau)-4\rho', \eta_1\tau-4\rho'\Big\}, $$
with~$\tau = (2+\rb^{-1})\rho_2$ and~$\tau'=\tau+2(\rho'+\rho_1+\rho_2)+\floor{\mu\ee}$, under the conditions:
$$ \rho_2 \leq \rho_1 \leq \mu, \quad 2(\rho'+\rho_1+\rho_2) \leq \mu, \quad \mu + c\rb^{-1}\rho_2 \leq 2(c+1)\rho' + (2c+1)\rho_1, \quad \rh\tau' \leq \tfrac12 \rb \mu. $$
Let~$K = 20\rh\rb^{-1}$, so that by hypothesis~$c \geq K$. Then with the choice
$$ \rho_1 = \frac\mu{K}+O(1), \quad \rho' = \frac{\eta''}{8}\gamma\Big(\floor{\frac\mu{K}}\Big) + O(1), \quad \rho_2 = \frac{\rb\mu}K + O(1), $$
the claimed result follows.

\section{Sums over prime elements: proof of Theorem~\ref{thm:meanvalue-prime}}\label{sec:sums-over-prime}

In this section we assume that~$\Ok$ is principal. Our goal is to use Propositions~\ref{prop:type-i} and~\ref{prop:type-ii} to estimate mean values over prime elements of~$\Ok$, and prove Theorem~\ref{thm:meanvalue-prime}.

\subsection{Combinatorial identity}

In this section we express the characteristic function of prime elements into convolutions for which Propositions~\ref{prop:type-i} and \ref{prop:type-ii} apply. The methods that have been developed to perform this step has a long history, since Vinogradov's work~\cite{Vin37}. We refer to~\cite{Ramare2013} for an account and references. In~\cite{MauduitRivat2010,MR}, this rôle is played by the combinatorial identity of Vaughan~\cite{Vaughan1980}; see~\cite{Hinz1988} for a number field analogue.

One advantage of Vaughan's identity as it is cast in~\cite{MauduitRivat2010} is the absence of divisor functions in the upper-bound. One inconvenient, as with all methods which pass through the von Mangoldt function, is the necessity to use partial summation to detect the size of~$\log N(n)$. Here we take the opportunity to proceed along a slightly different argument (see~\cite[Theorem~3.3]{DT-2018}), with the benefit that we avoid completely partial summation.

For a non-zero ideal~$\fn\subset \Ok$, let
$$ P^+(\fn) = \max_{\fp\mid\fn} N(\fp), \qquad P^-(\fn) = \min_{\fp\mid\fn} N(\fp), $$
where~$\fp$ denotes a prime ideal, and by convention~$P^+(\Ok) = 1$ and~$P^-(\Ok) = +\infty$.

\begin{lemma}\label{lem:comb-ident}
  Let~$X\geq 2$, and~$(g(\fn))_{\fn\neq 0}$ be complex numbers with~$g(\fn)=0$ if~$N(\fn)>X$. Then
  $$ \Big|\sum_{\fp} g(\fp) \Big| \ll \|g\|_\infty X^{\frac12} + \sum_{N(\fm)\leq X^{\frac14}} \Big| \sum_\fn g(\fm\fn) \Big| + (\log X) \sup_{(\alpha, \beta)} \Big| \underset{X^{\frac14} < N(\fm) \leq X^{\frac34}}{\sum_\fm\sum_\fn} \alpha_\fm \beta_\fn g(\fm\fn) \Big|, $$
  where the supremum is over all sequences~$(\alpha_\fm), (\beta_\fn)$ satisfying~$\abs{\alpha_\fm} \leq 1$ and~$\abs{\beta_\fn} \leq \tau(\fn)$.
\end{lemma}

\begin{proof}
  Discarding those prime ideals of norm at most~$X^{\frac12}$, the sum we wish to evaluate is
  $$ O(\|g\|_\infty X^{\frac12}) + \ssum{\fn \\ P^-(\fn)>X^{\frac12}} g(\fn). $$
  The condition is detected by Möbius inversion,
  $$ \ssum{\fn \\ P^-(\fn)>X^{\frac12}} g(\fn) = \sum_{P^+(\fm)\leq X^{\frac12}} \mu(\fm) \sum_{\fn} g(\fm\fn). $$
  The contribution of those~$\fm$ with~$N(\fm) \leq X^{\frac14}$ yields the first term. Suppose that~$N(\fm)> X^{\frac14}$ and~$\fm$ squarefree. Let~$\prec$ be any ordering of the prime ideals which respects the norm, \emph{i.e.} $N(\fp_1)<N(\fp_2)$ implies~$\fp_1 \prec \fp_2$. Enumerating the prime ideals with respect to this ordering induces a bijection~$\phi:\{\fp\text{ prime}\} \to\N_{>0}$, satisfying~$\phi(\fp) \ll N(\fp)$. Let~$\fp^+(\fn)$ (resp.~$\fp^-(\fn)$) denote the maximal (resp. minimal) prime divisor of~$\fn\neq (1)$ with respect to~$\phi$. Write
  $$ \fm = \fp_1 \dotsb \fp_k, $$
  where~$\phi(\fp_j)<\phi(\fp_{j+1})$. Then there is a minimal index~$j_\fm$ for which, letting~$\fm_1 = \fp_1 \dotsb \fp_{j_\fm}$, we have~$N(\fm_1)>X^{\frac14}$. The ideal~$\fm_1$ is characterized by the conditions
  $$ \fm_1\mid \fm, \quad X^{\frac14}<N(\fm_1) \leq X^{\frac14}P^+(\fm_1), \quad \phi(\fp^+(\fm_1)) < \phi(\fp^-(\fm/\fm_1)). $$
  We deduce
  $$ \ssum{N(\fm) > X^{\frac14} \\ P^+(\fm)\leq X^{\frac12}} \mu(\fm) \sum_{\fn} g(\fm\fn) = \underset{\substack{X^{\frac14} < N(\fm_1) \leq X^{\frac14} P^+(\fm_1) \\ P^+(\fm_2)\leq X^{\frac12} \\ \phi(\fp^+(\fm_1)) < \phi(\fp^-(\fm_2))}}{\sum\sum} \sum_{\fn} \mu(\fm_1)\mu(\fm_2) g(\fm_1\fm_2\fn). $$
  The condition~$\phi(\fp^+(\fm_1)) < \phi(\fp^-(\fm_2))$ is detected by means of Lemma 13.11 of~\cite{IwaniecKowalski2004}, so that setting
  \begin{align*}
    \alpha_\fm(t) = {}& \phi(\fp^+(\fm))^{it} \mu(\fm), \\
    \beta_\fn(t) = {}& \ssum{\fd\mid\fn, \fd\neq (1) \\ P^+(\fd)\leq X^{\frac12}} \phi(\fp^-(\fd))^{-it} \mu(\fd),
  \end{align*}
  we have
  $$ \Big|\ssum{X^{\frac14} < N(\fm) \\ P^+(\fm)\leq X^{\frac12}} \mu(\fm) \sum_{\fn} g(\fm\fn)\Big| \ll (\log X) \sup_{t\in\R} \Big| \underset{\substack{P^+(\fm)\leq X^{\frac12} \\ X^{\frac14} < N(\fm) \leq X^{\frac14} P^+(\fm)}}{\sum_\fm\sum_\fn} \alpha_\fm(t) \beta_\fn(t) g(\fm\fn) \Big|, $$
  as claimed.
\end{proof}

\subsection{Quotient by units}

When translating sums over ideals (coming from the combinatorial identities) to sums over~$\Ok$, we will use the following partition of unity, inspired from~\cite[Lemma~4.2]{Toth2000}, to account for the quotient by units.

\begin{lemma}\label{lem:units}
  There exists a smooth, homogeneous function~$\Phi_0 : \R^d\smallsetminus\{0\} \to [0, 1]$ such that, letting~$\Phi = \Phi_0\circ\iota^{-1}$, we have
  \begin{equation}
    \sum_{\ee\in\Ok^*} \Phi(n\ee) = 1 \qquad (n\in\Ok\smallsetminus\{0\}),\label{eq:partunity-units}
  \end{equation}
  and for any given~$n$, there are only finitely many non-zero terms in the sum. Moreover, for all~$x\in K^\ast$ with~$\Phi(x)\neq 0$, and all~$\pi\in G_K$, we have~$|x^\pi|\asymp N(x)^{1/d}$ with an implied constant depending only on~$K$.

  In particular, if~$\Ok$ is principal, then for any function~$g: \{\fn\neq 0\} \to \C$ of finite support and any~$\ee\in\Ok^\ast$, we have
  \begin{equation}
    \sum_{\fn\neq 0} g(\fn) = \sum_{n\in\Ok\smallsetminus\{0\}} \Phi(n\ee) g((n)).\label{eq:partunity-sumfn}
  \end{equation}
\end{lemma}

\begin{proof}\ 
  Let~$r$ be the rank of the free part of~$\O^\ast$, and~$\ee_1, \dotsc, \ee_r$ be any fixed basis~\cite[Theorem~I.7.3]{Neukirch1999}, so
  $$ \O^\ast = \{\omega \ee_1^{n_1} \dotsb \ee_r^{n_r}, \omega\in \Omega, n_j\in \Z\} $$
  where~$\Omega$ are the roots of unity in~$\O$. As in~\cite[p.55]{Murty2007}, we let~$\Psi:\R^d \to \R^r$ be the map defined by~$\Psi(x) = (\psi_1(x), \dotsc, \psi_r(x))$, where
  $$ \log(\abs{\iota(x)^\pi / N(\iota(x))}) = \sum_{j=1}^r \psi_j(x) \log\abs{\ee_j^\pi} $$
  for all~$\pi \in G_K$. Then for~$\lambda \in \R^\ast$ and~$1\leq j \leq r$, $\psi_j$ is smooth and~$\psi_j(\lambda x) = \psi_j(x)$. Let a smooth function~$w:\R \to [0, 1]$ with~$\supp w \subset[-1, 1]$ be a partition of unity as
  \begin{equation}
    \sum_{n\in\Z} w(x + n) = 1 \qquad (x\in \R),\label{eq:partunity-parts}
  \end{equation}
  and define, for all~$x\in K^\ast$, $\Phi(x) := w(\psi_1(x)) \dotsb w(\psi_r(x))$. Then the function~$\Phi$ is well-defined, smooth and homogeneous on~$\R^d/\R^*$, and the property~\eqref{eq:partunity-units} follows by~$r$ applications of~\eqref{eq:partunity-parts}.

  To prove~\eqref{eq:partunity-sumfn}, let~$\fn = (n_0) \neq 0$ be an integral ideal, with~$n_0\in\O$. Then
  $$ \ssum{n\in\O \\ (n) = \fn} \Phi(n\ee) = \sum_{\ee'\in\O^\ast} \Phi(n_0 \ee \ee') = 1 $$
  by~\eqref{eq:partunity-units}.
\end{proof}

\subsection{Proof of Theorem~\ref{thm:meanvalue-prime}}

\subsubsection{Preparations}

We borrow the notation~$\chi_\tau$ from Lemma~\ref{lem:Fourier-psi} (see~\eqref{eq:def-chitau}). Let~$W:\R_+\to\R_+$ be a smooth function defining a partition of unity along powers of~$Q$ in the sense that for all~$x\geq 0$,
$$ W(x) \leq \1_{[1/(2Q), 1]}(x), \qquad \sum_{k\in\Z} W(Q^{-k}x) = 1. $$
For all~$\kappa\in\N$, let
$$ g_\lambda(\fn) = \ssum{n\in\cN_\lambda \\ (n) = \fn} f(n), \qquad S_\lambda(\kappa) = \sum_{\fp} W\Big(\frac{N(\fp)}{Q^\kappa}\Big)g_\lambda(\fp). $$
Note that by Lemma~\ref{lem:count-units}, we have
\begin{equation}
  \abs{g_\lambda(\fn)} \ll \lambda^{d-1}. \label{eq:trivial-gl}
\end{equation}
Next we smooth out the condition~$n\in\cN_\lambda$ as in Lemma~\ref{lem:Fourier-psi}. Let~$\tau\in\N$ be a parameter, and for all~$X\geq 1$,
$$ g_{\lambda,\tau}(\fn) = \ssum{n\in\Ok \\ (n)=\fn} \chi_\tau\circ\iota^{-1}\Big(\frac n{q^\lambda}\Big)f(n), \qquad S_{\lambda,\tau}(\kappa) = \sum_{\fp} W\Big(\frac{N(\fp)}{Q^\kappa}\Big) g_{\lambda,\tau}(\fp). $$
The function~$g_{\lambda,\tau}$ also satisfies the trivial bound~\eqref{eq:trivial-gl}, so that
$$ S_{\lambda,\tau}(\kappa) \ll \lambda^{d-1} Q^\kappa. $$
Borrowing temporarily the notations~$\phi_\tau$ and~$V_2$ from Lemma~\ref{lem:Fourier-psi}, we have
\begin{align*}
  \abs{S_\lambda(\kappa) - S_{\lambda,\tau}(\kappa)} \leq {}& \sum_{n\in\Ok} (\1_{V_2}\ast \phi_\tau)\Big(\iota^{-1}\Big(\frac{n}{q^\lambda}\Big)\Big) \\
  \ll {}& Q^{\lambda - \eta_2 \tau}
\end{align*}
by Poisson summation, the bound~\eqref{eq:fourier-bound-volV2} and Lemma~\ref{lem:lattice-count}. Finally, let~$\kappa_0\in[\lambda/2, \lambda]\cap\N$. We use the trivial bound for~$\kappa\leq \kappa_0$. Since~$g(\fn)\neq 0$ implies~$1\leq N(\fn)\ll Q^\lambda$, we deduce
$$ \ssum{n\in\cN_\lambda \\ n\text{ prime}} f(n) \ll \lambda Q^{\lambda-\eta_2\tau} + \lambda^d Q^{\kappa_0} + \lambda \sup_{\kappa_0 \leq \kappa \leq \lambda + C} \abs{S_{\lambda,\tau}(\kappa)}, $$
for some~$C$ depending on~$(q,\cD)$ at most. Using Lemma~\ref{lem:comb-ident}, we find
\begin{equation}
  \ssum{n\in\cN_\lambda \\ n\text{ prime}} f(n) \ll \lambda Q^{\lambda-\eta_2\tau} + \lambda^d Q^{\kappa_0} + \lambda \sup_{\kappa_0 \leq \kappa \leq \lambda + C} \big(\abs{S_{\lambda,\tau}^I(\kappa)} + \sup_{\alpha,\beta}\abs{S_{\lambda,\tau}^{II,\alpha,\beta}(\kappa)}\big),\label{eq:after-combidtt}
\end{equation}
where
\begin{align}
  S_{\lambda,\tau}^I(\kappa) = {}& \sum_{N(\fm) \leq Q^{\kappa/4}} \Big| \sum_\fn W\Big(\frac{N(\fm\fn)}{Q^\kappa}\Big)g_{\lambda,\tau}(\fm\fn) \Big| \label{eq:primes-tI} \\
  S_{\lambda,\tau}^{II,\alpha,\beta}(\kappa) = {}& \underset{ Q^{\kappa/4} < N(\fm) \leq Q^{3\kappa/4}}{\sum_\fm \sum_\fn} \alpha_\fm \beta_\fn W\Big(\frac{N(\fm\fn)}{Q^\kappa}\Big) g_{\lambda,\tau}(\fm\fn) \label{eq:primes-tII}
\end{align}
Before we proceed we require the following estimate. Define functions on~$\R^d$, resp. $K$, by
$$ V_0(x) = W(Q^{\lambda-\kappa}N(\iota(x))) \chi_\tau(x), \qquad V = V_0\circ\iota^{-1}. $$
\begin{lemma}\label{lem:majo-sumintVhat}
  We have
  $$ \sum_{\xi\in\Z^d} \abs{{\hat V_0}(\xi)} + \int_{\R^d} \abs{{\hat V_0}(\xi)} \dd\xi \ll Q^{\tau + \lambda-\kappa}, $$
  with an implied constant depending only on~$(q, \cD)$.
\end{lemma}
\begin{proof}
  We introduce a smooth, compactly supported function~$W_0$, majorizing the indicator function of the support of~$\chi_0$, and redundant in the sense that~$\chi_\tau = \chi_\tau W_0$. We then have, for all~$\xi\in\R^d$,
  $$ {\hat V}(\xi) = \int_{\R^d} \chih_\tau(\xi') {\hat W_1}(\xi-\xi') \dd\xi', $$
  with~$W_1(x) = W(N\circ \iota(q^{\lambda-\kappa} x)) W_0(x)$. By~\eqref{eq:def-chitau} and \eqref{eq:bound-der-phi}, we have
  \begin{equation}
    \abs{\chih_\tau(\xi)} \ll_A (1 + \|q^{-\tau} \xi\|)^{-A}\label{eq:majo-pw-chih}
  \end{equation}
  for all~$A\geq 0$. On the other hand, we have
  $$ {\hat W_1}(\xi) = Q^{\kappa-\lambda} \int_{\R^d} W(N\circ \iota(x)) W_0(q^{\kappa-\lambda} x) \e(\tc{x, \qt^{\kappa-\lambda}\xi})\dd x. $$
  By partial differentiation, since~$\|\qt^{\kappa-\lambda}\| \ll 1$, we obtain for all~$A\geq 0$
  \begin{equation}
    \abs{{\hat W_1}(\xi)} \ll_A (1 + \|\qt^{\kappa-\lambda}\xi\|)^{-A}.\label{eq:majo-pw-W1}
  \end{equation}
  The bound for~$\int_{\R^d}\abs{{\hat V_0}(\xi)}\dd\xi$ immediately follows by multiplying the integrals of~\eqref{eq:majo-pw-chih} and \eqref{eq:majo-pw-W1} with respect to~$\xi$. The bound for~$\sum_{\xi\in\Z^d}\abs{{\hat V_0}(\xi)}$ follows by the upper bound
  $$ \sum_{\xi\in\Z^d} (1 + \|\qt^{-\tau}(\xi + \xi_0)\|)^{-d-1} \ll Q^\tau, $$
  valid for all~$\xi_0\in\R^d$: indeed, by translating we may ensure that~$\qt^{-\tau}\xi_0 \in \FR$, and the resulting sum is estimated by Lemma~\ref{lem:lattice-count}.

\end{proof}

\subsubsection{Type~I sums}

For each~$\fm$ in the sum~\eqref{eq:primes-tI}, we have
$$ \sum_\fn W\Big(\frac{N(\fm\fn)}{Q^\kappa}\Big) g_{\lambda,\tau}(\fm\fn) = \ssum{n\in\Ok \\ \fm\mid (n)} V\Big(\frac{n}{q^\lambda}\Big) f(n). $$
Using Lemma~\ref{lem:units} on the~$\fm$-sum, we deduce, for any~$\ee\in\Ok^\ast$,
\begin{equation}
  S_{\lambda,\tau}^I(\kappa) = \ssum{m\in\Ok \\ 0<N(m) \leq Q^{\kappa/4}} \Phi(m\ee) \Big| \ssum{n\in\Ok} V\Big(\frac{mn}{q^\lambda}\Big) f(mn) \Big|.\label{eq:transfer-ideal-el}
\end{equation}
Let~$\mu\geq \floor{\kappa/4}+1$. We pick~$\ee$ so that~$\abs{(q^\mu \ee)^\pi} \asymp Q^{\mu/d}$. This ensures that for any~$m$ in the sum, we have~$\abs{m/q^{\mu}} \ll 1$, so that for some choice~$\mu = \kappa/4 + O(1)$, we have~$m\in\cN_\mu$. We deduce
$$ S_{\lambda,\tau}^I(\kappa) \leq \ssum{m\in\cN_\mu}\Big| \ssum{n\in\Ok} V\Big(\frac{mn}{q^\lambda}\Big)f(mn) \Big|. $$
Note that~$\supp(V) \subset \supp \chi_0$, which depends only on~$(q,\cD)$. Apply Proposition~\ref{prop:type-i} along with Lemma~\ref{lem:majo-sumintVhat} yields
\begin{equation}
  S_{\lambda,\tau}^I(\kappa) \ll \lambda^{d+1} Q^{\lambda - \frac{\eta_1}{1+\eta_1}\gamma(\lambda/3) + \tau + \lambda-\kappa}.\label{eq:prime-bound-I}
\end{equation}

\subsubsection{Type~II sums}

Splitting the interval~$[Q^{\kappa/4}, Q^{3\kappa/4}]$, we have
$$ \sup_{(\alpha, \beta)}\abs{S_{\lambda,\tau}^{II,\alpha,\beta}(\kappa)} \leq \kappa \sup_{\substack{\mu\in\N \\ \frac\kappa4 \leq \mu \leq \frac{3\kappa}4}} \sup_{(\alpha, \beta)}
\Big| \underset{ Q^\mu < N(\fm) \leq Q^{\mu+1}}{\sum_\fm \sum_\fn} \alpha_\fm \beta_\fn W\Big(\frac{N(\fm\fn)}{Q^\kappa}\Big) g_{\lambda,\tau}(\fm\fn) \Big|. $$
Let~$\mu, \alpha, \beta$ satisfy the conditions in the suprema. By arguing as in~\eqref{eq:transfer-ideal-el}, we have
$$ \underset{ Q^\mu < N(\fm) \leq Q^{\mu+1}}{\sum_\fm \sum_\fn} \alpha_\fm \beta_\fn W\Big(\frac{N(\fm\fn)}{Q^\kappa}\Big) g_{\lambda,\tau}(\fm\fn) = \underset{\substack{m, n\in\Ok \\ Q^\mu < N(m) \leq Q^{\mu+1}}}{\sum\sum} \alpha_m \Phi(m\ee) \beta_n V\Big(\frac{mn}{q^\lambda}\Big) f(mn) $$
for all~$\ee\in\Ok^\ast$. Here we abbreviated~$\alpha_m := \alpha_{(m)}$ and~$\beta_n:=\beta_{(n)}$. We pick~$\ee$ so that~$\abs{(m/q^\mu)^\pi} \asymp 1$. Since also~$\|mn/q^\lambda\| \ll 1$ by the support of~$V$, we deduce that for some~$\mu' = \mu + O(1)$ and~$\nu$ with~$\mu' + \nu = \lambda + O(1)$, we have~$m\in\cN_{\mu'}$ and~$n\in\cN_{\nu}$. Writing, for all~$x\in K$,
$$ V(x) = \int_{\R^d} {\hat V_0}(\xi) \e\left(\tc{\xi, \iota^{-1}(x)}\right)\dd\xi, $$
we apply Proposition~\ref{prop:type-ii}, exchanging the roles of~$\mu$ and~$\nu$ if~$\mu'> \nu$, and setting~$\psi(x) = \tc{\xi, \iota^{-1}(x/q^\lambda)}$. By Lemma~\ref{lem:majo-sumintVhat}, and the divisor-bound
$$ \sum_{n\in \cN_\nu} \tau((n))^4 \ll \sum_{N(\fn) \ll Q^\nu} \tau(\fn)^4 \ssum{n\in \fn \\ \|n/q^\nu\| \ll 1} 1 \ll  \nu^{O(1)} Q^\nu, $$
we deduce
$$ \Big| \underset{\substack{m, n\in\Ok \\ Q^\mu < N(m) \leq Q^{\mu+1}}}{\sum\sum} \alpha_m \Phi(m\ee) \beta_n V\Big(\frac{mn}{q^\lambda}\Big) f(mn) \Big| \ll \lambda^{O(1)} Q^{\lambda - \delta \gamma(\floor{\frac{\theta\lambda}{100\Theta}}) + \tau + \lambda-\kappa}, $$
where~$\delta \gg \min\{\eta_1^2\eta_2, \eta_1\theta \}$. We conclude that
\begin{equation}
  \sup_{(\alpha, \beta)}\abs{S_{\lambda,\tau}^{II,\alpha,\beta}(\kappa)} \ll \lambda^{O(1)} Q^{\lambda - \delta \gamma(\floor{\frac{\theta\lambda}{100\Theta}}) + \tau + \lambda-\kappa}\label{eq:prime-bound-II}
\end{equation}

\subsubsection{Conclusion}

The claimed bound follows upon grouping the estimates~\eqref{eq:after-combidtt}, \eqref{eq:prime-bound-I} and \eqref{eq:prime-bound-II}, and optimizing~$\tau$ and~$\kappa_0$ by
$$ \lambda - \kappa_0 = \frac{\delta}{2+\eta_2^{-1}}\gamma\Big(\floor{\frac{\theta\lambda}{100\Theta}}\Big) + O(1), \quad \tau = \eta_2^{-1}(\lambda-\kappa_0) + O(1). $$

\section{Two arithmetic applications : sums of digits and Rudin-Shapiro sequences}\label{sec:two-arithm-appl}

In this section we prove Theorems~\ref{thm:main-thue} and~\ref{thm:main-rudin}. In view of Theorem~\ref{thm:meanvalue-prime}, it will suffice to prove that the functions~$s_{q,\cD}(n)$ and~$r_{q,\cD}(n)$ defined in~\eqref{eq:def-sqK}--\eqref{eq:def-rqK} satisfy the Carry and Fourier properties~\eqref{eq:carry-prop}--\eqref{eq:Fourier-prop}.

\subsection{Sums of digits in~$\O$}

We let~$\sum_{j=0}^d c_j X^j$ be the minimal polynomial of~$q$ (with~$c_d = 1$), and also
$$ \mu_q = \sum_{j=0}^{d} c_j \in \Z, \qquad M_q := \sum_{j=0}^d \abs{c_j}^2. $$

\begin{lemma}\label{lem:props-sumdig}
  Let~$\alpha\in K$. The function given by~$f(n) = \e(\tc{\alpha s_{q,\cD}(n)})$  satisfies the Carry property~\eqref{eq:carry-prop} with~$\eta_1 = \eta_2$, and the Fourier property~\eqref{eq:Fourier-prop} with a function~$\gamma$ satisfying, for some~$\delta_Q>0$ depending on~$Q$ only,
  \begin{equation}
    \gamma(\lambda) \geq C_{q,\cD,\alpha} \lambda + O(1), \qquad C_{q, \cD, \alpha} = \frac{\delta_Q}{M_q(d+1)} \sum_{b\in\cD}\norm{\tc{\mu_q\alpha b}}_{\R/\Z}^2.\label{eq:Fourierprop-sumdig}
  \end{equation}
\end{lemma}

\begin{proof}
  We consider first the Carry property. If~\eqref{eq:carry-prop} holds, then there is a carry propagation in the sum~$m + n$, where~$m = u_1 + vq^\kappa$ and~$n = u_2$. Then, in the notations of Lemma~\ref{lem:carry}, the exists some~$b\in\B_{st}$ such that in the addition~$v + b$, the carry propagates beyond the~$\rho$-th digit. By Lemma~\ref{lem:carry} and finiteness of~$\B_{st}$, there are at most~$O(2^{\lambda-\eta_2 \rho})$ possibilities for~$v$.

  To establish the Fourier property, we argue as in Lemme~20 of~\cite{MauduitRivat2010}, and Lemma~6.3 of~\cite{DrmotaRivatEtAl2008}. We let
  $$ \phi(t) := \abs{\sum_{b\in\cD} \e(\tc{(\alpha + t) b})}. $$
  Using that~$0\in\cD$ and Taylor expansion near the origin, we obtain the existence of~$\varpi_Q>0$, depending only on~$Q$, such that
  $$ \abs{\sum_{b\in\cD} \e(\theta_b)} \leq Q^{1 - \varpi_Q\sum_{b\in\cD} \norm{\theta_b}^2} $$
  for all tuples of real numbers~$(\theta_b)_{b\in\cD}$ with~$\theta_0 = 0$. We deduce, for any fixed~$t\in K$,
  \begin{equation}
    \abs{\phi(t)\phi(tq) \dotsb \phi(tq^d)} \leq Q^{d+1 - \varpi_Q \sum_{b\in\cD} \sum_{j=0}^d \norm{\tc{(\alpha + tq^d)b}}^2}\label{eq:fourprop-sq-1}
  \end{equation}
  On the other hand, by the triangle and the Cauchy--Schwarz inequalities,
  \begin{align*}
    \norm{\tc{\mu_q \alpha b}}_{\R/\Z}^2 \leq {}& \Big(\sum_{j=0}^d \abs{c_j} \norm{\tc{(\alpha + tq^j) b}}_{\R/\Z}\Big)^2 \\
    \leq {}& M_q \sum_{j=0}^d \norm{\tc{(\alpha + tq^j)b}}_{\R/\Z}^2.
  \end{align*}
  Summing this inequality over~$b$ and inserting in~\eqref{eq:fourprop-sq-1} yields
  $$ \sup_{t\in K} \abs{\phi(t)\phi(tq) \dotsb \phi(tq^n)} \leq Q^{(n+1)(1-C_{q,\cD,\alpha})}, $$
  where~$C_{q,\cD,\alpha}$ is given in~\eqref{eq:Fourierprop-sumdig} with~$\delta_Q = Q\varpi_Q$. From here, reasonning as in Lemme~20 of~\cite{MauduitRivat2010} concludes the proof.
\end{proof}

\begin{proof}[Proof of Theorem~\ref{thm:main-thue}]
  Let~$h\in\Z_{\neq 0}$. For some~$\alpha\in K$ and all~$x\in K$, we have~$\phi(x) = \tc{\alpha x}$. If~$\phi(b) \not\in\Q$ for some~$b\in\cD$, then~$\norm{h\tc{\mu_q \alpha b}}_{\R/\Z}>0$. We apply Theorem~\ref{thm:meanvalue-prime} with~$f(n) = \e(h\phi(s_q(n)))$. Using Lemma~\ref{lem:props-sumdig}, we deduce the existence of~$\delta>0$ such that for all~$\lambda\in\N$,
  $$ \sum_{n\in\cN_\lambda} \e(h \phi(s_q(n))) \ll Q^{(1-\delta)\lambda}. $$
  The Weil criterion~\cite[Theorem~I.6.13]{Tenenbaum2015} concludes the proof.
\end{proof}

\subsection{Rudin-Shapiro sequences}

\begin{lemma}\label{lem:props-rs}
  Let~$\alpha \in \R$, and~$(q, \cD)$ be a binary FNS. The function given by~$f(n) = \e(\alpha r_{q,\cD}(n))$ satisfies the Carry property~\eqref{eq:carry-prop} with~$\eta_1 = \eta_2$, and the Fourier property~\eqref{eq:Fourier-prop} with a function~$\gamma$ satisfying
  $$ \gamma(\lambda) \geq \frac\lambda2 \log\Big(\frac2{1+\abs{\cos(\pi \alpha)}}\Big) + O(1), $$
  and any~$\kappa\in\N$.
\end{lemma}
\begin{proof}
  The Carry property follows by an argument identical to the one used in Lemma~\ref{lem:props-sumdig}.
  For the Fourier property, we use Theorem~3.1 in~\cite{AlloucheLiardet1991}. Note that the sum is restricted to integers there, but what is actually considered is a sum over all words of fixed length. The corresponding reduction in pages~12-13 is not needed in our case, since we are summing over the full set~$\cN_\lambda$.
\end{proof}

\begin{proof}[Proof of Theorem~\ref{thm:main-rudin}]
  The deduction of Theorem~\ref{thm:main-rudin} is identical to the argument used in the case~$s_{q,\cD}(n)$.
\end{proof}

\subsection*{Acknowledgements}

This work was partly supported by the ANR (France) and FWF (Austria) through the project ANR-14-CE34-0009 MUDERA.
As the present paper draws from the influential works of Mauduit and Rivat on this topic, we wish to dedicate this paper to the memory of C. Mauduit. We are grateful to J. Rivat, J. Thuswaldner, C. Müllner and the anonymous referee for helpful discussions and remarks on the topics of this work.

\appendix

\section{Asymptotic behaviour of the addition constant}\label{sec:asympt-behav-addit}

The constant~$\eta_2$ from Lemma~\ref{lem:carry} does not seem to admit an explicit expression in terms \textit{e.g.} of the minimal polynomial of~$q$. In this section we consider the special case of canonical number systems (CNS), meaning those number systems $(q, \cD)$ satisfying~$\cD = \{0, 1, \dotsc, Q-1\}$. By~\cite{Kovacs1981}, if~$q_0$ is the basis of a CNS, then for all large enough~$m\in\N$, $-m+q$ is also the basis of a CNS. The goal of this appendix is to show that as~$m\to+\infty$, there are admissible carry constants (from Lemma~\ref{lem:carry}) which are very close to the best possible value.

\begin{proposition}
  Suppose that~$q_0$ is the basis of a CNS, and let~$m\in\N$ be large enough that~$q_m :=-m+q_0$ also is. Then, for the CNS associated with~$q_m$, Lemma~\ref{lem:carry} holds for a value of~$\eta_{2,m}$ satisfying
  $$ \eta_{2, m} \geq \frac 1d - O\Big(\frac1{\log m}\Big), $$
  where the implied constant depends at most on~$K$ and~$q_0$. Consequently, the border~$\partial \cF_m$ of the fundamental tile associated to~$q_m$ has upper-box dimension
  $$ \bar{\dim_B}( \partial \cF_m) \leq d-1 + O\Big(\frac{1}{\log m}\Big). $$
\end{proposition}
Note that we always have~$\bar{\dim_B}( \partial \cF_m) \geq d-1$.
\begin{proof}
  The value~$\eta_{2, m}$, as was apparent from the proof of Lemma~\ref{lem:carry}, is related to the largest eigenvalue of the adjacency matrix of the tranducer describing carry propagation in base~$q_m$ (which was used in the above proof of Lemma~\ref{lem:carry}). We will work with the formalism described in~\cite{ScheicherThuswaldner2002}, where this transducer was described explicitely for CNS. For all~$m\in\N_{\geq 0}$ we let~$\sum_{j=0}^d c_{j,m} X^j$ be the minimal polynomial of~$q_m = -m + q_0$, with~$c_{d,m} = 1$. Note that as~$m\to\infty$, we have~$c_{j, m} \sim m^{d-j}\binom{d}{j}$, so that for~$m$ large enough in terms of~$q_0$ the condition~$c_{j,m} < c_{j-1,m}$ is satisfied for~$1\leq j \leq d$. We define a transducer~${\mathcal T}$ in the following way~:
  \begin{itemize}
    \item The set of states is indexed by subsets~$I \subset \{0, \dotsc, d\}$,
    \item The set of labels is~$\cD = \{0, \dotsc, b_0-1\}$,
    \item Given a state~$I = \{i_0, \dotsc, i_r\}$ (with~$i_0 \leq \dotsb \leq i_r$), we define
    $$ \eta(I) = \sum_{j\geq 0} (-1)^j c_{i_j, m}, $$
    with the convention~$\eta(\varnothing) = 0$.
    \item From a labeled state~$(I_1, d_1)$, there is a transition to another labeled state~$(I_2, d_2)$ determined as follows~:
    \begin{itemize}
      \item If~$d_1 + \eta(I_1) < c_{0,m}$, then~$d_2 = d_1 + \eta(I_1)$ and~$I_2 = I+1 := \{i+1, i\in I\}$.
      \item Otherwise, $d_2 = d_1 + \eta(I_1) - c_{0,m}$ and~$I_2 = ((I+1)\smallsetminus\{0, 1\}) \cup (\{0, 1\} \smallsetminus(I+1))$.
    \end{itemize}
  \end{itemize}
  The states~$I = \{0\}$ and~$I = \varnothing$ are absorbing. Let~$N_m(\ell)$ be the number of possible length~$\ell$ paths in~$\mathcal{T}$ not leading to an absorbing state. Then any value~$\eta_{2, m}>0$ such that~$N_m(\ell) = O(Q^{(1-\eta_2)\ell})$ is admissible as a carry constant.

  We wish to upper-bound the number~$N_m(\ell)$. To this end, we partition the subsets of~$\{0, \dotsc, d\}$ into~$d+1$ classes, according to their smallest or second smallest element~:
  $$ V_{\varnothing}, V_1, V_2 \dotsc, V_d, $$
  where~$V_\varnothing = \{\{0\}, \varnothing\}$, and for~$j\geq1$, $V_j$ consists of the sets whose minimal nonzero element is~$j$. We consider the directed graph~$G$ whose vertices are~$V_\varnothing, V_1, \dotsc, V_d$, and for each pair~$(V, V')$ of vertices, we have an edge~$V\to V'$ with (possibly nil) multiplicity given by the number of transitions~$(I_1, d_1) \to (I_2, d_2)$ in~$\mathcal{T}$ where~$I_1\in V$ and~$I_2\in V'$. Let~$N'_m(\ell)$ be the total number of paths in~$G$ avoiding~$V_\varnothing$ with multiplicity. Then, by construction, we have~$N_m(\ell) \leq N'_m(\ell)$
  
  For~$j\geq 1$, the number of edges in~$G$ from~$V_j$ to~$V_{j+1}$ (with the convention~$V_{d+1} = V_\varnothing$) is given by
  \begin{align*}
    \alpha_{j,m} = {}& \ssum{I \\ \min(I) = j} (c_{0,m} - \eta(I)) + \ssum{I:\ 0\in I \\ \min(I\smallsetminus\{0\}) = j} \eta(I) \\
    = {}& 2^{d-j+1}(c_{0,m}-c_{j,m}) + 2^{d-j} c_{j+1,m},
  \end{align*}
  while the rest of the edges going from~$V_j$ lead to~$V_1$, and the number of them is given by
  $$ \beta_{j,m} = 2^{d-j+1}c_{j,m} - 2^{d-j} c_{j+1,m}. $$
  By Perron-Frobenius' theorem, the number of such path is controlled by the leading eigenvalue~$\lambda_m>0$ of the adjacency matrix (where the absorbing state~$V_\varnothing$ is taken away)
  \[ M_m = \begin{pmatrix} \beta_{1,m} & \beta_{2,m} & \beta_{3,m} & \cdots & \beta_{d-1,m} & \beta_{d,m} \\
    \alpha_{1,m} & 0 & 0 & \cdots & 0 & 0 \\
    0 & \alpha_{2,m} & 0 & \cdots & 0 & 0 \\
    \vdots & \vdots & \ddots & \ddots & \vdots & \vdots \\
    0 & 0 & 0 & \ddots & 0 & 0 \\
    0 & 0 & 0 & \cdots & \alpha_{d-1,m} & 0
  \end{pmatrix}, \]
  in the sense that~$N'_m(\ell) = O((2\lambda_m)^\ell)$, say; we will not require anything more precise. The characteristic polynomial of~$M_m$ is
  $$ P_m(x) = x^d - \sum_{k=1}^d \alpha_{1,m} \dotsb \alpha_{k-1,m} \beta_{k,m} x^{d-k}. $$
  Uniformly for~$x\geq 0$, as~$m\to \infty$, we have
  \begin{align*}
    \sum_{k=1}^d \alpha_{1,m} \dotsb \alpha_{k-1,m} \beta_{k,m} x^{d-k} = {}& (1+o(1)) m^{-d} \sum_{k=1}^d 2^{\frac{k(2d+1-k)}2} (m^d)^k (mx)^{d-k} \binom{d}{k} \\
    \leq {}& (1+o(1)) 2^{\frac{d(d+1)}2} ((x + m^{d-1})^d - x^d).
  \end{align*}
  Therefore~$P_m(x)>0$ if~$x \geq C m^{d-1}$ for a suitable number~$C$ (depending on~$K$ and~$x$), and so~$\lambda_m = O(m^{d-1})$, so that~$N_m(\ell)^{1/\ell} \ll m^{d-1}$. We deduce that there is an admissible constant~$\eta_{2,m}$ satisfying~$Q^{1-\eta_{2,m}} \ll m^{d-1}$. Since~$Q \sim m^d$, we conclude~$\eta_{2,m} \geq \frac1d - O(\frac1{\log m})$ as claimed. The bound on the upper-box dimension follows by~\cite[Theorem~4.7]{ScheicherThuswaldner2002} (with~$\mu = Q^{1-\eta_2}$, $Q\sim m^d$, and~$\beta_{max} \sim m$).
\end{proof}

\bibliographystyle{plain}
\bibliography{bib-RS}

\providecommand{\bysame}{\leavevmode\hbox to3em{\hrulefill}\thinspace}
\providecommand{\MR}{\relax\ifhmode\unskip\space\fi MR }
\providecommand{\MRhref}[2]{%
  \href{http://www.ams.org/mathscinet-getitem?mr=#1}{#2}
}
\providecommand{\href}[2]{#2}
\begin{thebibliography}{10}

\bibitem{AkiyamaEtAl2003}
S.~Akiyama, H.~Brunotte, and A.~Peth\H{o}, \emph{Cubic {CNS} polynomials, notes
  on a conjecture of {W}. {J}. {G}ilbert}, J. Math. Anal. Appl. \textbf{281}
  (2003), no.~1, 402--415.

\bibitem{AkiyamaPethHo2002}
S.~Akiyama and A.~Peth\H{o}, \emph{On canonical number systems}, Theoret.
  Comput. Sci. \textbf{270} (2002), no.~1-2, 921--933.

\bibitem{AlloucheLiardet1991}
J.-P. Allouche and P.~Liardet, \emph{Generalized {R}udin-{S}hapiro sequences},
  Acta Arith. \textbf{60} (1991), no.~1, 1--27. \MR{1129977}

\bibitem{BaratBertheEtAl2006}
G.~Barat, V.~Berthé, P.~Liardet, and J.~M. Thuswaldner, \emph{Dynamical
  directions in numeration}, Ann. Inst. Fourier (Grenoble) \textbf{56} (2006),
  no.~7, 1987--2092.

\bibitem{BarbeHaeseler2005}
A.~Barb\'{e} and F.~von Haeseler, \emph{Correlation and spectral properties of
  higher-dimensional paperfolding and {R}udin-{S}hapiro sequences}, J. Phys. A
  \textbf{38} (2005), no.~12, 2599--2622. \MR{2132076}

\bibitem{BarbeHaeseler2006}
\bysame, \emph{Binary number systems for {${\mathbb Z}^k$}}, J. Number Theory
  \textbf{117} (2006), no.~1, 14--30. \MR{2204733}

\bibitem{Bourgain2015}
J.~Bourgain, \emph{Prescribing the binary digits of primes, {II}}, Israel J.
  Math. \textbf{206} (2015), no.~1, 165--182. \MR{3319636}

\bibitem{BrunotteEtAl2006}
H.~Brunotte, A.~Huszti, and A.~Peth\H{o}, \emph{Bases of canonical number
  systems in quartic algebraic number fields}, J. Th\'{e}or. Nombres Bordeaux
  \textbf{18} (2006), no.~3, 537--557.

\bibitem{CohenDaubechies1993}
A.~Cohen and I.~Daubechies, \emph{Nonseparable bidimensional wavelet bases},
  Rev. Mat. Iberoamericana \textbf{9} (1993), no.~1, 51--137.

\bibitem{Cohen1996}
\bysame, \emph{A new technique to estimate the regularity of refinable
  functions}, Rev. Mat. Iberoamericana \textbf{12} (1996), no.~2, 527--591.

\bibitem{DartygeTenenbaum2005}
C.~Dartyge and G.~Tenenbaum, \emph{Sommes des chiffres de multiples d'entiers},
  Ann. Inst. Fourier (Grenoble) \textbf{55} (2005), no.~7, 2423--2474.
  \MR{2207389}

\bibitem{DT-2018}
S.~Drappeau and B.~Topacogullari, \emph{Combinatorial identities and
  {Titchmarsh}'s problem for multiplicative functions}, Preprint, 2018.

\bibitem{DrmotaEtAl2008}
M.~Drmota, P.~J. Grabner, and P.~Liardet, \emph{Block additive functions on the
  {G}aussian integers}, Acta Arith. \textbf{135} (2008), no.~4, 299--332.
  \MR{2465714}

\bibitem{DrmotaEtAl2009}
M.~Drmota, C.~Mauduit, and J.~Rivat, \emph{Primes with an average sum of
  digits}, Compositio Math. \textbf{145} (2009), no.~2, 271--292. \MR{2501419}

\bibitem{DrmotaEtAl2011}
\bysame, \emph{The sum-of-digits function of polynomial sequences}, J. London
  Math. Soc. \textbf{84} (2011), no.~1, 81--102. \MR{2819691}

\bibitem{DrmotaEtAl2019}
\bysame, \emph{Normality along squares}, J. Eur. Math. Soc. (JEMS) \textbf{21}
  (2019), no.~2, 507--548. \MR{3896209}

\bibitem{DrmotaMorgenbesser2012}
M.~Drmota and J.~F. Morgenbesser, \emph{Generalized {T}hue-{M}orse sequences of
  squares}, Israel J. Math. \textbf{190} (2012), 157--193. \MR{2956237}

\bibitem{DrmotaRivatEtAl2008}
M.~Drmota, J.~Rivat, and T.~Stoll, \emph{The sum of digits of primes in
  {${\mathbb Z}[i]$}}, Monatsh. Math. \textbf{155} (2008), no.~3-4, 317--347.

\bibitem{FouvryMauduit1996a}
\'E. Fouvry and C.~Mauduit, \emph{M\'{e}thodes de crible et fonctions sommes
  des chiffres}, Acta Arith. \textbf{77} (1996), no.~4, 339--351. \MR{1414514}

\bibitem{FouvryMauduit1996}
\'{E}. Fouvry and C.~Mauduit, \emph{Sommes des chiffres et nombres presque
  premiers}, Math. Ann. \textbf{305} (1996), no.~3, 571--599.

\bibitem{Gelfond1967/1968}
A.~O. Gel'fond, \emph{Sur les nombres qui ont des propri\'{e}t\'{e}s additives
  et multiplicatives donn\'{e}es}, Acta Arith. \textbf{13} (1967/1968),
  259--265. \MR{0220693}

\bibitem{GermanKovacs2007}
L.~Germ{\'{a}}n and A.~Kov{\'{a}}cs, \emph{On number system constructions},
  Acta Math. Hungar. \textbf{115} (2007), no.~1-2, 155--167.

\bibitem{Gilbert1986}
W.~J. Gilbert, \emph{The fractal dimension of sets derived from complex bases},
  Canad. Math. Bull. \textbf{29} (1986), no.~4, 495--500.

\bibitem{GittenbergerThuswaldner2000}
B.~Gittenberger and J.~M. Thuswaldner, \emph{Asymptotic normality of
  {$b$}-additive functions on polynomial sequences in the {G}aussian number
  field}, J. Number Theory \textbf{84} (2000), no.~2, 317--341.

\bibitem{Grabner1999}
P.~Grabner and P.~Liardet, \emph{Harmonic properties of the sum-of-digits
  function for complex bases}, Acta Arith. \textbf{91} (1999), no.~4, 329--349.

\bibitem{GrabnerEtAl1998}
P.~J. Grabner, P.~Kirschenhofer, and H.~Prodinger, \emph{The sum-of-digits
  function for complex bases}, J. Lond. Math. Soc. (2) \textbf{57} (1998),
  no.~1, 20--40. \MR{1624777}

\bibitem{GroechenigHaas1994}
K.~Gr\"ochenig and A.~Haas, \emph{Self-similar lattice tilings}, J. Fourier
  Anal. Appl. \textbf{1} (1994), no.~2, 131--170.

\bibitem{Hanna2017}
G.~Hanna, \emph{Sur les occurrences des mots dans les nombres premiers}, Acta
  Arith. \textbf{178} (2017), no.~1, 15--42. \MR{3626236}

\bibitem{HarmanKatai2008}
G.~Harman and I.~K\'{a}tai, \emph{Primes with preassigned digits. {II}}, Acta
  Arith. \textbf{133} (2008), no.~2, 171--184. \MR{2417463}

\bibitem{Hinz1988}
J.~G. Hinz, \emph{A generalization of {B}ombieri's prime number theorem to
  algebraic number fields}, Acta Arith. \textbf{51} (1988), no.~2, 173--193.

\bibitem{Huxley1968}
M.~N. Huxley, \emph{The large sieve inequality for algebraic number fields},
  Mathematika \textbf{15} (1968), 178--187.

\bibitem{IwaniecKowalski2004}
H.~Iwaniec and E.~Kowalski, \emph{Analytic number theory}, vol.~53, Cambridge
  Univ Press, 2004.

\bibitem{KataiKovacs1980}
I.~K\'atai and B.~Kov\'acs, \emph{Kanonische {Z}ahlensysteme in der {T}heorie
  der quadratischen algebraischen {Z}ahlen}, Acta Sci. Math. (Szeged)
  \textbf{42} (1980), no.~1-2, 99--107.

\bibitem{KataiKovacs1981}
\bysame, \emph{Canonical number systems in imaginary quadratic fields}, Acta
  Math. Acad. Sci. Hungar. \textbf{37} (1981), no.~1-3, 159--164.

\bibitem{KataiSzabo1975}
I.~K\'atai and J.~Szab\'o, \emph{Canonical number systems for complex
  integers}, Acta Sci. Math. (Szeged) \textbf{37} (1975), no.~3-4, 255--260.

\bibitem{Keesling1999}
J.~Keesling, \emph{The boundaries of self-similar tiles in {${\bf R}^n$}},
  Topology Appl. \textbf{94} (1999), no.~1-3, 195--205.

\bibitem{Knuth1981}
D.~E. Knuth, \emph{The art of computer programming. {V}ol. 2}, second ed.,
  Addison-Wesley Publishing Co., Reading, Mass., 1981, Seminumerical
  algorithms, Addison-Wesley Series in Computer Science and Information
  Processing. \MR{633878}

\bibitem{Kovacs1981}
B.~Kov\'acs, \emph{Canonical number systems in algebraic number fields}, Acta
  Math. Acad. Sci. Hungar. \textbf{37} (1981), no.~4, 405--407.

\bibitem{KovacsPethHo1991}
B.~Kov\'acs and A.~Peth\H{o}, \emph{Number systems in interal domains,
  especially in orders of algebraic number fields}, Acta Sci. Math. (Szeged)
  \textbf{55} (1991), no.~3-4, 287--299.

\bibitem{LagariasWang1996}
J.~C. Lagarias and Y.~Wang, \emph{Self-affine tiles in {$R^n$}}, Adv. Math.
  \textbf{121} (1996), no.~1, 21--49.

\bibitem{Madritsch2010}
M.~G. Madritsch, \emph{Asymptotic normality of {$b$}-additive functions on
  polynomial sequences in number systems}, Ramanujan J. \textbf{21} (2010),
  no.~2, 181--210.

\bibitem{Mandelbrot1982}
B.~B. Mandelbrot, \emph{The fractal geometry of nature}, W. H. Freeman and Co.,
  San Francisco, Calif., 1982, Schriftenreihe f\"{u}r den Referenten. [Series
  for the Referee]. \MR{665254}

\bibitem{MR-Carres}
C.~Mauduit and J.~Rivat, \emph{La somme des chiffres des carr\'{e}s}, Acta
  Math. \textbf{203} (2009), no.~1, 107--148.

\bibitem{MauduitRivat2010}
\bysame, \emph{Sur un probl\`{e}me de {Gelfond}: la somme des chiffres des
  nombres premiers}, Ann. of Math. (2) \textbf{171} (2010), no.~3, 1591--1646.

\bibitem{MR}
\bysame, \emph{Prime numbers along {Rudin}-{Shapiro} sequences}, J. Eur. Math.
  Soc. \textbf{17} (2015), no.~10, 2595--2642.

\bibitem{MauduitRivat2018}
\bysame, \emph{Rudin-{S}hapiro sequences along squares}, Trans. Amer. Math.
  Soc. \textbf{370} (2018), no.~11, 7899--7921. \MR{3852452}

\bibitem{Montgomery}
H.~L. Montgomery, \emph{The analytic principle of the large sieve}, Bull. Amer.
  Math. Soc. \textbf{84} (1978), no.~4, 547--567.

\bibitem{Morgenbesser2010}
J.~F. Morgenbesser, \emph{The sum of digits of squares in {$Z[i]$}}, J. Number
  Theory \textbf{130} (2010), no.~7, 1433--1469.

\bibitem{Morgenbesser}
\bysame, \emph{The sum of digits of {Gaussian} primes}, Ramanujan J.
  \textbf{27} (2012), no.~1, 43--70.

\bibitem{MTT-Fractal}
W.~M\"uller, J.~M. Thuswaldner, and R.~F. Tichy, \emph{Fractal properties of
  number systems}, Period. Math. Hungar. \textbf{42} (2001), no.~1-2, 51--68.

\bibitem{Muellner2017}
C.~M{\"u}llner, \emph{Automatic sequences fulfill the {S}arnak conjecture},
  Duke Math. J. \textbf{166} (2017), no.~17, 3219--3290. \MR{3724218}

\bibitem{MuellnerSpiegelhofer2017}
C.~M\"ullner and L.~Spiegelhofer, \emph{Normality of the {T}hue-{M}orse
  sequence along {P}iatetski-{S}hapiro sequences, {II}}, Israel J. Math.
  \textbf{220} (2017), no.~2, 691--738.

\bibitem{Murty2007}
R.~M. Murty and J.~{Van Order}, \emph{Counting integral ideals in a number
  field}, Exposition. Math. \textbf{25} (2007), no.~1, 53--66.

\bibitem{Narkiewicz}
W.~Narkiewicz, \emph{Elementary and analytic theory of algebraic numbers},
  third ed., Springer Monographs in Mathematics, Springer-Verlag, Berlin, 2004.

\bibitem{Neukirch1999}
J.~Neukirch, \emph{Algebraic number theory}, Grundlehren der Mathematischen
  Wissenschaften [Fundamental Principles of Mathematical Sciences], vol. 322,
  Springer-Verlag, Berlin, 1999, Translated from the 1992 German original and
  with a note by Norbert Schappacher, With a foreword by G. Harder.
  \MR{1697859}

\bibitem{PethoeThuswaldner2017}
A.~Peth\H{o} and J.~M. Thuswaldner, \emph{Number systems over orders},
  Monatshefte f\"{u}r Mathematik \textbf{187} (2018), no.~4, 681--704.
  \MR{3861324}

\bibitem{Pollicott2008}
M.~Pollicott and H.~Weiss, \emph{How smooth is your wavelet? {W}avelet
  regularity via thermodynamic formalism}, Commun. Math. Phys. \textbf{281}
  (2008), no.~1, 1--21.

\bibitem{Ramare2013}
O.~Ramar\'e, \emph{Prime numbers: emergence and victories of bilinear forms
  decomposition}, Eur. Math. Soc. Newsl. (2013), no.~90, 18--27.

\bibitem{Rudin1959}
W.~Rudin, \emph{Some theorems on {F}ourier coefficients}, Proc. Amer. Math.
  Soc. \textbf{10} (1959), 855--859. \MR{0116184}

\bibitem{Sarnak}
P.~Sarnak, \emph{Mobius randomness and dynamics}, Not. S. Afr. Math. Soc.
  \textbf{43} (2012), no.~2, 89--97.

\bibitem{ScheicherThuswaldner2002}
K.~Scheicher and J.~M. Thuswaldner, \emph{Canonical number systems, counting
  automata and fractals}, Math. Proc. Cambridge Phil. Soc. \textbf{133} (2002),
  no.~1, 163--182.

\bibitem{Shapiro1953}
H.~S. Shapiro, \emph{Extremal problems for polynomials and power series},
  ProQuest LLC, Ann Arbor, MI, 1953, Thesis (Ph.D.)--Massachusetts Institute of
  Technology. \MR{2938495}

\bibitem{Spiegelhofer2018}
L.~Spiegelhofer, \emph{The level of distribution of the {Thue--Morse}
  sequence}, Preprint, 2018.

\bibitem{Steiner2002}
W.~Steiner, \emph{Parry expansions of polynomial sequences}, Integers
  \textbf{2} (2002), Paper A14, 28.

\bibitem{Swaenepoel}
C.~Swaenepoel, \emph{Prime numbers with a positive proportion of preassigned
  digits}, Preprint, 2019.

\bibitem{Tenenbaum2015}
G.~Tenenbaum, \emph{Introduction to analytic and probabilistic number theory},
  third ed., Graduate Studies in Mathematics, vol. 163, American Mathematical
  Society, Providence, RI, 2015, Translated from the 2008 French edition by
  Patrick D. F. Ion. \MR{3363366}

\bibitem{Thuswaldner}
J.~M. Thuswaldner, \emph{Fractal dimension of sets induced by bases of
  imaginary quadratic fields}, Math. Slovaca \textbf{48} (1998), no.~4,
  365--371.

\bibitem{Thuswaldner1998}
\bysame, \emph{The sum of digits function in number fields}, Bull. London Math.
  Soc. \textbf{30} (1998), no.~1, 37--45. \MR{1479034}

\bibitem{Thuswaldner2001}
\bysame, \emph{Fractals and number systems in real quadratic number fields},
  Acta Math. Hungar. \textbf{90} (2001), no.~3, 253--269.

\bibitem{Toth2000}
\'{A}. T{\'o}th, \emph{Roots of quadratic congruences}, Int. Math. Res. Notices
  \textbf{2000} (2000), no.~14, 719--739.

\bibitem{Vaaler1985}
J.~D. Vaaler, \emph{Some extremal functions in {F}ourier analysis}, Bull. Amer.
  Math. Soc. (N.S.) \textbf{12} (1985), no.~2, 183--216. \MR{776471}

\bibitem{Vaughan1980}
R.~C. Vaughan, \emph{An elementary method in prime number theory}, Polska
  Akademia Nauk. Instytut Matematyczny. Acta Arith. \textbf{37} (1980),
  111--115.

\bibitem{Vince1995}
A.~Vince, \emph{Rep-tiling {Euclidean} space}, Aequationes Math. \textbf{50}
  (1995), no.~1-2, 191--213.

\bibitem{Vin37}
I.~M. Vinogradov, \emph{Representation of an odd number as a sum of three
  primes}, Dokl. Akad. Nauk SSSR \textbf{15} (1937), 291--294, English
  Translation in \textit{Selected Works}, pages 129--132, Springer-Verlag,
  Berlin, 1985.

\bibitem{Wirth98}
F.~Wirth, \emph{On the calculation of time-varying stability radii}, Internat.
  J. Robust Nonlinear Control \textbf{8} (1998), no.~12, 1043--1058.

\end{thebibliography}

\end{document}